\documentclass[11pt]{amsart}
\usepackage{amssymb, enumerate, xcolor, bbm, hyperref}

\usepackage{setspace}
\usepackage{algorithm, algorithmic}
\usepackage{fullpage}

\hypersetup{colorlinks=true,linktoc=page}
\hypersetup{%
    colorlinks,
    linkcolor={red!50!black},
    citecolor={green!50!black},
    urlcolor={blue!50!black}
}

\usepackage{tikz}
\usepackage{tkz-euclide}
\usetikzlibrary{calc,decorations.pathmorphing}
\usetikzlibrary{arrows,decorations.pathreplacing}
\pgfdeclarelayer{background}
\pgfdeclarelayer{foreground}
\pgfdeclarelayer{front}
\pgfsetlayers{background,main,foreground,front}

\newcommand{\redge}[8]{

		\ifx\relax#5\relax
		\def\qoffs{0pt}
	\else
		\def\qoffs{#5}
	\fi

				\def\rhedge{
		($#1+#4!\qoffs!-90:#2-#4$) -- 
		($#2+#1!\qoffs!-90:#3-#1$) -- 
		($#3+#2!\qoffs!-90:#4-#2$) -- 
		($#4+#3!\qoffs!-90:#1-#3$) -- cycle}

	\coordinate (12) at ($#1!\qoffs!90:#2$);
	\coordinate (14) at ($#1!\qoffs!-90:#4$);
	\coordinate (23) at ($#2!\qoffs!90:#3$);
	\coordinate (21) at ($#2!\qoffs!-90:#1$);
	\coordinate (34) at ($#3!\qoffs!90:#4$);
	\coordinate (32) at ($#3!\qoffs!-90:#2$);
	\coordinate (41) at ($#4!\qoffs!90:#1$);
	\coordinate (43) at ($#4!\qoffs!-90:#3$);
	
	\def\nrhedge{
		(14) let \p1=($(14)-#1$), \p2=($(12)-#1$) in 
			arc[start angle={atan2(\y1,\x1)}, delta angle={atan2(\y2,\x2)-atan2(\y1,\x1)-360*(atan2(\y2,\x2)-atan2(\y1,\x1)>0)}, x radius=\qoffs, y radius=\qoffs] --
		(21) let \p1=($(21)-#2$), \p2=($(23)-#2$) in 
			arc[start angle={atan2(\y1,\x1)}, delta angle={atan2(\y2,\x2)-atan2(\y1,\x1)-360*(atan2(\y2,\x2)-atan2(\y1,\x1)>0)}, x radius=\qoffs, y radius=\qoffs] --
		(32) let \p1=($(32)-#3$), \p2=($(34)-#3$) in 
			arc[start angle={atan2(\y1,\x1)}, delta angle={atan2(\y2,\x2)-atan2(\y1,\x1)-360*(atan2(\y2,\x2)-atan2(\y1,\x1)>0)}, x radius=\qoffs, y radius=\qoffs] --
		(43) let \p1=($(43)-#4$), \p2=($(41)-#4$) in 
			arc[start angle={atan2(\y1,\x1)}, delta angle={atan2(\y2,\x2)-atan2(\y1,\x1)-360*(atan2(\y2,\x2)-atan2(\y1,\x1)>0)}, x radius=\qoffs, y radius=\qoffs] --
		cycle}

		\ifx\relax#6\relax
		\def\rlwidth{1pt}
	\else
		\def\rlwidth{#6}
	\fi
	
		\ifx\relax#8\relax
		\fill \nrhedge;
	\else
		\fill[#8]\nrhedge;
	\fi

		\ifx\relax#7\relax
		\draw[line width=\rlwidth,rounded corners=\qoffs]\nrhedge;
	\else
		\draw[line width=\rlwidth,#7]\nrhedge;
	\fi
}

\makeatletter
\renewcommand{\fnum@algorithm}{} 
\makeatother

\newtheorem{theorem}{Theorem}[section]
\newtheorem{lemma}[theorem]{Lemma}
\newtheorem{claim}[theorem]{Claim}
\newtheorem{corollary}[theorem]{Corollary}
\newtheorem{proposition}[theorem]{Proposition}

\newtheorem{problem}[theorem]{Problem}

\numberwithin{equation}{section}

\def\a{\alpha}
\def\b{\beta}
\def\D{\Delta}
\def\d{\delta}
\def\ex{\text{ex}}
\def\g{\gamma}
\def\C{\mathcal{C}}
\def\cH{\mathcal{H}}
\def\cA{\mathcal{A}}
\def\tH{\tilde{H}}
\def\eps{\varepsilon}
\def\mP{\mathbb{P}}
\def\mE{\mathbb{E}}
\DeclareMathOperator{\bouquet}{BOUQUET}

\begin{document}


\title{Independence number of hypergraphs under degree conditions}
\thanks{The first two authors are partially supported by NSF grant DMS 1764385.
The third author is partially supported by NSF grant DMS 1700622 and Simons Collaboration Grant 710094.}
\author{Vojt\v{e}ch R\"odl}
\author{Marcelo Sales}
\author{Yi Zhao}

\address{Department of Mathematics, Emory University, 
    Atlanta, GA, USA}
\email{\{vrodl|mtsales\}@emory.edu}

\address{Department of Mathematics and Statistics, Georgia State University, Atlanta, GA 30303}
\email{yzhao6@gsu.edu}

\date{\today}

\begin{abstract}

A well-known result of Ajtai Koml\'{o}s, Pintz, Spencer, and Szemer\'{e}di [J. Combin. Theory Ser. A 32, 321--335, 1982] states that every $k$-graph $H$ on $n$ vertices, with girth at least five, and average degree $t^{k-1}$ contains an independent set of size $c n (\log t)^{1/(k-1)}/t$ for some $c>0$. In this paper we show that an independent set of the same size can be found under weaker conditions allowing certain cycles of length 2, 3 and 4. 

Our work is motivated by a problem of Lo and Zhao, who asked for $k\ge 4$, how large of an independent set a $k$-graph $H$ on $n$ vertices necessarily has when its maximum $(k-2)$-degree $\D_{k-2}(H)\le dn$. (The corresponding problem with respect to $(k-1)$-degrees was solved by Kostochka, Mubayi, and Verstra\"ete [Random Structures \&
Algorithms 44, 224--239, 2014].)
In this paper we show that every $k$-graph $H$ on $n$ vertices with  $\D_{k-2}(H)\le dn$ contains an independent set of size $c (\frac nd \log\log \frac nd)^{1/(k-1)}$, and under additional conditions, an independent set of size $c (\frac nd \log \frac nd)^{1/(k-1)}$. 
The former assertion gives a new upper bound for the $(k-2)$-degree Tur\'an density of complete $k$-graphs.
\end{abstract} 
\maketitle

\section{Introduction}
\subsection{History and main results}

A \emph{$k$-uniform hypergraph} ($k$-graph) consists of a vertex set $V$ and an edge set $E$, which is a family of $k$-element subsets ($k$-subsets) of $V$.  Given a (hyper)graph $H$, an \emph{independent set} is a set of vertices that contains no edge. The \emph{independence number} $\a(H)$ is the size of a largest independent set of $H$. Many extremal problems in Combinatorics can be phrased as problems of finding independence numbers, which has found applications in Ramsey theory, number theory, discrete geometry, and coding theory.

Given a $k$-graph $H$ on $n$ vertices, how large is $\a(H)$ in terms of 
\emph{average (vertex) degree} $d= k |H| /n$? 
Tur\'an's theorem \cite{MR0018405} shows that $\a(H)\ge \frac{n}{d+1}$ for any graph $H$ on $n$ vertices with average degree $d$. In \cite{MR0297614} Spencer observed that every $k$-graph $H$ with $n$ vertices and average degree $d$ satisfies 
\begin{equation}
\label{Spencer}
\a(H)\ge \left(1 - \frac1k\right) \frac{n}{d^{1/(k-1)}}. 
\end{equation}
To see it, let $s= n/d^{1/(k-1)}$. By averaging, there exists an $s$-subset $S\subset V(H)$ spanning at most 
\[
|H| \frac{ \binom sk }{ \binom nk} \le |H| \left(\frac sn \right)^k = \frac{nd}{k} \left(\frac1{d^{1/(k-1)}}\right)^k = \frac{n}{k d^{1/(k-1)}} = \frac sk
\]
edges. Removing one vertex from each of these edges creates an independent subset of $S$ with at least $(1 - 1/k) s $ vertices.

A fundamental question on (hyper)graphs is \emph{under which condition, a $k$-graph on $n$ vertices with average degree $d$ contains an independent set of size $ \omega(d) {n}/{d^{1/(k-1)}} $ for some function $\omega(d)\to \infty$ as $d \to \infty$}?
A seminal result of Ajtai, Koml\'os, and Szemer\'edi \cite{MR611925} states that if $H$ is a triangle-free graph on $n$ vertices with average degree $d$, then $\a(H)\ge c(\log d) n /{d}$ for a constant $c>0$. (All logarithms in this paper have base $e=2.718...$)

Fix $k\ge 2$. A \emph{($k$-uniform) cycle} of length $\ell$ (in short, $\ell$-cycle) is a $k$-graph $(V, E)$ such that $E=\{e_1, \dots e_{\ell} \}$ and there exist distinct vertices $v_1, \dots, v_{\ell}\in V$ such that $v_i\in e_i\cap e_{i+1}$ for $i=1, \dots, \ell$ (where $e_{\ell+1}=e_1$). In particular, a $2$-\emph{cycle} consists of two edges that share at least two vertices. 
Ajtai, Koml\'{o}s, Pintz, Spencer, and Szemer\'{e}di \cite{MR657047} showed that a large independent set exists if the (hyper)graph contains no cycles of length at most four (\emph{i.e.}, girth at least five). Throughout the paper we will use standard $O, o, \Omega, \Theta$ notations assuming $n\to \infty$.
\begin{theorem}[Ajtai et al. \cite{MR657047}]
\label{thm:AKPSS}
Suppose $k\ge 2$ and $k\ll t\ll n$.
Let $H$ be a $k$-graph on $n$ vertices with average degree $t^{k-1}$. If $H$ contains no cycles of length at most four, then 
\begin{align}
\label{eq:aG}
\a(H)=\Omega\left( \frac{n}{t} (\log t)^{1/(k-1)} \right).
\end{align}
\end{theorem}

It was shown in \cite{MR1370956} that  for $k\ge 3$, the bound \eqref{eq:aG} holds even for $H$ containing no 2-cycle (known as \emph{linear $k$-graphs}). On the other hand, random hypergraphs show that \eqref{eq:aG} is sharp. Let $G^k(n, p)$ be a random $k$-graph on $n$ vertices, in which each $k$-set is chosen to be an edge independently with probability $p=(t/n)^{k-1}$. After removing (a small amount of) vertices from cycles of length at most four, we obtain a subhypergraph $H$ of $G^k(n, p)$ with girth at least five and 
with $\a(H)= O( \frac{n}{t} (\log t)^{1/(k-1)}) $.

\medskip
In this paper we study the independence number of $k$-graphs ($k\ge 3$) under various degree conditions.
Given a hypergraph $H$ and a set $S\subseteq V(H)$, the \emph{degree} $\deg_H(S)$ is the number of edges containing $S$. Given $0\le \ell< k$, let $\d_{\ell}(H)$ and $\D_{\ell}(H)$ denote the minimum and maximum of $\deg_H(S)$ over all $\ell$-subsets of $V(H)$, respectively. We often write $\d(H)$ for $\d_1(H)$ and $\Delta(H)$ for  $\Delta_1(H)$.

Theorem~\ref{thm:AKPSS} can be derived from its vertex degree version, which assumes $\D_1(H)\le 2 t^{k-1}$. Indeed, if the average degree of $H$ is at most $t^{k-1}$, then at least a half of its vertices have degree at most $2 t^{k-1}$, and one can find an independent set of the subhypergraph induced on these vertices. 

Kostochka, Mubayi, and Verstra\"ete \cite{MR3158630} considered the independence number of hypergraphs with given maximum $(k-1)$-degree and proved the following theorem, which also follows from the aforementioned result of \cite{MR1370956}.
\begin{theorem}[Kostochka, Mubayi, Verstra\"ete \cite{MR3158630}]
\label{thm:KMV}
Suppose  $3\le k\le n$ and $0< d< n/ (\log n)^{3(k-1)^2}$. Then every $k$-graph on $n$ vertices with $\D_{k-1}(H)\le d$ satisfies $\a(H) = \Omega\big( (\frac{n}d \log \frac nd)^{1/(k-1)}\big)$. 
\end{theorem}

Note that $\D_{k-1}(H)\le d$ implies that $\D_{\ell}(H) \le O(d n^{k-\ell-1})$ for every $\ell \le k-1$; in particular, $\D_1(H)\le d n^{k-2}$. Let $t>0$ satisfy $t^{k-1}= d n^{k-2}$. Then $n/t = (n/d)^{1/(k-1)}$ and consequently, the bounds for $\alpha(H)$ in Theorems~\ref{thm:AKPSS} and~\ref{thm:KMV} are about the same. We further note that in \cite{MR3158630} two constructions were given to show that Theorem~\ref{thm:KMV} is best possible up to the constant factor hidden in $\Omega$: one construction is a blow-up of a partial Steiner system considered in~\cite{MR1248185}; the other one is the random $k$-graph $G^k(n, p)$. 

It is natural to ask if one can strengthen Theorem~\ref{thm:KMV} by replacing  $\D_{k-1}(H)\le d$ with $\D_{\ell}(H) \le d n^{k-\ell-1}$ for some $\ell \le k-2$. Note that we cannot let $\ell = 1$ (without any additional condition) because the union of disjoint cliques gives a counterexample. The following problem was raised (implicitly) by Lo and Zhao when they studied the $\ell$-degree Tur\'an density of complete $k$-graphs \cite{MR3807943}.
\begin{problem}[Lo and Zhao \cite{MR3807943}]
\label{pro:LZ}
Given $2\le \ell < k$ and $0< d\ll n$, let $H$ be a $k$-graph on $n$ vertices such that $\D_{\ell}(H) \le d n^{k-\ell-1}$. Is it true that $\alpha(H)/ (\frac{n}d )^{1/(k-1)} \to \infty$ and even
$\alpha(H) = \Omega( (\frac{n}d \log\frac{n}d)^{1/(k-1)})$? 
\end{problem}

In \cite{2021arXiv210104258B,MR3846882} a $k$-graph $H$ on $n$ vertices with $\D_{\ell}(H)\le \lambda$ is called an \emph{$(n, k, \ell, \lambda)$-system}.
Tian and Liu \cite{MR3846882} and later Bohman, Liu, and Mubayi  \cite{2021arXiv210104258B} 
studied the independence number of $(n, k, \ell, \lambda)$-system and obtained a tight bound when $\lambda =o(n)$. 
Unfortunately these results offer no help on Problem~\ref{pro:LZ} because $d n^{k-\ell-1}> n$ when $\ell \le k-2$.

Problem~\ref{pro:LZ} seems much harder when $2\le \ell \le k-2$ than the case of $\ell=k-1$.
Indeed, when $\D_{k-1}(H)\le d$, we can use the {deletion method} to delete all 2-cycles and apply the result of \cite{MR1370956} to obtain a large independent set.
In contrast, knowing $\D_{\ell}(H)$ for some $\ell < k-1$ gives no control on the number of 2-cycles with two edges sharing $k-1$ vertices.  Nevertheless, we are able to show that $\a(H)/ (n/d)^{1/(k-1)}\to \infty$ for all $k$-graphs $H$ on $n$ vertices with $\D_{k-2}(H)\le dn$. 

\begin{theorem}
\label{thm:loglog}
Let $k\ge 4$ and $0< d< n$. 
If $H$ is a $k$-graph on $n$ vertices with $\D_{k-2}(H)\le dn$, then $\a(H) = \Omega \left( (\frac nd \log\log \frac nd )^{\frac1{k-1}} \right)$.
\end{theorem}

The main tool in the proof of Theorem~\ref{thm:loglog} is a recent result of Bohman, Liu, and Mubayi  \cite{2021arXiv210104258B}.

\medskip
As a supportive evidence of Problem~\ref{pro:LZ}, we obtain the desired independence number  $\alpha(H) = \Omega( (\frac{n}d \log\frac{n}d)^{1/(k-1)})$
after forbidding a small range of values for $(k-1)$-degrees in Theorem~\ref{th:application}. If we in addition assume that $H$ does not contain a certain family of 4-cycles, then we can reduce the forbidden range to a logarithmic size.\footnote{In fact, we can reduce the upper bound $n^{\frac{k-2}{k-1}}d^{\frac{1}{k-1}}$ to $n^{\frac{k-2}{k-1}}d^{\frac{1}{k-1}}/(\log (n/d))^{\frac{1}{k-1}}$ both in Conditions (1) and (2) and the proof is similar. We choose not to present this stronger version because the gap between the lower and (new) upper bounds is still in the same order of magnitude while the proof becomes more technical.} We say that a $4$-cycle with edges $e_1,e_2,e_3,e_4$ has \emph{clean intersections} if $e_1\cap e_3=\emptyset$ and $e_2\cap e_4=\emptyset$. See Figure~\ref{fig:cleanintersection} for 4-uniform 4-cycles with clean intersections.

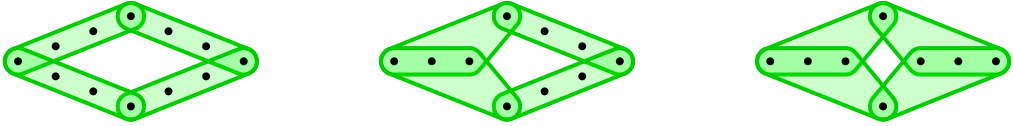
\begin{figure}[ht]
	\begin{tikzpicture}[scale=1.0]
		
		\coordinate (x1) at (-6.5,0);
		\coordinate (x2) at (-6,0.2);
		\coordinate (x3) at (-5.5,0.4);
		\coordinate (x4) at (-5,0.6);
		\coordinate (x5) at (-4.5,0.4);
		\coordinate (x6) at (-4,0.2);
		\coordinate (x7) at (-3.5,0);
		\coordinate (x8) at (-4,-0.2);
		\coordinate (x9) at (-4.5,-0.4);
		\coordinate (x10) at (-5,-0.6);
		\coordinate (x11) at (-5.5,-0.4);
		\coordinate (x12) at (-6,-0.2);
		\coordinate (y1) at (-1.5,0);
		\coordinate (y2) at (-1,0);
		\coordinate (y3) at (-0.5,0);
		\coordinate (y4) at (0,0.6);
		\coordinate (y5) at (0.5,0.4);
		\coordinate (y6) at (1,0.2);
		\coordinate (y7) at (1.5,0);
		\coordinate (y8) at (1,-0.2);
		\coordinate (y9) at (0.5,-0.4);
		\coordinate (y10) at (0,-0.6);
		\coordinate (z1) at (3.5,0);
		\coordinate (z2) at (4,0);
		\coordinate (z3) at (4.5,0);
		\coordinate (z4) at (5,0.6);
		\coordinate (z5) at (5.5,0);
		\coordinate (z6) at (6,0);
		\coordinate (z7) at (6.5,0);
		\coordinate (z8) at (5,-0.6);

		\redge{(x4)}{(x3)}{(x2)}{(x1)}{5pt}{1.5pt}{green!80!black}{green,opacity=0.2};
		\redge{(x7)}{(x6)}{(x5)}{(x4)}{5pt}{1.5pt}{green!80!black}{green,opacity=0.2};
		\redge{(x10)}{(x9)}{(x8)}{(x7)}{5pt}{1.5pt}{green!80!black}{green,opacity=0.2};
		\redge{(x1)}{(x12)}{(x11)}{(x10)}{5pt}{1.5pt}{green!80!black}{green,opacity=0.2};
		
		\redge{(y4)}{(y3)}{(y2)}{(y1)}{5pt}{1.5pt}{green!80!black}{green,opacity=0.2};
		\redge{(y7)}{(y6)}{(y5)}{(y4)}{5pt}{1.5pt}{green!80!black}{green,opacity=0.2};
		\redge{(y10)}{(y9)}{(y8)}{(y7)}{5pt}{1.5pt}{green!80!black}{green,opacity=0.2};
		\redge{(y1)}{(y2)}{(y3)}{(y10)}{5pt}{1.5pt}{green!80!black}{green,opacity=0.2};
		
		\redge{(z4)}{(z3)}{(z2)}{(z1)}{5pt}{1.5pt}{green!80!black}{green,opacity=0.2};
		\redge{(z7)}{(z6)}{(z5)}{(z4)}{5pt}{1.5pt}{green!80!black}{green,opacity=0.2};
		\redge{(z8)}{(z5)}{(z6)}{(z7)}{5pt}{1.5pt}{green!80!black}{green,opacity=0.2};
		\redge{(z1)}{(z2)}{(z3)}{(z8)}{5pt}{1.5pt}{green!80!black}{green,opacity=0.2};
		
		\fill (x1) circle (1.5pt);
		\fill (x2) circle (1.5pt);
		\fill (x3) circle (1.5pt);
		\fill (x4) circle (1.5pt);
		\fill (x5) circle (1.5pt);
		\fill (x6) circle (1.5pt);
		\fill (x7) circle (1.5pt);
		\fill (x8) circle (1.5pt);
		\fill (x9) circle (1.5pt);
		\fill (x10) circle (1.5pt);
		\fill (x11) circle (1.5pt);
		\fill (x12) circle (1.5pt);
		
		\fill (y1) circle (1.5pt);
		\fill (y2) circle (1.5pt);
		\fill (y3) circle (1.5pt);
		\fill (y4) circle (1.5pt);
		\fill (y5) circle (1.5pt);
		\fill (y6) circle (1.5pt);
		\fill (y7) circle (1.5pt);
		\fill (y8) circle (1.5pt);
		\fill (y9) circle (1.5pt);
		\fill (y10) circle (1.5pt);
		
		\fill (z1) circle (1.5pt);
		\fill (z2) circle (1.5pt);
		\fill (z3) circle (1.5pt);
		\fill (z4) circle (1.5pt);
		\fill (z5) circle (1.5pt);
		\fill (z6) circle (1.5pt);
		\fill (z7) circle (1.5pt);
		\fill (z8) circle (1.5pt);
		
	\end{tikzpicture}
	\caption{$4$-cycles with clean intersections}
	\label{fig:cleanintersection}
\end{figure}

\begin{theorem}\label{th:application}
Let $k\geq 4$ and  $0< d\leq n$. Suppose $H$ is a $k$-graph on $n$ vertices with $\D_{k-2}(H)\leq dn$ satisfying \textbf{either} of the following conditions:
\begin{enumerate}
\item there exists $\eps>0$ such that $\deg(S) \not\in \left( n^{ \frac{k-2}{k-1}-\eps} d^{ \frac{1}{k-1}+\eps}, n^{ \frac{k-2}{k-1} } d^{\frac{1}{k-1} }  \right)$ for every $S \in \binom{V(H)}{k-1}$;
\item $\deg(S)\not\in \left(n^{\frac{k-2}{k-1}}d^{\frac{1}{k-1}}/\log(\frac{n}{d})^{k+1}, n^{\frac{k-2}{k-1}}d^{\frac{1}{k-1}} \right)$ for every $S \in \binom{V(H)}{k-1}$, and $H$ has no 4-cycle with clean intersections.
\end{enumerate}
Then $\a(H)=\Omega\left((\frac{n}{d}\log \frac nd)^{1/k-1}\right)$, where the constant factor may depend on $\eps$ in Case~(1).
\end{theorem}

Theorem~\ref{th:application} follows from our main technical result, Theorem~\ref{thm:main}, presented near the end of this section. 

\subsection{Other results and technical theorem}

Problem~\ref{pro:LZ} has an immediate application on hypergraph Tur\'an densities.
For a $k$-graph $F$, let $\ex(n, F)$ be the largest number of edges in a $k$-graph on $n$ vertices containing no copy of $F$ (namely, $F$-free). 
Determining or estimating $\ex(n, F)$ is a classical problem in extremal combinatorics.
In \cite{MR3248027, MZ07}, the following variant of the Tur\'an problem was introduced.
 For $0\le \ell\le k-1$, we denote by $\ex_{\ell}(n, F)$ the maximum minimum $\ell$-degree $\d_{\ell}(H)$ among all $F$-free $k$-graphs $H$ on $n$ vertices and let $\pi_{\ell}(F)= \lim_{n\to \infty} \ex_{\ell}(n, F)/\binom{n}{k- \ell}$ be the corresponding Tur\'an density.
It is well-known that 
\begin{align}
\label{eq:pi}
\pi_0(F) = \pi_1(F) \ge \pi_2(F) \cdots \ge \pi_{k-1}(F)
\end{align}
holds for all $k$-graphs $F$. Let $K_t^k$ denote the complete $k$-graph on $t$ vertices.
Improving the bounds in \cite{MR3158267,MR3248027}, Lo and Zhao \cite{MR3807943} obtained asymptotically tight bounds for $\pi_{k-1}(K_t^k)$ by showing that $\pi_{k-1}(K^k_t) = 1 - \Theta( \log t/ t^{k-1})$ as $t\to \infty$. If the second question in  Problem~\ref{pro:LZ} has a positive answer, then $\pi_{\ell}(K^k_t) = 1 - \Theta( \log t/ t^{k-1})$ for $2\le \ell\le k-2$.

As a corollary of Theorem~\ref{thm:loglog}, we improve the upper bound for the $(k-2)$-degree Tur\'an density of $K_t^k$ (for large $t$). Previously it was only known that $\pi_{k-2}( K_t^k)\le \pi_{0}( K_t^k) \le 1 - 1/\binom{t-1}{k-1}$ by \eqref{eq:pi} and a result of de Caen \cite{MR734038}.
\begin{corollary}
\label{cor:k-2}
Given $k\ge 4$, there exists $c_0>0$ such that $\pi_{k-2}( K_t^k) \le 1 - c_0 \log\log t/ t^{k-1}$.
\end{corollary}

The following theorem is a corollary of our Theorem~\ref{thm:main} and its proof is very similar to the one of Theorem~\ref{th:application}.
\begin{theorem}
\label{cor2}
Given $k\ge 4$, $0<\eps\le 1$, suppose $(\log n)^{12 k/\eps}\leq t \leq n^{1/(1+\eps)}$ and $n$ is sufficiently large. 
Let $H$ be a $k$-graph on $n$ vertices with
\begin{align}
\label{eq:Cor2}
\D_1(H)\le t^{k-1}, \quad \D_i(H)\le t^{k-i-\eps} \text{  for  } 2\le i\le k-2, \quad \text{and} \quad \D_{k-1}(H)\le t/ (\log t)^{k+1}.
\end{align}
If $H$ contains no $4$-cycles with clean intersections, then $\a(H)= \Omega\left(\frac nt (\log t)^{1/(k-1)} \right) $ where the constant factor may depend on $k$ and $\eps$.
\end{theorem}

Bennett and Bohman \cite{MR3545824} showed that with high probability, the random greedy algorithm on $H$ produces a large independent set as long as $H$ satisfies certain degree and pair degree conditions.
Cooper and Mubayi \cite{MR3507547} showed that such large independent set exists under degree conditions alone.

\begin{theorem}[Cooper and Mubayi \cite{MR3507547}]
\label{thm:CM}
Fix $k\ge 3$ and let $H$ be a $k$-graph with $\D_1(H)\le t^{k-1}$ and $\D_{i}(H)\le t^{k - i} /f$ for $2\le i<k$. Then $\chi(H) = O\left((\frac{\D}{\log f})^{1/(k-1)}\right)$, in particular, $\a(H) = \Omega\left( \frac{n}{t} (\log f)^{1/(k-1)} \right)$.
\end{theorem}

Unlike Theorem~\ref{thm:CM}, our Theorem~\ref{cor2} has no implication on the chromatic number $\chi(H)$ and assumes that $H$ does not contain certain family of $4$-cycles. On the other hand, applying Theorem~\ref{thm:CM} under the assumption \eqref{eq:Cor2}, one would obtain that $\a(H)=\Omega\left(\frac nt (\log \log t)^{1/(k-1)} \right)$, which is weaker than the conclusion of Theorem~\ref{cor2}.

\medskip
At last, we present our main technical result.
Let $H=\{H_2,\ldots, H_k\}$ be a hypergraph, where each $H_i$ is an $i$-uniform hypergraph on $V(H)$. We say that $H$ has the \textbf{BOUQUET} property if $H$ satisfies the following structural conditions. Note that a \emph{linear $3$-cycle} (in a hypergraph) consists of three edges $e_1, e_2, e_3$ such that $e_1\cap e_2=\{v_1\}$, $e_2\cap e_3 =\{v_2\}$, and $e_3\cap e_1= \{v_3\}$ for distinct vertices $v_1, v_2, v_3$.
Recall that a \emph{$4$-cycle with clean intersections} is a cycle on four edges $e_1,e_2,e_3,e_4$ such that $e_1\cap e_3=e_2\cap e_4=\emptyset$.

\begin{enumerate}[i)]
    \item If $e \in H_i$ and $f\in H_j$ for $i\neq j$, then $|e\cap f|\leq 1$;
    \item For every $2\leq i \leq k$, the intersection of two edges of $H_i$ has sizes $0,1$ or $i-1$;
    \item $H$ contains no linear $3$-cycle unless at least two edges of the cycle are in $H_2$;
    \item $H$ contains no $4$-cycles with clean intersections;
    \item There are no three edges $e_1$, $e_2$, $e_3\in H_3$ such that $|e_1 \cap e_2|=|e_2\cap e_3|=2$ and $|e_1\cap e_3|=1$.
\end{enumerate}

If a hypergraph contains no cycles of length $2,3,4$, then it trivially has the $\bouquet$ property. Therefore the following theorem is a generalization of Theorem~\ref{thm:AKPSS}, provided that $N$ and $T$ satisfy the given assumption.

\begin{theorem}\label{thm:main}
Suppose $N, T, k$ are positive integers such that $N$ is sufficiently large and
\[
(\log N)^3 \le T \le N^{1/4k}.
\]
Let $H=\{H_2,\ldots,H_k\}$ be a hypergraph on $N$ vertices with the $\bouquet$ property. If $H$ satisfies
\begin{align*}
   & \Delta_1(H_i)\leq T^{i-1}(\log T)^{\frac{k-i}{k-1}} \quad \text{for} \quad 2\leq i \leq k \quad \text{and} \\
   & \Delta_{i-1}(H_i) \leq \frac{T}{(\log T)^{i+1}} \quad \text{for} \quad 3\leq i \leq k,
\end{align*}
then $\alpha(H)=\Omega\left(\frac{N}{T}(\log T)^{1/(k-1)}\right)$ where the constant factor may depend on $k$.
\end{theorem}

Lee and Lefmann \cite{MR4100800} recently obtained a different non-uniform generalization of Theorem~\ref{thm:AKPSS}, which allows certain $3$- and $4$-cycles but forbid all 2-cycles. Different from \cite{MR4100800}, our proof of Theorem~\ref{thm:main} follows the semi-random method of Ajtai et al. \cite{MR657047} so it is self-contained and does not rely on any aforementioned result, such as Theorem~\ref{thm:AKPSS} or \ref{thm:KMV}. 

We note two very recent results on related problems \cite{GJKKL22, LP22}. 
In particular, \cite[Theorem 1.5]{LP22} implies that $\alpha(H)=\Omega\left(\frac{N}{T}(\log T)^{1/(k-1)}\right)$ whenever $H$ satisfies similar degree conditions as in Theorem~\ref{thm:main} and does not contain certain $3$-cycles.

\bigskip
The structure of this paper is as follows. We prove Theorem~\ref{thm:loglog} and Corollary~\ref{cor:k-2} in the next section. 
We present the proofs of Theorems~\ref{th:application} and \ref{cor2} in the appendix because they are very similar to that of Theorem~\ref{thm:loglog}, except that we apply Theorem~\ref{thm:main} instead of a theorem of \cite{2021arXiv210104258B}.
In Section~\ref{sec:pre}, we prove preliminary results and explain the semi-random method along with the properties of hypergraphs satisfying the $\bouquet$ property. In Section~\ref{sec:key} we introduce a key lemma and prove Theorem~\ref{thm:main} by using this lemma. We present the proof of the key lemma in the last section.

\medskip
\textbf{Notation.} A hypergraph $H$ consists of a vertex set and an edge set, which is a family of subsets of the vertex set. We often identify the hypergraph with its edge set and thus write $f\in H$ when $f$ is an edge of $H$ and denote the size of $H$ by $|H|$. 
For a vertex $x$ in a hypergraph $H$, define \emph{the link} $H_x$ to be the family of sets $e\setminus \{x\}$ for all edges $e\in H$ that contain $x$. 
We often write $\deg_H(x)$, instead of $\deg_H(\{x\})$, which equals to $| H_x |$.
Define $N(x) = \{x\} \cup \bigcup_{f\in H_x} f$ as the \emph{closed neighborhood} of $x$, which consists of all vertices that lie on a common edge with $x$ (including $x$ itself). 
Given a set $S$ and $k\in \mathbb{N}$, $\binom Sk$ is the family of all $k$-subsets of $S$. 
Given $x, y, z\in \mathbb{R}$, we write $x= y\pm z$ when $y-z \le x \le y+z$.

\section{Proofs of Theorem~\ref{thm:loglog} and Corollary~\ref{cor:k-2}}

We first show that Theorem~\ref{thm:loglog} implies Corollary~\ref{cor:k-2}. 

\begin{proof}[Proof of Corollary~\ref{cor:k-2}]
Theorem~\ref{thm:loglog} provides a constant $c_1= c_1(k)$ such that every $k$-graph $H$ on $n$ vertices with $\D_{k-2}(H)\le dn$ satisfies $\a(H)\ge c_1  (\frac nd \log\log \frac nd )^{\frac1{k-1}} $. Let $c_0 = c_1^{k-1}$ and assume that $t$ is sufficiently large. In order to have $\pi_{k-2}( K_t^k) \le 1 - c_0 \log\log t/ t^{k-1}$, it suffices to show that every $k$-graph $G$ on $n$ vertices with $\d_{k-2}(G)\ge (1- c_0 \log\log t/ t^{k-1})\binom n2$ contains a copy of $K_t^k$. Let $H$ be the complement $k$-graph of $G$, namely $V(H)= V(G)$ and $E(H)= \binom {V(H)}k\setminus E(G)$. 
We have 
\[\D_{k-2}(H)\le c_0 \frac{\log\log t}{t^{k-1}} \binom n2\le dn, \quad \text{where} \quad d= \frac{c_0 \log\log t}{2t^{k-1}} n.
\]
Then $n/d = c' t^{k-1}/ \log\log t$, where $c' = 2/ c_0$.
By Theorem~\ref{thm:loglog}, we have 
\[
\a(H)\ge c_1 \left( \frac{c' t^{k-1}}{\log\log t} \log\log\left( \frac{c' t^{k-1}}{\log\log t} \right) \right)^{\frac1{k-1}} \ge  c_1 \left( \frac{c' t^{k-1}}{\log\log t} \frac{\log\log t}2 \right)^{\frac1{k-1}} 
= c_1 \left(\frac{c'}2\right)^{\frac1{k-1}} t
\]
because $t$ is sufficiently large. Since $c_1 (c'/2)^{1/(k-1)} = c_1/ c_0^{1/(k-1)} = 1$, we obtain $\a(H)\ge t$, which gives a desired copy of $K_t^k$ in $G$.
\end{proof}

\medskip
Theorem~\ref{thm:loglog} follows from the following lemma and deletion method.
Given a $k$-graph $G$ and $2\le \ell<k$, a $(2, \ell)$-cycle is a $2$-cycle formed by two edges that share exactly $\ell$ vertices.
We let $C_G(2, \ell)$ denote the number of $(2, \ell)$-cycles of $G$.

\begin{lemma}
\label{lem:m}
Suppose $k\ge 4$, $c>0$, and $m^{\frac{1}{2k-2}}/\log m\geq k$. If $H$ is a $k$-graph on $m$ vertices with $\D(H)\le c m^{1/2}$ and without $(2, \ell)$-cycles for all $\ell =2, \dots, k-2$, then $\a(H) = \Omega \left( m^{\frac{2k-3}{2k-2}} (\log \log m)^{\frac{1}{k-1}} \right)$.
\end{lemma}

\begin{proof}[Proof of Theorem~\ref{thm:loglog}]
For $0 \leq \ell \leq k-2$, the assumption $\D_{k-2}(H)\le dn$ implies that 
\begin{align}\label{eq:degH}
    \D_{\ell}(H) \le \frac{ \binom{n-\ell}{k-2-\ell} \, dn }{\binom{k-\ell}{k-2-\ell} } = \frac{ (n-\ell) \cdots (n-k+3) }{ (k-\ell) \cdots 3 } dn    \le \frac{2}{(k-\ell)!} n^{k-\ell-1} d.
\end{align}
In particular, $|H| \le \frac2{k!} n^{k-1} d \le \frac1{12} n^{k-1} d$ because $k\ge 4$.

If $n/d$ is bounded by a constant, then $\a(H) = \Omega \left( (\frac nd \log\log \frac nd )^{\frac1{k-1}} \right)$ follows from \eqref{Spencer}. Therefore, we assume that $n/d$ is sufficiently large in the rest of the proof, .

Let $U\subseteq V(H)$ be a random subset whose vertices are chosen independently, each one with probability $p = n^{-\frac{2k-5}{2k-3}} d^{-\frac2{2k-3}}$. 
Then $\mE(|U|)= pn = (n/d)^{2/(2k-3)}$.
Since $|U|$ follows the binomial distribution $B(n, p)$, by Chernoff's bound, we have
\begin{align}\label{eq:|U|}
\mP\left( \left| |U| - pn \right| \ge \frac{pn}{9}  \right) < e^{-c pn} < \frac1{8}
\end{align}
for some constant $c>0$ and sufficiently large $n$. We also have
\[
\mE( |H[U]| ) \le p^k  \cdot \frac1{12} \, n^{k-1} d = \frac1{12} \left(\frac nd \right)^{\frac{3}{2k-3}}.
\]
By Markov's inequality, the probability of $|H[U]| \ge 4 \mE( |H[U]| )$ is at most $1/4$.
For $2\le \ell \le k-2$, let $X_{\ell} = C_{H[U]}(2, \ell) $.
we derive from \eqref{eq:degH} that
\[
\mE\left( X_{\ell} \right) \le p^{2k- \ell} \binom n{\ell} \binom{\D_{\ell}(H)}{2} \le p^{2k- \ell} \cdot \frac{n^{\ell}}{\ell!} \cdot \frac12 (n^{k-\ell-1}d)^2
\le \frac14 \, p^{2k - \ell} n^{2k -\ell -2} d^2.
\]
Since $p = n^{-\frac{2k-5}{2k-3}} d^{-\frac2{2k-3}}$, it follows that $\mE(X_2) \le pn/4$ and $\mE(X_{\ell}) \le o(pn)$ for $3\le \ell \le k-2$. Let $X = \sum_{\ell=2}^{k-2} X_{\ell}$ denote the number of $(2, \ell)$-cycles for any $2\le \ell\le k-2$. 
We derive that $\mE(X)\le pn/3$. By Markov's inequality, it follows that $\mP( X \ge 2\mE(X) ) \le 1/2$. 

Therefore, with probability at least $1 - (\frac18 + \frac14 + \frac12) = \frac18$, we have 
\[
|U| > \frac89\, pn = \frac89 \left(\frac nd \right)^{\frac{2}{2k-3}}, \quad |H[U]| < \frac4{12} \left(\frac nd \right)^{\frac{3}{2k-3}}, \quad \text{and} 
\quad X< \frac23 pn = \frac23 \left(\frac nd \right)^{\frac{2}{2k-3}}.
\]
We now pick a set $U$ with these properties. For convenience, write $M= (n/d)^{2/(2k-3)}$. Thus $|U|\ge \frac89 M$ and $X\le \frac23 M$.
Let $U^*$ be the set of all vertices $v\in U$ with $\deg_{H[U]}(v) \ge 3k M^{1/2}$. 
Since $ |H[U]| \le \frac13 (n/d)^{3/(2k-3)} = \frac13 M^{3/2}$ and $\sum_{v\in U} \deg_{H[U]}(v) = k |H[U]|$, we have
\[
|U^*| \le \frac{\frac13 k M^{3/2} }{ 3k M^{1/2} } = \frac{M}9.
\]
We now remove from $U\setminus U^*$ one vertex from each of the $(2, \ell)$-cycles with $2\le \ell\le k-2$ obtaining a set of size at least $|U|-|U^*|-X\geq M/9$. After removing more vertices if necessary,
we obtain a set $U'\subset U$ with $|U'| = M/9$. Let $m= M/9$. The induced subhypergraph $H[U']$ contains no $(2, \ell)$-cycle for any $2\le \ell\le k-2$ and satisfies 
$\D(H[U]) < 3k M^{1/2} = 9k m^{1/2}$.
We now apply Lemma~\ref{lem:m} to $H[U']$. Since  $m= \frac19 (\frac{n}{d})^{2/(2k-3)}$, it gives that 
\[
\a(H)\ge \a(H[U']) = \Omega \left( m^{\frac{2k-3}{2k-2}} (\log \log m)^{\frac{1}{k-1}} \right) = \Omega \left( \left(\frac nd \log\log \frac nd \right)^{\frac1{k-1}} \right). \qedhere
\]
\end{proof}

\medskip
It remains to prove Lemma~\ref{lem:m}. We will apply the following theorem of Bohman, Liu, and Mubayi \cite[Theorem 2.6]{2021arXiv210104258B}, which was derived from the random greedy algorithm of Bennett and Bohman \cite{MR3545824} .
Given a $k$-graph $H$, let $\Gamma(H)$ denote the maximum $g$ such that there are two vertices $x, y\in V(H)$ sharing $g$ \emph{common neighbors}, namely, $(k-1)$-sets $S\subset V(H)$ such that $\{x\}\cup S\in E(H)$ and $\{y\}\cup S\in E(H)$.

\begin{theorem}[Bohman, Liu, and Mubayi \cite{2021arXiv210104258B}]
\label{thm:BLM}
Let $k_1, k_2\ge 2$, $m\in \mathbb{N}$, $D, d, \eps > 0$ such that
\begin{itemize}
\item[(i)] $D > m^{\eps}$ and $\omega:= d(\frac{\log m}{D})^{\frac{k_2 - 1}{k_1 - 1} } \gg 1$.
\end{itemize}
Let $V$ be a set of size $m$, $G_1=(V, E_1)$ be a $k_1$-graph, and $G_2= (V, E_2)$ be a $k_2$-graph with the following properties:
\begin{itemize}
\item[(ii)] $\D(G_1)\le D$, $\D_i(G_1)< D^{\frac{k_1 - i}{k_1 - 1} - \eps}$ for $2\le i\le k_1 - 1$ and $\Gamma(G_1)< D^{1-\eps}$;
\item[(iii)] $k |G_2| \le dm$ and $C_{G_2}(2, i)\ll m(\frac{D}{\log m})^{\frac{2k_2 - i - 1}{k_1 -1}}$ for $2\le i\le k_2 - 1$. 
\end{itemize}
Then $\a(G_1\cup G_2) = \Omega \left( n \left( \frac{\log \omega}{d}  \right)^{\frac1{k_2 - 1}} \right)$, where $G_1\cup G_2 = (V, E_1\cup E_2)$.
\end{theorem}

\begin{proof}[Proof of Lemma~\ref{lem:m}]
Let $k\ge 4$, $m\in \mathbb{N}$, and $c>0$ be given.
Suppose $H=(V, E)$ is a $k$-graph on $m$ vertices with $\D(H)\le c m^{1/2}$ and without $(2, \ell)$-cycle for any $\ell =2, \dots, k-2$.
Fix a real number $\beta$ such that $\frac{k}{2k-2} < \b <1$. 
We will apply Theorem~\ref{thm:BLM} with the following two hypergraphs. Let $G_1$ be a $(k-1)$-graph $(V, E_1)$, where $E_1$ consists of all $(k-1)$-sets $S\subset V$ such that $\deg_{H}(S)\ge m^{\frac1{2k-2}} / (\log m)^{\b}$. Let $G_2$ be a $k$-graph $(V, E_2)$, where $E_2$ consists of all edges $e\in H$ such that $\binom{e}{k-1} \cap G_1 = \emptyset$. Thus 
\begin{equation}
\label{eq:DG2}
\D_{k-1}(G_2)\le m^{\frac1{2k-2}} / (\log m)^{\b}.
\end{equation}

Let $D= c(k-1) m^{\frac{k-2}{2k-2} } (\log m)^{\b}$, $d= c m^{1/2}$, and $0< \eps < \min\{\frac{k-2}{2k-2}, \frac1{k-2} \}$. Let us verify (i)--(iii). To see (i), we note that $D> m^{\eps}$ and
\begin{align}\label{eq:omega}
\omega:= d \left( \frac{\log m}{D} \right)^{\frac{k-1}{k-2}} = c m^{\frac12} \left( \frac{\log m}{ c(k-1) m^{\frac{k-2}{2k-2} } (\log m)^{\b} } \right)^{ \frac{k-1}{k-2} } 
= c \left( \frac{ (\log m)^{1-\b} }{ c(k-1) } \right)^{ \frac{k-1}{k-2} } \gg 1
\end{align}
because $\b< 1$. For $x\in V$, double counting gives
\[
\deg_H(x)=\frac{\sum_{S\ni x, |S| = k-1}\deg_H(S)}{k-1}\geq \frac{\deg_{G_1}(x)m^{\frac1{2k-2}} / (\log m)^{\b}}{k-1}
\]
because $\deg_H(S)\ge m^{\frac1{2k-2}} / (\log m)^{\b}$ for every $S\in G_1$. 
This implies that 
\[
\D(G_1) \le \frac{ (k-1) \D(H)}{m^{\frac1{2k-2}} / (\log m)^{\b} } \le \frac{ (k-1) c m^{1/2} }{m^{\frac1{2k-2}} / (\log m)^{\b} } = D.
\] 
For $2\le i\le k-2$, we claim that $\D_i(G_1)\le 1< D^{\frac1{k-2}-\eps}$. Suppose that there exists two distinct edges $e$ and $f$ of $G_1$ sharing $j\geq i $ vertices. Since $\deg_H(e), \deg_H(f)\geq m^{\frac1{2k-2}} / (\log m)^{\b}\geq k-j+2$, there exist vertices $x\neq y$ with $x, y \notin e\cup f$ such that $e\cup\{x\}, f\cup\{y\} \in H$. Thus $e\cup\{x\}$ and $f\cup\{y\}$ share exactly $j$ vertices, contradicting the assumption that $H$ contains no $(2, j)$-cycles. Consequently, $\Gamma(H_1) =0$. Thus (ii) holds.

To see (iii), we note that $|G_2| \le \D(G_2) m/k \le c m^{3/2} /k$. Thus $k |G_2| \le c m^{3/2} = d m$. For $2\le i\le k-2$, since $H$ contains no $(2, i)$-cycle,  we have $C_{G_2}(2, i)= 0$.  The number of $(2, k-1)$-cycles in $G_2$ can be counted by first selecting an edge $e_1$ of $G_2$, then selecting a $(k-1)$-subset $T$ of $e_1$ (in $k$ ways) and an edge $e_2\ne e_1$ of $G_2$ that contains $T$ (in at most $\D_{k-1}(G_2)$ ways), and finally dividing the result by two.
Therefore, in view of \eqref{eq:DG2}, we have 
\[
C_{G_2}(2, k-1) \le |G_2| \D_{k-1}(G_2) \frac{k}2 \le  \frac{c m^{\frac32}}{2} \frac{m^{\frac1{2k-2}} }{ (\log m)^{\b} } \ll m \left( \frac{ c(k-1) m^{\frac{k-2}{2k-2} } (\log m)^{\b} }{\log m} \right)^{\frac{k}{k-2}}
\]
because $-\b < (\b - 1)k/(k-2)$. We now apply Theorem~\ref{thm:BLM} with $\omega = \Omega( (\log m)^{(1-\b)\frac{k-1}{k-2}} )$ from \eqref{eq:omega} and obtain that
\[
\a(G_1\cup G_2)= \Omega \left( m  \left( \frac{\log \omega}{d}  \right)^{\frac1{k - 1}}   \right) = \Omega \left( m  \left( \frac{\log \log m}{ m^{1/2} }  \right)^{\frac1{k - 1}}   \right)
=\Omega  \left( m^{\frac{2k-3}{2k-2}} (\log \log m)^{\frac{1}{k-1}} \right),
\]
as desired.
\end{proof}

\section{Preliminary results and semi-random algorithm}
\label{sec:pre}

\subsection{Preliminary results}

For a vertex $x$ in a hypergraph $H$, recall that \emph{the link} $H_x$ is the set of all $e\setminus \{x\}$ for $e\in H$ such that $x\in e$. Let $N(x) = \{x\} \cup \bigcup_{e\in H_x} e$,
which consists of all vertices that lie on a common edge with $x$ (including $x$ itself). For $S\subseteq V$, define $N(S)= \bigcup_{x\in S} N(x)$, $N^0(S)=S$, and $N^i(S)= N(N^{i-1}(S))$ for $i\ge 1$. We define the \emph{distance} between two vertices $x$ and $y$ as the minimum $i$ such that $y\in N^{i}(x)$ (the distance is $\infty$ if such $i$ does not exist).

\begin{lemma}[Almost Regular Lemma]
\label{lem:AR}
Let $H= \{H_2, \ldots, H_k\}$ be a hypergraph on $V$ with the $\bouquet$ property such that $\D_1(H_i)\le a_{1,i}$ for $2\le i\le k$ and $\D_{i-1} (H_{i})\le a_{i-1,i}$ for $3\leq i \leq k$. 
Then there exists a hypergraph $H'= \{H'_2, \ldots, H'_k\} \supseteq H$ on $V$ such that 
\begin{enumerate}[i)]
\item $H' \in \bouquet$;
\item $\D_1(H'_i)\le a_{1,i}$ for $2\le i\le k$ and $\D_{i-1} (H'_{i})\le a_{i-1,i}$ for $3\leq i \leq k$;
\item there exists a set $B\subseteq V$ with at most $k^2b^3$ vertices such that all $x\in V\setminus B$ satisfy $\deg_{H'_{i}}(x) = a_{1,i}$ for $2\le i\le k$, where $b= 1 + \sum_{i=2}^k (i-1)a_{1,i}$.
\end{enumerate}
\end{lemma}

\begin{proof}
Let $H'= \{H'_2, \ldots, H'_k\}$ be the largest supergraph of $H$ on $V$ satisfying $i)$ and $ii)$. For $2\le i\le k$, let $B_i = \{x\in V: \deg_{H'_{i}}(x)< a_{1,i} \}$ and $B= \bigcup_{i=2}^k B_i$. Note that for $x\not\in B$,  we have $\deg_{H'_i}(x) = a_{1,i}$ for every $2\le i\le k$.
We claim that $|B_i|\le (i-1) b^3$ for $2\le i\le k$ and consequently $|B|\le \sum_{i=2}^k |B_i|\le k^2b^3$. Suppose $|B_i| > (i-1) b^3$ for some $i$. Let $x_1$ be an arbitrary vertex of $B_i$. For $2\le j\le i$, we inductively select $x_j\in B_i\setminus (N^3(x_1)\cup \dots \cup N^3(x_{j-1}))$ (recall that $N^3(y)$ is the set of vertices that are at distance at most $3$ from $y$). Such $x_j$ exists because for any vertex $x\in H'$, $|N^3(x)|\leq\left(1+\sum_{j=2}^k (j-1)a_{1,j}\right)^3= b^3$ and hence 
\[
|B_i| - \left( |N^3(x_1)| + \cdots + |N^3(x_{j-1})| \right) > (i - 1) b^3 - (j-1) b^3 \ge 0.
\]
Since $x_2\not\in N(x_1)$, the set $f:=\{x_1, \dots, x_i\}$ is not an edge of $H'_i$. We add $f$ to $H'$. The addition of $f$ does not create a new cycle of length at most $4$ because every pair of vertices of $f$ is at distance at least $4$ from each other in $H'$. 
Therefore the resulting hypergraph $H''$ does not create any forbidden structure stated in Properties i) - v) of $\bouquet$ and thus $H'' \in \bouquet$. In addition, $H''$ satisfies $ii)$ because $f \subseteq B_i$. Thus we obtain a larger hypergraph satisfying both $i)$ and $ii)$, which is a contradiction.
\end{proof}

We also need a lemma on the concentration of almost independent Bernoulli random variables.

\begin{lemma}[Almost Independent Lemma]
\label{lem:AI}
Let $X_1, \dots, X_n$ be Bernoulli random variables with $\mP(X_i=1)\ge p$ such that each $X_i$ is mutually independent of all but at most $s$ other $X_j$'s.  Let $X= X_1 + \cdots + X_n$ and $\mu = \mathbb{E}(X)$. For any $0< \eps< 3/2$, 
\[
\mathbb{P}\left( |X - \mu | > \eps \mu \right) \le 2(s+1) \exp\left( - \frac{ \eps^2 pn }{3 (s+1)} \right).
\]
\end{lemma}

\begin{proof}
Let $H$ denote the dependency graph on $X_1, \dots, X_n$, in which each $X_i$ is adjacent to at most $s$ other vertices. By the Hajnal-Szemer\'edi theorem \cite{HS70}, the vertices of a graph $H$ of maximum degree $\D(H)\leq s$ can be partitioned into $s+1$ independent sets $V_1, \dots, V_{s+1}$ of size $\lfloor \frac{n}{s+1} \rfloor$ and $\lceil \frac{n}{s+1} \rceil$. Consider such a partition for our dependency graph. Let $Y_j = \sum_{X_i\in V_j} X_i$ and $\mathbb{E}(Y_j) = \mu_j \ge p\lfloor \frac{n}{s+1} \rfloor$ for $1\le j\le s+1$. Since $V_j$ consists of mutually independent $X_i$'s, we can apply Chernoff's bound and derive that $\mathbb{P}\left( |Y_j - \mu_j | > \eps \mu_j \right) \le 2 \exp\left( - \eps^2 \mu_j /{3} \right)$. Consequently
\begin{align*}
\mathbb{P}\left( |X - \mu | > \eps \mu \right) &\le \bigcup_{j=1}^{s+1} \mP\left( |Y_j - \mu_j | > \eps \mu_j \right)
\le \sum_{j=1}^{s+1} \mathbb{P}\left( |Y_j - \mu_j | > \eps \mu_j \right)\\
&\le 2(s+1) \exp\left( - \frac{ \eps^2 pn }{3 (s+1)} \right). \qedhere
\end{align*}
\end{proof}

\bigskip

\subsection{Semi-random algorithm} 

To prove Theorem \ref{thm:main} we use a modification of the semi-random algorithm of Ajtai et al. \cite{MR657047}. Let $H=\{H_2,\ldots,H_k\}$ be a hypergraph on a set of vertices $V:=V(H)$, where each $H_i$ is an $i$-graph. Let $M$ be a positive integer, $H^0=H$ and $V_0=V(H^0)$. For $0\leq m \leq M-1$, we recursively construct a hypergraph $H^{m+1}$ with vertex set $V_{m+1} \subseteq V$ and a set $I_{m+1}$ of $H^m$. Given $H^m$ with vertex set $V_m$, we apply Lemma \ref{lem:AR} to obtain a superhypergraph $\tilde{H}^m\supseteq H^m$ and a small set $B$ such that $\tilde{H}^m$ is regular on $V_m \setminus B$. Since an independent set in $\tilde{H}^m$ is an independent set in $H^m$, we still denote $\tilde{H}^m$ by $H^m$. The almost regularity of $H^m$ will be useful in the later steps of the proof (see Section~\ref{sec:51}). 

Choose a set $C_{m+1} \subseteq V_m$ uniformly and independently with probability $p_{m+1}$ to be defined later. Define 
\[
D_{m+1} := \{ x\in V_m: e\subseteq C_{m+1} \text{ for some } e\in (H^m)_x \},
\]
i.e., $D_{m+1}$ is the set of all vertices $x$ of $H^m$ such that at least one of the edges in the link $(H^m)_x$ is completely contained in $C_{m+1}$. 

At every iteration we will define a waste set $W_{m+1}:=W_{m+1}(H^m,C_{m+1})$ depending on the hypergraph $H^m$ and the choice of set $C_{m+1}$. The definition of $W_{m+1}$ is somewhat extensive and we leave it to a later stage of the proof. With the waste set defined, let 
\begin{align*}
    I_{m+1} :=C_{m+1}\setminus (D_{m+1}\cup W_{m+1}).
\end{align*}

We define $H^{m+1}$ as the hypergraph with vertex set
\begin{align*}
    V_{m+1} :=V_m\setminus \left(C_{m+1}\cup D_{m+1}\cup W_{m+1}\right).
\end{align*}
The edges of $H^{m+1}$ are obtained as follows. First consider the set of edges 
\[
G^{m+1} :=H^{m}\left[V_{m+1}\cup I_{m+1}\right],
\]
i.e., all the edges of the induced subgraph of $H^{m}$ on the set $V_{m+1}\cup I_{m+1}$. Let \[
G^{m+1}\mid_{V_{m+1}}:=\left\{e\cap V_{m+1}: e\in G^{m+1}\right\}
\]
be the \emph{contraction} (also known as \emph{restriction} or \emph{trace}) of $G^{m+1}$ on $V_{m+1}$. We obtain $H^{m+1}$ from $G^{m+1}\mid_{V_{m+1}}$ after removing multiplicities and edges $e \in G^{m+1}\mid_{V_{m+1}}$ such that there exists another $f \in G^{m+1}\mid_{V_{m+1}}$ with $f\subseteq e$. In other words, the hypergraph $H^{m+1}$ is obtained by contracting the induced graph $H^m[V_{m+1}\cup I_{m+1}]$ on the set $V_{m+1}$ and then removing multiplicities and any edge that properly contains another edge. Figure \ref{fig:fig1} shows how to obtain 
$H^{m+1}$ from $H^{m}$.

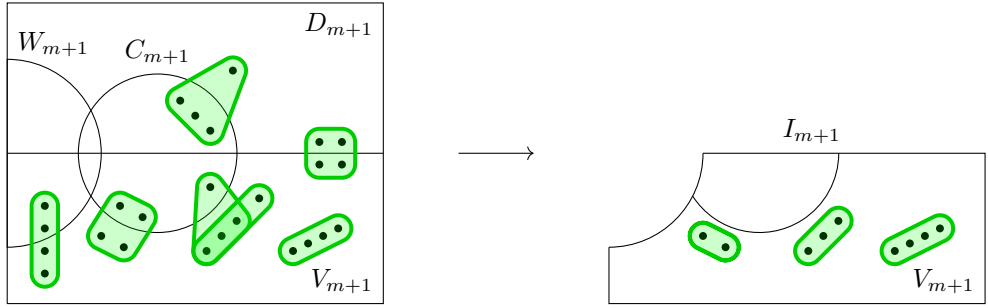
\begin{figure}[h]
	\begin{tikzpicture}[scale=1]
		
		\coordinate (A) at (0,0);
		\coordinate (B) at (0,4);
		\coordinate (C) at (5,4);
		\coordinate (D) at (5,0);
		\coordinate (M) at (0,2);
		\coordinate (N) at (5,2);
		
		\coordinate (x_1) at (4.45,2.15);
		\coordinate (y_1) at (4.15,2.15);
		\coordinate (z_1) at (4.15,1.85);
		\coordinate (w_1) at (4.45,1.85);
		
		\coordinate (x_2) at (4.4,1);
		\coordinate (y_2) at (4.2,0.9);
		\coordinate (z_2) at (4.0,0.8);
		\coordinate (w_2) at (3.8,0.7);
		
		\coordinate (x_3) at (3,3.1);
		\coordinate (y_3) at (2.3,2.7);
		\coordinate (z_3) at (2.5,2.5);
		\coordinate (w_3) at (2.7,2.3);
		
		\coordinate (x_4) at (0.5,0.4);
		\coordinate (y_4) at (0.5,0.7);
		\coordinate (z_4) at (0.5,1);
		\coordinate (w_4) at (0.5,1.3);
		
		\coordinate (x_5) at (1.8,1.15);
		\coordinate (y_5) at (1.5,1.3);
		\coordinate (z_5) at (1.25,0.9);
		\coordinate (w_5) at (1.55,0.75);
		
		\coordinate (x_6) at (2.7,1.55);
		\coordinate (y_6) at (2.65,0.7);
		\coordinate (z_6) at (2.85,0.9);
		\coordinate (w_6) at (3.05,1.1);
		\coordinate (v_6) at (3.35,1.4);
	    
	    \draw (A)--(B)--(C)--(D)--(A);
	    \draw (M)--(N);
	    
	    \draw (2,2) circle (30pt);
	    \draw (0,3.25) -- (0,0.75) arc(-90:90:1.25) --cycle;
	    
	    \fill (x_1) circle (1.5pt);
		\fill (y_1) circle (1.5pt);
		\fill (z_1) circle (1.5pt);
		\fill (w_1) circle (1.5pt);
		
		\fill (x_2) circle (1.5pt);
		\fill (y_2) circle (1.5pt);
		\fill (z_2) circle (1.5pt);
		\fill (w_2) circle (1.5pt);
		
		\fill (x_3) circle (1.5pt);
		\fill (y_3) circle (1.5pt);
		\fill (z_3) circle (1.5pt);
		\fill (w_3) circle (1.5pt);
		
		\fill (x_4) circle (1.5pt);
		\fill (y_4) circle (1.5pt);
		\fill (z_4) circle (1.5pt);
		\fill (w_4) circle (1.5pt);
		
		\fill (x_5) circle (1.5pt);
		\fill (y_5) circle (1.5pt);
		\fill (z_5) circle (1.5pt);
		\fill (w_5) circle (1.5pt);
		
		\fill (x_6) circle (1.5pt);
		\fill (y_6) circle (1.5pt);
		\fill (z_6) circle (1.5pt);
		\fill (w_6) circle (1.5pt);
		\fill (v_6) circle (1.5pt);
				
		\redge{(w_1)}{(z_1)}{(y_1)}{(x_1)}{5pt}{1.5pt}{green!80!black}{green,opacity=0.2};
		\redge{(w_2)}{(z_2)}{(x_2)}{(y_2)}{5pt}{1.5pt}{green!80!black}{green,opacity=0.2};
		\redge{(w_3)}{(z_3)}{(y_3)}{(x_3)}{5pt}{1.5pt}{green!80!black}{green,opacity=0.2};
	    \redge{(w_4)}{(z_4)}{(x_4)}{(y_4)}{5pt}{1.5pt}{green!80!black}{green,opacity=0.2};
	    \redge{(w_5)}{(z_5)}{(y_5)}{(x_5)}{5pt}{1.5pt}{green!80!black}{green,opacity=0.2};
	    \redge{(w_6)}{(z_6)}{(y_6)}{(x_6)}{5pt}{1.5pt}{green!80!black}{green,opacity=0.2};
	    \redge{(w_6)}{(v_6)}{(z_6)}{(y_6)}{5pt}{1.5pt}{green!80!black}{green,opacity=0.2};
	    
	    \coordinate (P) at (8,2);
	    \coordinate (Q) at (10,2);
	    
		\tkzInterCC[R](P,1.25cm)(Q,30pt) \tkzGetPoints{O1}{C1};
		
		\coordinate (A1) at (8,0);
		\coordinate (B1) at (8,0.75);
		\coordinate (D1) at (10cm+30pt,2);
		\coordinate (E1) at (13,2);
		\coordinate (F1) at (13,0);
		\coordinate (G1) at (9.25,2);
		
		\tkzDrawArc[R with nodes, color=black](P,1.25cm)(B1,G1);
		\tkzDrawArc[R with nodes, color=black](Q,30pt)(C1,D1);
		\draw (G1)--(E1)--(F1)--(A1)--(B1);
	    
	    \coordinate (x'_2) at (12.4,1);
		\coordinate (y'_2) at (12.2,0.9);
		\coordinate (z'_2) at (12.0,0.8);
		\coordinate (w'_2) at (11.8,0.7);
		
		\coordinate (z'_5) at (9.25,0.9);
		\coordinate (w'_5) at (9.55,0.75);
		
		\coordinate (y'_6) at (10.65,0.7);
		\coordinate (z'_6) at (10.85,0.9);
		\coordinate (w'_6) at (11.05,1.1);
		
		\fill (x'_2) circle (1.5pt);
		\fill (y'_2) circle (1.5pt);
		\fill (z'_2) circle (1.5pt);
		\fill (w'_2) circle (1.5pt);
		
		\fill (z'_5) circle (1.5pt);
		\fill (w'_5) circle (1.5pt);
		
		\fill (y'_6) circle (1.5pt);
		\fill (z'_6) circle (1.5pt);
		\fill (w'_6) circle (1.5pt);
	   
	   \redge{(w'_2)}{(z'_2)}{(x'_2)}{(y'_2)}{5pt}{1.5pt}{green!80!black}{green,opacity=0.2};
	   \redge{(w'_5)}{(z'_5)}{(w'_5)}{(z'_5)}{5pt}{1.5pt}{green!80!black}{green,opacity=0.2};
	   \redge{(w'_6)}{(z'_6)}{(y'_6)}{(w'_6)}{5pt}{1.5pt}{green!80!black}{green,opacity=0.2};
	   
	   \node (q1) at (5,4) [below left, font=\small] {$D_{m+1}$};
	   \node (q2) at (5,0) [above left, font=\small] {$V_{m+1}$};
	   \node (q3) at (0,3.75) [below right, font=\small] {$W_{m+1}$};
	   \node (q4) at (2,2cm +30pt) [above, font=\small] {$C_{m+1}$};
	   \node (q5) at (13,0) [above left, font=\small] {$V_{m+1}$};
	   \node (q6) at (10.7,2) [above, font=\small] {$I_{m+1}$};
	   
	   \draw[->] (6,2)--(7,2);
	\end{tikzpicture}
	\caption{Hypergraphs $H^m$ and $H^{m+1}$}
	\label{fig:fig1}
\end{figure}

Finally, the output of our algorithm is given by the union of the disjoint sets
\[
I:=\bigcup_{m=1}^M I_m.
\]
Here is a formal description of the algorithm explained above.

\begin{algorithm}[H]\caption{{\bfseries AKPSS Algorithm}}\label{alg:AKPSS} 
\begin{algorithmic}[1]
\STATE Set~$H^0:=H$ and~$V_0:=V(H)$.
\FOR{$m = 0$ \textbf{to} $M-1$}
		\STATE Update $H^m$ by applying Almost regular lemma (Lemma \ref{lem:AR})
    \STATE Generate~$C_{m+1}$: independently include each vertex of $V_{m}$ with probability~$p_{m+1}$. 
    \STATE Set $D_{m+1}:=\{x\in V_m:\: e\subseteq C_{m+1} \text{ for some } e\in (H^m)_x\} $. 
    \STATE Set~$I_{m+1}:=C_{m+1}\setminus (D_{m+1}\cup W_{m+1})$.
      \STATE Set $V_{m+1}:=V_m\setminus (C_{m+1}\cup D_{m+1}\cup W_{m+1})$ and $G^{m+1} :=H^{m}\left[V_{m+1}\cup I_{m+1}\right]$.
      \STATE Construct $H^{m+1}$ by removing multiplicities and edges that properly contain other edges from $G^{m+1}\mid_{V_{m+1}}$.
\ENDFOR
\STATE Return~$I:=\bigcup_{m=1}^M I_m$.
\end{algorithmic}
\end{algorithm}

Our first observation of the Algorithm is that it actually works.

\begin{proposition}\label{cl:independence}
The set $I$ obtained by AKPSS algorithm is an independent set of $H$.
\end{proposition}

\begin{proof}
The proof is by a reversed induction on the steps of the algorithm. We claim that for $m=M,M-1,\ldots,1$ the set $\bigcup_{j=m}^M I_j$ is an independent set of $H^{m-1}$. In particular, $I=\bigcup_{j=1}^M I_j$ is an independent set of $H^0=H$. 

The base case is that $I_M$ is an independent set of $H^{M-1}$. We claim that $I_m$ is an independent set of $H^{m-1}$ for every $m\ge 1$. Indeed, since the updated $H^{m-1}$ obtained after Lemma \ref{lem:AR} is a supergraph of the original $H^{m-1}$, we just need to show that $I_m$ is an independent set of the updated $H^{m-1}$. Suppose that there exists an edge $e\in H^{m-1}$ with $e\subseteq I_m$. Let $x$ be a vertex of $e$. By the definition of $I_m$, we have that $e\setminus\{x\} \in C_m$. Hence $x\in D_m$, which is a contradiction because $D_m\cap I_m=\emptyset$.

Now suppose that $\bigcup_{j=m+1}^MI_j$ is an independent set of $H^m$. We will show that $\bigcup_{j=m}^M I_j$ is an independent set of $H^{m-1}$. We know that $I_{m}$ is an independent set of $H^{m-1}$. Hence, if $e$ is an edge of $H^{m-1}$ in $\bigcup_{j=m}^M I_j$, then $f:=e \cap \left(\bigcup_{j=m+1}^M I_j \right)\neq \emptyset$. Since $\bigcup_{j=m+1}^M I_j \subseteq V_{m}$ and $I_m \cap V_m=\emptyset$, we obtain that $f \in G^{m}\mid_{V_m}$. By construction, $H^m$ is obtained from $G^m\mid_{V_m}$ by removing multiplicities or edges that properly contains other edge. Therefore, $H^m$ contains an edge $f'\subseteq f$ (it could be that $f'=f)$ and consequently $f' \subseteq \bigcup_{j=m+1}^M I_j$, which contradicts the assumption that $\bigcup_{j=m+1}^M I_j$ is independent in $H^m$.
\end{proof}

\subsection{Structural results}

Now we focus on some structural results concerning the $\bouquet$ property. The following lemma shows that the algorithm AKPSS described above preserves the property.

\begin{lemma}\label{lem:structural}
If $H^{m} \in \bouquet$, then $H^{m+1} \in \bouquet$.
\end{lemma}

\begin{proof}
For simplicity of notation, let $H=H^m=\{H_2,\ldots, H_k\}$ and let $H^*=H^{m+1}=\{H_2^*, \ldots, H_k^*\}$. By the description of the algorithm, the graph $H^*$ has no multiple edges or edges that properly contain other edges. Moreover, every edge $e^* \in H^*$ is obtained from an edge of~$H$ by contracting to $V^*=V(H^*)$. Let $e$ be the corresponding edge from $H$ that was contracted to $e^*$, i.e., $e^*=e\cap V^*$. Note that in some cases $e=e^*$. The following proposition describes the possible relations between intersection on $H^*$ and $H$.

\begin{claim}\label{claim:intersections}
Let $e^*,f^* \in H^*$ and $e,f \in H$ be the corresponding edges that were contracted to $e^*$ and $f^*$, respectively. Then the following holds:
\begin{enumerate}
    \item If $|e^*\cap f^*|>1$, then $e,f \in H_r$ for some $r$ and $|e\cap f|=r-1$.
    \item If $e^*\notin H_2^*$ and $|e^* \cap f^*|=1$, then $|e\cap f|=1$.
    \item If $e^*\cap f^*=\emptyset$, then $|e\cap f|\leq 1$.
\end{enumerate}
\end{claim}

\begin{proof}
The proof of $(1)$ is just a consequence of $H$ satisfying $\bouquet$. Since $1<|e^*\cap f^*|\leq|e\cap f|$ and $H \in \bouquet$, we have by Property i) of $\bouquet$ that $|e|=|f|$. Thus $e,f \in H_r$ for some $r$ and by Property ii) we have that $|e\cap f|=r-1$.

For $(2)$, suppose that $e^*\notin H_2^*$ and $|e^* \cap f^*|=1$ but $|e\cap f|>1$. Then, the argument above shows that $e,f \in H_r$ for some $r$ and~$|e\cap f|=r-1$. Since 
\[
| e\cap f|= |e^*\cap f^*| + |(e \cap f)\setminus V^*| = 1 +  |(e \cap f)\setminus V^*|,
\] 
we obtain that~$|(e\cap f)\setminus V^*|=r-2$ and thus $|e\setminus e^*|=|f\setminus f^*|=r-2$. This implies that both $e^*$ and $f^*$ are in $H_2^*$, which is a contradiction.

Finally, to prove $(3)$, suppose that $e^*\cap f^*=\emptyset$ and $|e\cap f|>1$. Then by Properties i) and ii) of $\bouquet$, we obtain that $e,f\in H_r$ for some $r$ and $|e\cap f|=r-1$. Since
\[
| e\cap f|= |e^*\cap f^*| + |(e \cap f)\setminus V^*| = |(e \cap f)\setminus V^*|,
\]
we obtain that $|(e\cap f)\setminus V^*|=r-1$ and thus $|e^*|=|f^*|=1$, which is a contradiction. 
\end{proof}

To check that $H^*$ satisfies $\bouquet$ we just need to verify all conditions.

\begin{enumerate}[i)]
    \item Suppose that there are edges $e^*\in H_i^*$ and $f^* \in H_j^*$ such that $|e^* \cap f^*|>1$ and $i<j$. By Claim \ref{claim:intersections} we have that $e,f \in H_r$ for some $r$ and~$|e\cap f|=r-1$. Moreover,
    \[
    r-1= |e \cap f|= |e^*\cap f^*|+|(e\cap f)\setminus V^*|\leq |e^*\cap f^*|+|f\setminus f^*|\leq |e^*\cap f^*|+r-j.
    \] 
    This implies that $|e^*\cap f^*|\geq j-1\geq i$, which forces $e^* \cap f^* = e^*$ contradicting the assumption that $H^*$ does not contain edges that are proper subset of other edges.
    
    \item If $e^*, f^* \in H_i^*$ and $|e^*\cap f^*|>1$, then by Claim \ref{claim:intersections} we have that $e,f \in H_r$ for some $r$ and $|e\cap f|=r-1$. Let $e\setminus f=\{x\}$ and $f\setminus e=\{y\}$. Since $e^*$ and $f^*$ are not properly contained in each other, then $x\in e^*$ and $y\in f^*$, which implies that $|e^*\cap f^*|=i-1$.
    
    \item Suppose that $H^*$ has a linear $3$-cycle with at most one edge in $H_2^*$. Let $e_1^*$, $e_2^*$ and $e_3^*$ be the edges of this cycle and let $e_1$, $e_2$ and $e_3$ be the corresponding edges in $H$. Since at most one of the $e_i^*$ is in $H_2^*$, then by Claim \ref{claim:intersections} we have that $|e_1\cap e_2|=|e_2\cap e_3|=|e_3\cap e_1|=1$ and $e_1$, $e_2$, $e_3$ forms a linear $3$-cycle with at most one edge in $H_2$, which contradicts Property iii) of $\bouquet$.
    
    \item Suppose that $H^*$ has a $4$-cycle with clean intersections. Let $e_1^*, e_2^*, e_3^*$ and $e_4^*$ be the edges of the cycle and $e_1, e_2, e_3$ and $e_4$ the corresponding edges in $H$. By definition of clean intersections, we have that $e_1^*\cap e_3^*=\emptyset$ and $e_2^*\cap e_4^*=\emptyset$. Hence, by Claim \ref{claim:intersections}, we obtain that $|e_1\cap e_3|\leq 1$ and $|e_2\cap e_4|\leq 1$.
    
    Since $e_1,e_2,e_3$ and $e_4$ forms a $4$-cycle in $H$ and $H$ satisfies Property iv) of $\bouquet$, we have that either $|e_1\cap e_3|=1$ or $|e_2\cap e_4|=1$. Suppose without loss of generality that $|e_1\cap e_3|=1$. Assume $e_1\cap e_3=\{y\}$ for some $y\notin V(H^*)$. We split the proof into two cases. First suppose that $y\notin e_2$. Thus 
    \begin{align*}
        \{y\}\cup (e_1^*\setminus e_2^*)&\subseteq e_1\setminus e_2,\\
        \{y\}\cup (e_3^*\setminus e_2^*)&\subseteq e_3\setminus e_2.
    \end{align*}
    Because $e_1^*\cap e_4^* \neq \emptyset$ and $e_2^*\cap e_4^*=\emptyset$, we have that $e_1^*\setminus e_2^*\neq \emptyset$. Similarly, $e_3^*\setminus e_2^*\neq \emptyset$. This implies that $|e_1\setminus e_2|\geq 2$ and $|e_3\setminus e_2|\geq 2$. Thus, by Property i)
    of $\bouquet$ and the fact that $e_1,e_2,e_3,e_4$ is a cycle, we obtain that $|e_1\cap e_2|=1$ and $|e_2\cap e_3|=1$. Hence, $e_1,e_2,e_3$ is a linear $3$-cycle with $e_1,e_2\notin H_2$, which is a contradiction to Property iii).
    
    For the second case, suppose that $y\in e_2$. Thus
    \begin{align*}
        (e_1^*\cap e_2^*)\cup\{y\}&\subseteq (e_1\cap e_2),\\
        (e_2^*\cap e_3^*)\cup\{y\}&\subseteq (e_2\cap e_3),
    \end{align*}
    which implies that $|e_1\cap e_2|\geq 2$ and $|e_2\cap e_3|\geq 2$. Hence, by Properties i) and ii) of $\bouquet$, we have that $e_1,e_2,e_3 \in H_r$ for some $r\geq 3$ and $|e_1\cap e_2|=|e_2\cap e_3|=r-1$. Therefore $r-2\leq |e_1\cap e_3|=1$, which implies that $r=3$. Thus, we obtained three edges $e_1,e_2,e_3$ such that $|e_1\cap e_2|=|e_2\cap e_3|=2$ and $|e_1\cap e_3|=1$, which contradicts Property v) of $\bouquet$.
    
    \item Suppose that there exist three edges $e^*_1, e^*_2, e^*_3\in H_3^*$ such that $|e_1^* \cap e^*_2|=|e_2^*\cap e_3^*|= 2$ and $|e_1^*\cap e_3^*|=1$. Then, by Claim \ref{claim:intersections}, we have $e_1, e_2, e_3 \in H_r$ for some $r\geq 3$ and also that $|e_1\cap e_2|=|e_2\cap e_3|=r-1$ and $|e_1\cap e_3|=1$. Therefore $r-2\leq|e_1\cap e_3|=1$ and we obtain that $r=3$. Hence, $e_1,e_2$ and $e_3$ contradicts Property v) of $\bouquet$. \qedhere
\end{enumerate}
\end{proof}

The reason why we call such a family of graphs as bouquets is due to the following fact. 

\begin{figure}[h]
	\begin{tikzpicture}[scale=0.5]
		
		\coordinate (x) at (0,0);
		\coordinate (y) at (2,0);
		\coordinate (z) at (4,0);
		\coordinate (w1) at (6,-3);
		\coordinate (w2) at (6,-2);
		\coordinate (w3) at (6,-1);
		\coordinate (w4) at (6,0);
		\coordinate (w5) at (6,1);
		\coordinate (w6) at (6,2);
		\coordinate (w7) at (6,3);

		\redge{(w1)}{(x)}{(y)}{(z)}{9pt}{1.5pt}{green!80!black}{green,opacity=0.2};
		\redge{(w2)}{(x)}{(y)}{(z)}{9pt}{1.5pt}{green!80!black}{green,opacity=0.2};
		\redge{(w3)}{(x)}{(y)}{(z)}{9pt}{1.5pt}{green!80!black}{green,opacity=0.2};
		\redge{(w4)}{(z)}{(y)}{(x)}{9pt}{1.5pt}{green!80!black}{green,opacity=0.2};
		\redge{(w5)}{(z)}{(y)}{(x)}{9pt}{1.5pt}{green!80!black}{green,opacity=0.2};
		\redge{(w6)}{(z)}{(y)}{(x)}{9pt}{1.5pt}{green!80!black}{green,opacity=0.2};
		\redge{(w7)}{(z)}{(y)}{(x)}{9pt}{1.5pt}{green!80!black}{green,opacity=0.2};
		
		\fill (x) circle (3pt);
		\fill (y) circle (3pt);
		\fill (z) circle (3pt);
		\fill (w1) circle (3pt);
		\fill (w2) circle (3pt);
		\fill (w3) circle (3pt);
		\fill (w4) circle (3pt);
		\fill (w5) circle (3pt);
		\fill (w6) circle (3pt);
		\fill (w7) circle (3pt);
		
	\end{tikzpicture}
	\caption{A bouquet}
	\label{fig:fig2}
\end{figure}
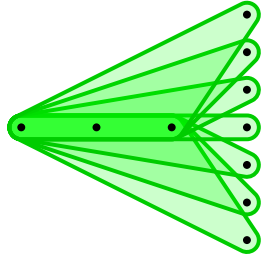

\begin{proposition}\label{cl:bouquet}
Suppose $H\in \bouquet$ contains a subhypergraph $F$ such that $|e\cap f| \ge 2$ for every $e,f \in F$. 
Then there exists $3\le i\le k$ such that $F\subseteq H_i$ and
either $F$ is a subhypergraph of $K_{i+1}^i$ or all members of $F$ share a common $(i-1)$-subset.
\end{proposition}

\begin{proof}
We assume $|F|>1$ -- otherwise the assertions holds trivially.
Since $|e\cap f| \ge 2$ for every $e,f \in F$, Properties i) and ii) of $\bouquet$ imply that $F\subseteq H_i$ for some $3\le i\le k$ and
$ |e\cap f|=i-1$ for every $e,f \in F$.

Fix $e, f\in F$ with $S=e \cap f$. For every $f'\in F\setminus\{e, f\}$, since $|f'\cap f|= |f'\cap e|= i-1$, exactly one of the following holds:
\begin{enumerate}[(i)]
\item there exists $x_0\in S$ such that $f'\cap e = e\setminus \{x_0\}$ and $f'\cap f = f\setminus \{x_0\}$ (thus $f\subset e\cup f$).
\item $f'\cap e= f'\cap f= S$ (thus $f\not\subseteq e\cup f$);
\end{enumerate}
Furthermore, either all $f'\in F\setminus\{e, f\}$ satisfy (i) or all $f'\in F\setminus\{e, f\}$ satisfy (ii) because an $f'$ satisfying (i) and an $f''$ satisfying (ii) force $|f'\cap f''|= i-2$, a contradiction.  The former case implies that $F$ is a subhypergraph of $K_{i+1}^i$ while the latter case implies that all members of $F$ contains $S$ as a subset.
\end{proof}

For $x\in V$, we often view its link $H_x$ as a hypergraph with vertex set $V\setminus \{x\}$ and edge set $H_x$. The following proposition describes the structure of $H_x$ when $H\in \bouquet$. A \emph{(connected) component} of a hypergraph $G$ is a maximal subhypergraph $C$ such that for every two edges $e, f\in C$, there are edges $e_0= e, e_1, \dots, e_r= f\in G$ such that $e_i\cap e_{i-1} \ne \emptyset$ for $i=1, \dots, r$. 

\begin{proposition}
\label{clm:BouStr}
Suppose $H\in \bouquet$ and $x\in V(H)$. For every component $C$ of size $|C|\ge 2$ in the link $H_{x}$, there exists $3\le i\le k$ such that all members of $C$ have size $i-1$, and $|V(C)|\le i$ or $|C|\le \D_{i-1}(H_i)$.
In particular, 
$|C|\le \max\{i, \D_{i-1}(H_i)\}$.
\end{proposition}
\begin{proof}
Property i) of $\bouquet$ and the definition of $C$ imply that all edges of $C$ have the same size, say, $i-1$ for some $2\le i\le k$. When $i=2$, we have $|C|=1$, contradicting our assumption. Hence $i\ge 3$.

By property ii) of $\bouquet$, if $e,f\in C$ are intersecting (namely, $e\cap f\neq \emptyset$), then $|e\cap f|= i-2$. We claim that all edges of $C$ intersects with each other. Suppose to the contrary, there are two disjoint edges in $C$. Since $C$ is a connected component, we may assume that there are edges $e, f, g\in C$ such that $e\cap f\neq \emptyset$ and $f\cap g\neq \emptyset$ but $e\cap g=\emptyset$. Thus, $|e\cap f|=|f \cap g|=i-2$. This implies that $|e\cap g|\ge i-3$ and since $e\cap g=\emptyset$, we have that $i=3$. Hence, $e\cup \{x\}, f\cup\{x\}, g\cup\{x\}$ are edges such that $|(e\cup\{x\})\cap (g\cup\{x\})|=1$ and $|(f\cup\{x\}) \cap (e\cup\{x\})|=|(f\cup\{x\}) \cap (g\cup\{x\})|=2$, which contradicts Property v) of $\bouquet$.

Therefore every $e\ne f\in C$ satisfy $e\cap f \neq \emptyset$. Applying Proposition \ref{cl:bouquet} to $F= \{ \{x\}\cup f: f\in C\}$, we know that either $F$ is a subhypergraph of $K^i_{i+1}$ or all members of $F$ contains a common $(i-1)$-subset. This implies that  either $C$ is a subhypergrah of $K_i^{i-1}$ or all members of $C$ contains a common $(i-2)$-subset. In the latter case we have $|C|\leq \D_{i-1}(H_i)$.
\end{proof}

We can strengthen Property v) of $\bouquet$ as follows and apply it in some proofs.
\begin{enumerate}
    \item[v')] There are no three edges $e_1$, $e_2$, $e_3\in H$ such that $|e_1 \cap e_2|=|e_2\cap e_3|=\ell-1$ and $|e_1\cap e_3|=\ell-2$ for some $\ell \ge 3$.
\end{enumerate}
Indeed, suppose such $e_1$, $e_2, e_3$ exist. Since $\ell-1\ge 2$, by Properties i) and ii) of $\bouquet$, we have $e_1, e_2, e_3\in H_\ell$. Then, $|e_1\cap e_3|=\ell-2$ forces that $\ell=3$. Now $e_1$, $e_2, e_3$ contradict Property v) of $\bouquet$.

\section{Key lemma and a proof of Theorem~\ref{thm:main}}
\label{sec:key}
 
In this section we state a key lemma, from which we can derive Theorem~\ref{thm:main} easily. First we introduce several variables. 
Given $N, T, k$ as in Theorem~\ref{thm:main}, we let  
\begin{align}\label{eq:ebM}
\eps=\frac{1}{\log T}, \quad \b = \frac1{1+\eps}, \quad M_0 =\log T \quad \text{and} \quad M= \frac{\log T}{2}.
\end{align}
We construct two sequences $\{\alpha_m\}_{m=0}^{M}$ and $\{\gamma_m\}_{m=1}^{M}$ as follows:
\begin{align}\label{eq:ag1}
\a_{0} = M_0^{\frac1{k-1}}, \quad \a_m = \left(M_0+ 1 + \b + \dots + \b^{m-1} \right)^{\frac1{k-1}} 
 \quad \text{and} \quad \g_m = \a_m - \a_{m-1} 
\end{align}
for $1\le m\le M$. 
Observe that $\{ \a_m \}$ is an increasing sequence and hence $\g_m\ge 0$ for all $m$. Next we will estimate the values of $\a_m$ and $\g_m$.
For every $0\le m\le M$, since $\eps m \le 1/2 $, we have
\begin{align*}
(1+\eps)^m\le e^{\eps m} \le e^{1/2} < 2, \quad \text{and consequently,} \quad \frac12 < \frac{1}{(1+\eps)^m} = \b^m \le 1.
\end{align*}
Thus, we derive the following bounds for $\a_m^{k-1}$:
\begin{align}
 \log T \le M_0+\frac m2 \le \a^{k-1}_{m} = M_0 +1 + \b + \dots + \b^{m-1} \le M_0+m\le \frac32 \log T.
 \label{eq:amk1}
\end{align}
For $0\le m\le M-1$, since
\[
\g_{m+1} = \a_{m+1} - \a_m = \frac{\a_{m+1}^{k-1}-\a_m^{k-1}}{ \sum_{i=0}^{k-2} \a_{m+1}^{k-2-i} \a_m^i } = \frac{\b^m}{ \sum_{i=0}^{k-2} \a_{m+1}^{k-2-i} \a_m^i },
\]
it follows that
\begin{align}\label{eq:gm1}
\frac {1/2}{(k-1)(\frac{3}{2}\log T)^{\frac{k-2}{k-1}}}  \le \frac{\b^{m}}{(k-1)\a_{m+1}^{k-2}} \le \g_{m+1} \le \frac{\b^{m}}{ (k-1) \a_{m}^{k-2} }  \le  \frac1{(k-1)(\log T)^{\frac{k-2}{k-1}}}
\end{align}
by \eqref{eq:amk1}. 

The next lemma uses the two sequences just described to control the parameters of the hypergraphs $H^m$ obtained in the AKPSS algorithm. Given real numbers $a$ and $b$, we say that $a=(1\pm \eps)b$ if $(1-\eps)b\leq a\leq (1+\eps)b$.

\begin{lemma}\label{lem:key}
Given integers $N, T, k$ and hypergraph $H$ as in Theorem~\ref{thm:main}, let $\eps, \b, M_0, M$ be defined as in \eqref{eq:ebM} and $\{\alpha_m\}_{m=0}^{M}$ and $\{\gamma_m\}_{m=1}^{M}$ as in \eqref{eq:ag1}. Suppose $\{t_m\}_{m=0}^M, \{n_m\}_{m=0}^M$ are sequences of integers such that 
\[
t_m = \frac{T}{e^m}, \quad \text{and} \quad  n_m = (1\pm \eps)^m \frac{N}{e^m}, \quad \text{where } e=2.718....
\]
Then there exist collections of sets $\{C_{m}\}_{m=1}^M$ and $\{W_m\}_{m=1}^M$ such that for every $1\leq m \leq M$, the AKPSS algorithm produces hypergraph $H^m$ with vertex set $V_m$ and set $I_m \subseteq V(H)$ satisfying
\begin{enumerate}[i)]
    \item $H^{m} \in \bouquet$;
    \item $|I_{m}|\ge \frac{n_{m-1}}{et_{m-1}}(1-\eps)\gamma_m$;
    \item $|V_m| = n_m(1\pm \eps)$;
    \item $\D_1(H^m_i) \le (1+\eps)^{m}\binom{k-1}{k-i}\a_{m}^{k-i} t_m^{i-1}$ for $i=2,\ldots,k$;
    \item $\D_{i-1}(H^m_{i}) \leq (1+\eps)^{m}\frac{t_m}{(\log t_m)^{i+1}}$ for $i=3,\ldots,k$. 
\end{enumerate}
\end{lemma}

Theorem~\ref{thm:main} easily follows from Lemma~\ref{lem:key}.

\begin{proof}[Proof of Theorem~\ref{thm:main}]

Let $H=\{H_2,\ldots,H_k\}$ be the hypergraph on $N$ vertices given by Theorem~\ref{thm:main}. Applying the AKPSS algorithm with the parameter choices given by Lemma \ref{lem:key} gives a set 
\[
I=\bigcup_{m=1}^MI_m,
\]
where $\{I_m\}_{m=1}^M$ are disjoint and satisfy
\[
|I_m|\ge  \frac{n_{m-1}}{et_{m-1}}(1-\eps)\g_{m} \ge \frac{N}{e T}(1-\eps)^{m}\g_m.
\]

By Proposition \ref{cl:independence}, we have that $I$ is an independent set of $H$ and therefore $\a(H)\geq |I| = \sum_{m=1}^{M} | I_m |$. Thus a computation gives that
\begin{align*}
\a(H) &\ge \frac{N}{e T} \sum_{m=1}^{M} \g_m (1-\eps)^{m} \ge \frac{N}{e T}(1-\eps)^{M} \sum_{m=1}^{M} \g_m = \frac{N}{e T}(1-\eps)^{M}(\a_{M}-\a_0)\\
&  \ge \frac N{e T}(1- M\eps) \left( \left( M_0 + \frac{M}2 \right)^{1/(k-1)} - M_0^{1/(k-1)} \right) \quad \text{by \eqref{eq:amk1}}\\
&  = \frac N{e T}(1- M\eps) \left( \left( \log T + \frac{\log T}4 \right)^{1/(k-1)} - (\log T)^{1/(k-1)} \right)\\ 
& = \frac N{e T} \left( 1 - \frac12  \right) \left( \left(\frac54\right)^{1/(k-1)} - 1\right)  (\log T)^{1/(k-1)} \quad \text{as } \eps M = \frac12 \\
&=\Omega\left(\frac{N}{T}(\log T)^{1/(k-1)}\right). \qedhere
\end{align*}
\end{proof}

\section{Proof of Lemma~\ref{lem:key}}
For the computations to follow, we always consider $k$ as a constant. The proof is by induction on $m$. We will prove that there exist choices of $\{C_{m+1}\}_{m=0}^{M-1}$ and $\{W_{m+1}\}_{m=0}^M$ such that for every $0\leq m \leq M-1$ the following holds. If the graph $H^m$ satisfies conditions i), iii), iv) and v) of the lemma, then the graph $H^{m+1}$ satisfies conditions i), iii), iv) and v) and $I_{m+1}$ satisfies condition ii). Note that for $m=0$, the hypergraph $H^0=H$ is the one given by Theorem \ref{thm:main}. Clearly $H^0=\{H^0_2,\ldots,H^0_k\}$ satisfies condition i), iii), iv) and v) since
\begin{align*}
   & \Delta_1(H^0_i)\leq T^{i-1}(\log T)^{\frac{k-i}{k-1}}\leq \binom{k-1}{k-i}\a_0^{k-i}t_0^{i-1} \quad \text{for} \quad 2\leq i \leq k \quad \text{and} \\
   & \Delta_{i-1}(H^0_i) \leq \frac{T}{(\log T)^{i+1}}=\frac{t_0}{(\log t_0)^{i+1}} \quad \text{for} \quad 3\leq i \leq k.
\end{align*}

In order to simplify the notation, for the proof exposition we are going to denote $H^m$ by $H=\{H_2,\ldots,H_k\}$ and $H^{m+1}$ by $H^*=\{H_2^*,\ldots,H_k^*\}$. Let
\[
n = n_m = (1\pm \eps)^m \frac{N}{e^m}, \quad t= t_m =\frac{T}{e^m}, \quad \text{where  } 0\le m< \frac{\log T}2.
\]
Recall that  $(\log N)^3 \le T \le N^{1/4k}$. It follows that
\begin{align}\label{eq:tn}
t^{4k} = \left( \frac{T}{e^m} \right)^{4k} \le \frac{T^{4k}}{e^m} (1-\eps)^m \le \frac{N}{e^m} (1-\eps)^m \le n 
\end{align}
because $1-\eps=1 -1/\log T \ge 1/e$. 
As $t= T/ e^m$ and $0\le m < \log T/2$, we have 
\begin{align}\label{eq:tlogn}
t \ge \frac{T}{e^{(\log T)/2}} = \sqrt{T}
\ge (\log N)^{3/2} \ge (\log n)^{3/2},
\end{align}
and
\begin{align}\label{eq:Tt}
\frac{\log T}2 \le \log t \le \log T.
\end{align}
For the induction step of Lemma~\ref{lem:key}, we will simply write $\a := \a_{m}$ and $\g := \g_{m+1} = \a_{m+1} - \a_{m}$. 
By \eqref{eq:amk1}, \eqref{eq:gm1}, and \eqref{eq:Tt}, we have 
\begin{align}
\label{eq:ag}
1\le \a = \Theta \left( (\log t)^{\frac1{k-1} } \right), \quad 1 > \g = \Theta\left( (\log t)^{-\frac{k-2}{k-1} }\right), \quad \text{and}
\end{align}
\begin{align}\label{eq:aga}
(\a + \g)^{k-1} - \a^{k-1} = \a_{m+1}^{k-1} - \a_m^{k-1} = \b^m = \left(\frac1{1+\eps}\right)^m.
\end{align}

As described in the AKPSS algorithm, the first step of the algorithm is to apply the Almost Regular Lemma (Lemma \ref{lem:AR}) to $H:=H^m$ and obtain a superhypergraph $\tilde{H} \supseteq H$ on $V:=V_m=V(H)$ and a subset $B:=B_{m+1} \subseteq V$ satisfying conditions i)--iii) of Lemma \ref{lem:AR}. By \eqref{eq:ag} and the induction hypothesis, we have
\begin{align*}
|B| &\le k^2\left(1+\sum_{i=2}^k(i-1)\D_1(H_i)\right)^3 \\
&\le k^2\left(1+ (1+\eps)^m \left( \sum_{i=2}^k (i-1)\binom{k-1}{k-i}\a^{k-i}t^{i-1} \right) \right)^3 = O(t^{3(k-1)}).
\end{align*}
Recall that $N(B)= \bigcup_{x\in B} N(x)$. It follows that
\begin{align}\label{eq:NB}
    |N(B)|\leq \left(1+\sum_{i=2}^k(i-1)\D_1(H_i)\right)|B|= O(t^{4(k-1)}).
\end{align}

From now on we regard $\tilde{H}=\{\tilde{H}_2,\ldots,\tilde{H}_k\}$ as our underlying hypergraph, but similar in the description of the algorithm, for simplicity on notation, we still write it as $H=\{H_2,\ldots,H_k\}$.  
Thus, for all $x\in V\setminus B$, we have
\begin{align}\label{degrees}
\deg_{H_i}(x)=(1+\eps)^m\binom{k-1}{k-i}\a^{k-i}t^{i-1}
\end{align}
for every $2\leq i \leq k$. 

\subsection{Construction of \texorpdfstring{$C_{m+1}$}{Cm+1} and \texorpdfstring{$W_{m+1}$}{Wm+1}}
\label{sec:51}

We now describe how to construct the sets $C:=C_{m+1}$ and $W:=W_{m+1}$ in the AKPSS algorithm. The set $C$ is \emph{a random set of $V$, in which we independently select every vertex $x\in C$ with probability $p:=p_{m+1}=\g/ t$}. 

The set $W$ requires more preparation. Recall that $D$ is the set of all $x\in V$ such that $e\subseteq C$ for some $e\in H_x$. Let
\[
V' := V\setminus (B\cup C\cup D) \quad \text{and} \quad I':=C\setminus (B \cup D).
\] 
Let $H'$ be the hypergraph with $V(H')=V'$ and
\[
E(H') \text{ is the multiset } \{ e \cap V': e \in H[V' \cup C] \}.
\]
There are some differences from $H'$ and the output $H^*=H^{m+1}$ of the AKPSS algorithm. The first one is that instead of contracting the edges in the induced graph $H[V'\cup I']$ as in the algorithm, for the graph $H'$ we contract the edges in the larger set $V'\cup C$. As a consequence, some edges that contain elements of $D$ might survive in $H'$ (if the vertex is in $C\cap D$), while the same edge would be deleted in $H^*$. This modification allows us to obtain better control on the degree of the vertices in $H'$.

Another difference is that we allow $H'$ to contain multiple edges and edges that properly contain another edge -- such edges will be removed at the end but it is convenient to include them in $E(H')$ during the process. Moreover, $H'$ does not use the waste $W$ in the definition. The set $W$ is going to be defined such that $B\subseteq W$. So in a certain way, one can consider the hypergraph $H'$ to be an intermediate hypergraph from $H$ to $H^*$ (see Figure \ref{fig:fig3}).

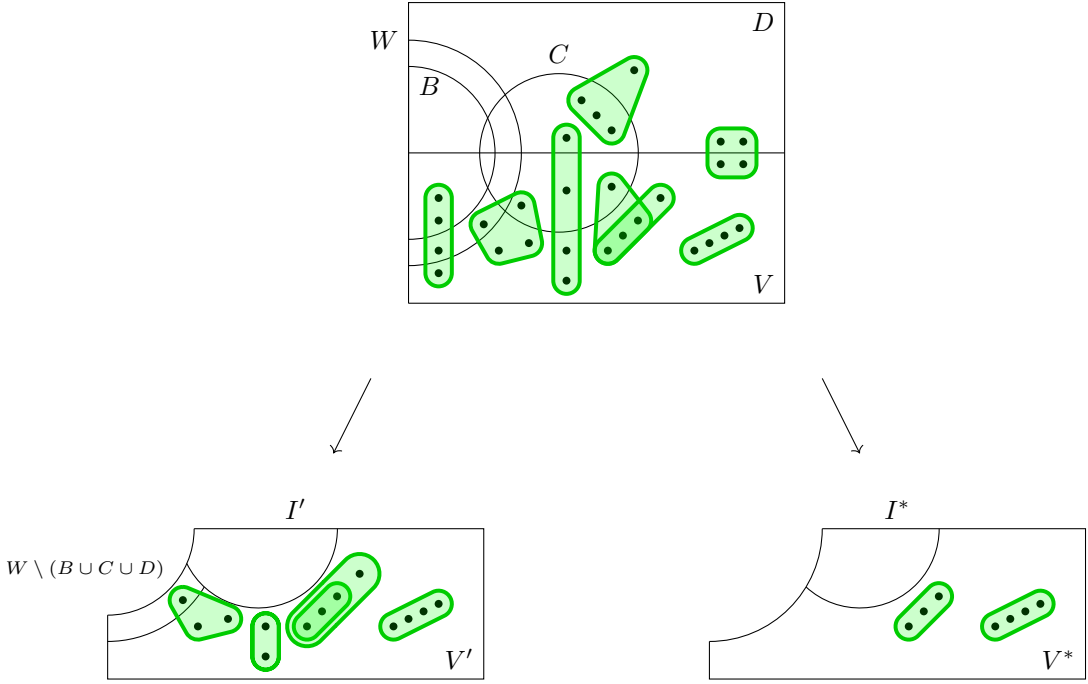
\begin{figure}[h]
	\begin{tikzpicture}[scale=1]
		
		\coordinate (A) at (0,0);
		\coordinate (B) at (0,4);
		\coordinate (C) at (5,4);
		\coordinate (D) at (5,0);
		\coordinate (M) at (0,2);
		\coordinate (N) at (5,2);
		
		\coordinate (x_1) at (4.45,2.15);
		\coordinate (y_1) at (4.15,2.15);
		\coordinate (z_1) at (4.15,1.85);
		\coordinate (w_1) at (4.45,1.85);
		
		\coordinate (x_2) at (4.4,1);
		\coordinate (y_2) at (4.2,0.9);
		\coordinate (z_2) at (4.0,0.8);
		\coordinate (w_2) at (3.8,0.7);
		
		\coordinate (x_3) at (3,3.1);
		\coordinate (y_3) at (2.3,2.7);
		\coordinate (z_3) at (2.5,2.5);
		\coordinate (w_3) at (2.7,2.3);
		
	    \coordinate (x_4) at (0.4,1.4);
		\coordinate (y_4) at (0.4,1.1);
		\coordinate (z_4) at (0.4,0.7);
		\coordinate (w_4) at (0.4,0.4);
		
		\coordinate (x_5) at (1.5,1.3);
		\coordinate (y_5) at (1,1.05);
		\coordinate (z_5) at (1.2,0.7);
		\coordinate (w_5) at (1.6,0.8);
		
		\coordinate (x_6) at (2.7,1.55);
		\coordinate (y_6) at (2.65,0.7);
		\coordinate (z_6) at (2.85,0.9);
		\coordinate (w_6) at (3.05,1.1);
		\coordinate (v_6) at (3.35,1.4);
		
		\coordinate (x_7) at (2.1,2.2);
		\coordinate (y_7) at (2.1,1.5);
		\coordinate (z_7) at (2.1,0.7);
		\coordinate (w_7) at (2.1,0.3);
	    
	    \draw (A)--(B)--(C)--(D)--(A);
	    \draw (M)--(N);
	    
	    \draw (2,2) circle (30pt);
	    \draw (0,3.15) -- (0,0.85) arc(-90:90:1.15) --cycle;
	    \draw (0,3.5) -- (0,0.5) arc(-90:90:1.5) --cycle;
	    
	    \fill (x_1) circle (1.5pt);
		\fill (y_1) circle (1.5pt);
		\fill (z_1) circle (1.5pt);
		\fill (w_1) circle (1.5pt);
		
		\fill (x_2) circle (1.5pt);
		\fill (y_2) circle (1.5pt);
		\fill (z_2) circle (1.5pt);
		\fill (w_2) circle (1.5pt);
		
		\fill (x_3) circle (1.5pt);
		\fill (y_3) circle (1.5pt);
		\fill (z_3) circle (1.5pt);
		\fill (w_3) circle (1.5pt);
		
		\fill (x_4) circle (1.5pt);
		\fill (y_4) circle (1.5pt);
		\fill (z_4) circle (1.5pt);
		\fill (w_4) circle (1.5pt);
		
		\fill (x_5) circle (1.5pt);
		\fill (y_5) circle (1.5pt);
		\fill (z_5) circle (1.5pt);
		\fill (w_5) circle (1.5pt);
		
		\fill (x_6) circle (1.5pt);
		\fill (y_6) circle (1.5pt);
		\fill (z_6) circle (1.5pt);
		\fill (w_6) circle (1.5pt);
		\fill (v_6) circle (1.5pt);
		
		\fill (x_7) circle (1.5pt);
		\fill (y_7) circle (1.5pt);
		\fill (z_7) circle (1.5pt);
		\fill (w_7) circle (1.5pt);
				
		\redge{(w_1)}{(z_1)}{(y_1)}{(x_1)}{5pt}{1.5pt}{green!80!black}{green,opacity=0.2};
		\redge{(w_2)}{(z_2)}{(x_2)}{(y_2)}{5pt}{1.5pt}{green!80!black}{green,opacity=0.2};
		\redge{(w_3)}{(z_3)}{(y_3)}{(x_3)}{5pt}{1.5pt}{green!80!black}{green,opacity=0.2};
		\redge{(w_4)}{(z_4)}{(y_4)}{(x_4)}{5pt}{1.5pt}{green!80!black}{green,opacity=0.2};
		\redge{(w_5)}{(z_5)}{(y_5)}{(x_5)}{5pt}{1.5pt}{green!80!black}{green,opacity=0.2};
	    \redge{(w_6)}{(z_6)}{(y_6)}{(x_6)}{5pt}{1.5pt}{green!80!black}{green,opacity=0.2};
	    \redge{(w_6)}{(v_6)}{(z_6)}{(y_6)}{5pt}{1.5pt}{green!80!black}{green,opacity=0.2};
	    \redge{(w_7)}{(z_7)}{(y_7)}{(x_7)}{5pt}{1.5pt}{green!80!black}{green,opacity=0.2};
	    
	    \node (q1) at (5,4) [below left, font=\small] {$D$};
	   \node (q2) at (5,0) [above left, font=\small] {$V$};
	   \node (q3) at (0,3.5) [left, font=\small] {$W$};
	   \node (q4) at (2,2cm +30pt) [above, font=\small] {$C$};
	   \node (q5) at (0,3.15) [below right, font=\small] {$B$};
	    
	    \coordinate (P1) at (-4,-3);
	    \coordinate (Q1) at (-2,-3);
	    
		\tkzInterCC[R](P1,1.15cm)(Q1,30pt) \tkzGetPoints{O1}{C1};
		
		\tkzInterCC[R](P1,1.5cm)(Q1,30pt) \tkzGetPoints{X1}{Y1};
		
		\coordinate (A1) at (-4,-5);
		\coordinate (B1) at (-4,-4.15);
		\coordinate (D1) at (30pt-2cm,-3);
		\coordinate (E1) at (1,-3);
		\coordinate (F1) at (1,-5);
		\coordinate (G1) at (-2.85,-3);
		\coordinate (H1) at (-2.5,-3);
		
		\tkzDrawArc[R with nodes, color=black](P1,1.15cm)(B1,G1);
		\tkzDrawArc[R with nodes, color=black](P1,1.5cm)(B1,Y1);
		\tkzDrawArc[R with nodes, color=black](Q1,30pt)(C1,D1);
		\draw (G1)--(E1)--(F1)--(A1)--(B1);
	    
	    \coordinate (x'_2) at (0.4,-4);
		\coordinate (y'_2) at (0.2,-4.1);
		\coordinate (z'_2) at (0,-4.2);
		\coordinate (w'_2) at (-0.2,-4.3);
		
		\coordinate (y'_5) at (-3,-3.95);
		\coordinate (z'_5) at (-2.8,-4.3);
		\coordinate (w'_5) at (-2.4,-4.2);
		
		\coordinate (y'_6) at (-1.35,-4.3);
		\coordinate (z'_6) at (-1.15,-4.1);
		\coordinate (w'_6) at (-0.95,-3.9);
		\coordinate (v'_6) at (-0.65,-3.6);
		
		\coordinate (z'_7) at (-1.9,-4.3);
		\coordinate (w'_7) at (-1.9,-4.7);
		
		\fill (x'_2) circle (1.5pt);
		\fill (y'_2) circle (1.5pt);
		\fill (z'_2) circle (1.5pt);
		\fill (w'_2) circle (1.5pt);
		
		\fill (y'_5) circle (1.5pt);
	    \fill (z'_5) circle (1.5pt);
		\fill (w'_5) circle (1.5pt);
		
		\fill (y'_6) circle (1.5pt);
		\fill (z'_6) circle (1.5pt);
		\fill (w'_6) circle (1.5pt);
		\fill (v'_6) circle (1.5pt);
		
		\fill (z'_7) circle (1.5pt);
		\fill (w'_7) circle (1.5pt);
	   
	   \redge{(w'_2)}{(z'_2)}{(x'_2)}{(y'_2)}{5pt}{1.5pt}{green!80!black}{green,opacity=0.2};
	   \redge{(w'_5)}{(z'_5)}{(y'_5)}{(y'_5)}{5pt}{1.5pt}{green!80!black}{green,opacity=0.2};
	   \redge{(w'_6)}{(z'_6)}{(y'_6)}{(w'_6)}{5pt}{1.5pt}{green!80!black}{green,opacity=0.2};
	   \redge{(w'_6)}{(v'_6)}{(z'_6)}{(y'_6)}{7.5pt}{1.5pt}{green!80!black}{green,opacity=0.2};
	   \redge{(w'_7)}{(z'_7)}{(w'_7)}{(z'_7)}{5pt}{1.5pt}{green!80!black}{green,opacity=0.2};
	   
	   \node (q6) at (1,-5) [above left, font=\small] {$V'$};
	   \node (q7) at (-4.3,-3.8) [above, font=\tiny] {$W\setminus (B\cup C\cup D)$};
	   \node (q8) at (-1.5,-3) [above, font=\small] {$I'$};

	   \draw[->] (-0.5,-1)--(-1,-2);
	   
	   
	    \coordinate (P2) at (4,-3);
	    \coordinate (Q2) at (6,-3);
	    
		\tkzInterCC[R](P2,1.5cm)(Q2,30pt) \tkzGetPoints{O2}{C2};
		
		\coordinate (A2) at (4,-5);
		\coordinate (B2) at (4,-4.5);
		\coordinate (D2) at (30pt+6cm,-3);
		\coordinate (E2) at (9,-3);
		\coordinate (F2) at (9,-5);
		\coordinate (G2) at (5.5,-3);
		
		\tkzDrawArc[R with nodes, color=black](P2,1.5cm)(B2,G2);
		\tkzDrawArc[R with nodes, color=black](Q2,30pt)(C2,D2);
		\draw (G2)--(E2)--(F2)--(A2)--(B2);
		
		\coordinate (x''_2) at (8.4,-4);
		\coordinate (y''_2) at (8.2,-4.1);
		\coordinate (z''_2) at (8,-4.2);
		\coordinate (w''_2) at (7.8,-4.3);

		\coordinate (y''_6) at (6.65,-4.3);
		\coordinate (z''_6) at (6.85,-4.1);
		\coordinate (w''_6) at (7.05,-3.9);
		
		\fill (x''_2) circle (1.5pt);
		\fill (y''_2) circle (1.5pt);
		\fill (z''_2) circle (1.5pt);
		\fill (w''_2) circle (1.5pt);

		\fill (y''_6) circle (1.5pt);
		\fill (z''_6) circle (1.5pt);
		\fill (w''_6) circle (1.5pt);
	   
	   \redge{(w''_2)}{(z''_2)}{(x''_2)}{(y''_2)}{5pt}{1.5pt}{green!80!black}{green,opacity=0.2};
	   \redge{(w''_6)}{(z''_6)}{(y''_6)}{(w''_6)}{5pt}{1.5pt}{green!80!black}{green,opacity=0.2};
	   
	    \node (q9) at (9,-5) [above left, font=\small] {$V^*$};
	    \node (q10) at (6.5,-3) [above, font=\small] {$I^*$};
	   
	   \draw[->] (5.5,-1)--(6,-2);
	   
	\end{tikzpicture}
	\caption{Hypergraphs $H'$ and $H^*$.}
	\label{fig:fig3}
\end{figure}

We define the degree $i$ to $j$ of a vertex $x\in V$
\[
\deg_{i\to j}(x)=\left|\left\{ e\in (H_i)_x:\: e\subseteq V'\cup C,\, |e\cap V'|=j-1 \right\}\right|.
\]
Note that for $x\in V'$, the degree $\deg_{i\to j}(x)$ is exactly the number of edges incident to $x$ of size $i$ that after the operation become of size $j$ in $H'$. We will prove in a later step of the proof that the expected value of $\deg_{i\to j}(x)$ for $x\in V\setminus N(B)$ conditioned to $x\not \in C$ is
\[
\mu_{i\to j}(x):=\mE(\deg_{i\to j}(x)\mid x\notin C)=\left(1\pm \frac{\eps}{4}\right)\binom{i-1}{j-1}\deg_{H_i}(x)p^{i-j}e^{-j+1}.
\]
For $2\leq j\leq i \leq k$, define the sets $Z_{i,j}$ by
\[
Z_{i,j}=\left\{ x\in V\setminus (N(B)\cup C):\: \deg_{i\to j}(x)> \left(1+\frac{\eps}{4}\right)\mu_{i\to j}(x)\right\},
\]
i.e., the vertices in $V\setminus (N(B)\cup C)$ that the actual $i$ to $j$ degree is substantially larger than the expected value. Let
\[
Z=\bigcup_{2\leq j \leq i \leq k}Z_{i,j}.
\]
We define the waste set $W$ as 
\[
W=N(B)\cup Z.
\]

\subsection{Proof scheme}

Since $H\in \bouquet$, by Lemma \ref{lem:structural} we obtain that $H^* \in \bouquet$. The proof of condition ii) to v) are more extensive. The following is a scheme of the main steps of our proof. Throughout the paper, \emph{w.h.p.} means \emph{with probability $1-o(1)$}.
\begin{description}
\item[Step 1] Show that $\mathbb{P}(x\not\in D) \approx 1/e$ for all $x\in V\setminus B$.

\item[Step 2] Show that w.h.p. $|V'|\approx n/e$ and $ |I'| \approx \frac{n\g}{et}$.

\item[Step 3] Show that $\mE(\deg_{i\to j}(x)\mid x\notin C) \approx \binom{i-1}{j-1} \deg_{H_i}(x) p^{i-j} e^{1-j}$ for $x\in V\setminus N(B)$ and $2\le j\le i\le k$.

\item[Step 4] Show that with probability at least $3/4$ we have $|Z|\leq \frac{\eps}{3}\frac{\g n}{et}$ and $\deg_{H'_i}(x)\le (1+\eps)^{m+1}\binom{k-1}{k-i}\a_{m+1}^{k-i} (t/e)^{i-1}$ for $2\leq i\leq k$ and $x\in V'\setminus (N(B)\cup Z)$.

\item[Step 5] Show that w.h.p $\deg_{H'_i}(S)\le (1+\eps)^{m+1} (t/e) / (\log (t/e))^{i+1}$ for all $(i-1)$-sets $S\subseteq V'$.
\end{description}
We complete the proof by letting $V^* = V_{m+1}= V' \setminus W$, $I =I_{m+1}= I'\setminus W$, and removing the appropriate edges from $H'$ to obtain the desired hypergraph $H^*$ (all multiple edges, edges that properly contain another edge from $H'$ and edges that intersect $D$ and the waste set $W$ before or after the contraction). 

\vspace{0.1cm}

\subsection{Step 1}

The goal of this step is to prove that $\mP(x\not\in D)\approx e^{-1}$ for $x\in V\setminus B$. Given $x\in V\setminus B$, we first show that $\mE\left( |H_x[C]| \right)$, the expected number of edges of $H_x$ that lie in $C$, is exactly one.

\begin{proposition}\label{clm:degC}
For every $x\in V\setminus B$, we have $\mE\left( |H_x[C]| \right) = 1$.
\end{proposition}
\begin{proof}
For $2\leq i \leq k$, let $H_{i,x} := (H_i)_x = \{ e\setminus \{x\}: e\in H_i\}$. 
Let $\mu_i = \mE\left( |H_{i, x}[C] |\right)$ be the expected number of edges of $H_{i, x}$ that lie in $C$. Then $\mu_i = p^{i-1} |H_{i, x}|$.
Let $\mu = \sum_{i=2}^k \mu_i = \mE\left( |H_x[C]| \right) $. We have 
\begin{align*}
\mu &= \sum_{i=2}^k p^{i-1} |H_{i, x}| =  \sum_{i=2}^k \left( \frac{\g}{t} \right)^{i-1} \deg_{H_i}(x)  \\
&= (1+\eps)^m \sum_{i=2}^k \binom{k-1}{k-i}\a^{k-i} \g^{i-1} \quad \text{because $x\not\in B$}\\ 
&=(1+\eps)^m\left((\a+\g)^{k-1}-\a^{k-1}\right) \\
&= (1+\eps)^m \beta^m = 1 \quad \text{by \eqref{eq:aga}}. \qedhere
\end{align*}
\end{proof}

Next we prove the following general claim which we will also use in later steps. Note that Proposition~\ref{clm:tH} below immediately implies that $\mathbb{P}(x\not\in D) = e^{-1+ o(1)}$ by letting $\tilde{H_x}= H_x$. This is due to the fact that $x\notin D$ is the same as $e\not\subseteq C$ for every $e\in H_x$. 

\begin{proposition}\label{clm:tH}
Let $x\in V\setminus B$ and $\tilde{H_x}$ be a subgraph of $H_x$ with $|\tilde{H_x}|= |H_x| - O\left(\frac{t}{(\log t)^{2}}\right)$. Then 
\begin{align}\label{eq:51}
e^{- 1 - 1/t } \le  \prod_{i=2}^k (1 - p^{i-1})^{| \tH_{i, x}|} \le \mP \left( \bigwedge_{e\in \tilde{H}_x} \left\{ e\not\subseteq C\right\} \right) = e^{-1 + O(1/(\log t)^{2})}.
\end{align}
\end{proposition}

\begin{proof}
We first compute $\mE( |\tH_x[C] |) $.
Suppose that $\tH_x = (\tH_{2, x}, \dots, \tH_{k, x})$, where $\tH_{i,x}$ is  the $(i-1)$-uniform subhypergraph of $\tH_x$.
Let $\mu_i = \mE\left( |\tH_{i, x}[C] | \right)$ and $\mu = \sum_{i=2}^k \mu_i = \mE\left( |\tH_x[C]| \right) $. 
Since $|\tH_{i, x}| = \deg_{H_i}(x) - O(t/(\log t)^{2})$, similar calculations as in Proposition~\ref{clm:degC} give that $\mu = 1 - O(1/(\log t)^{2})$. 

Since $H \in \bouquet$, for $i\neq j$, every pair of edges $e\in \tH_{i,x}$ and $f\in \tH_{j,x}$ are disjoint. Thus
\begin{align}\label{eq:Ce}
\mP \left( \bigwedge_{e\in \tH_x} \left\{ e\not\subseteq C \right\} \right) = \prod_{i=2}^k \mP\left( \bigwedge_{e\in \tH_{i, x}} \left\{ e\not\subseteq C \right\} \right). 
\end{align}
We use Janson's inequality to bound $\mP\left( \bigwedge_{e\in \tH_{i, x}} \left\{ e\not\subseteq C \right\} \right)$ for each $2\leq i\leq k$. 
Janson's inequality (see \cite[Theorem 8.1.1]{MR2437651}) gives that
\begin{align}\label{eq:Janson}
(1 - p^{i-1})^{|\tH_{i, x}|} \le \mP\left(\bigwedge_{e\in \tH_{i, x}} \left\{ e\not\subseteq C \right\} \right) \le e^{-\mu_i + \Lambda_i /2},
\end{align}
where 
\[
\Lambda_i = \sum_{\substack{e,f\in \tH_{i,x}\\e\cap f\neq \emptyset}} \mP(\left\{e,f \subseteq C\right\}). 
\]
Note that $\Lambda_2=0$ since every member of $H_{2,x}$ is a singleton. For $3\leq i \leq k$, since $H\in \bouquet$, we obtain that if $e\cap f\neq \emptyset$, then $|e\cap f|=i-2$. Then there are at most $ \deg_{H_i}(x) \binom{i-1}{i-2}\D_{i-1}(H_i)$ pairs $e, f\in \tH_{i, x}$ such that $e\cap f \ne \emptyset$. Thus
\begin{align*}
\Lambda_i &\le p^i \deg_{H_i}(x)  \binom{i-1}{i-2} \D_{i-1}(H_i) \le (i-1)\left( \frac{\g}{t} \right)^i (1+ \eps)^{2m} \binom{k-1}{k-i} \frac{\a^{k-i} t^i}{(\log t)^{i+1}}\\
&\le  (i-1)\binom{k-1}{i-1} e^{2m\eps} \frac{ \g^i \a^{k-i} }{ (\log t)^{i+1} } \\
&= O\left((\log t)^{\frac{(k-i)-i(k-2)}{k-1}-(i+1)}\right)
\quad \text{by \eqref{eq:ag} and }\eps m \le \frac12.
\end{align*}
Since $k-i < i(k-2)$, it follows that $\Lambda_i = O\left( \frac1{ (\log t)^{i+1} } \right)$ for $i\geq 3$. The second inequality of \eqref{eq:Janson} thus gives $\mP(\bigwedge_{e\in \tH_{i, x}} \left\{ e\not\subseteq C \right\} ) = e^{-\mu_i + O((1/\log t)^{i+1})} = e^{-\mu_i + O((1/\log t)^{4})}$ for $i\ge 3$ and $\mP(\bigwedge_{e\in \tH_{2, x}} \left\{ e\not\subseteq C \right\} ) = e^{-\mu_2 + \Lambda_2/2} = e^{-\mu_2}$.
Applying \eqref{eq:Ce} and $\mu =  \sum_{i=2}^k \mu_i = 1 - O(1/(\log t)^{2})$, we derive the desired upper bound
\[
\mP \left( \bigwedge_{e\in \tilde{H}_x} \left\{ e\not\subseteq C \right\} \right) = \prod_{i=2}^k \mP\left( \bigwedge_{e\in \tH_{i, x}} \left\{ e\not\subseteq C \right\} \right) = e^{-\mu + O(1/(\log t)^{4})} = e^{-1 + O(1/(\log t)^{2})}. 
\]

On the other hand, since $1- y\ge e^{-y (1+y)}$ for $0\le y\le 1/2$, $p^{i-1}<p<\frac{1}{t}$ and $ |\tH_{i, x}| \le  |H_{i,x}|$, the first term in \eqref{eq:Janson} can be bounded below as 
\begin{align*}
(1 - p^{i-1})^{| \tH_{i, x}|} &\ge e^{ - \left( 1 + p^{i-1} \right) p^{i-1}  |\tH_{i,x}| } \ge e^{ - \left( 1 + 1/t \right) p^{i-1}  |H_{i,x}| }. 
\end{align*} 
Proposition~\ref{clm:degC} gives that $\sum_{i=2}^k p^{i-1}  |H_{i,x}|  = 1$. Thus, by  \eqref{eq:Ce}, we obtain the desired lower bound
\[
\mP \left( \bigwedge_{e\in \tH_x} \left\{ e\not\subseteq C \right\} \right)  \ge \prod_{i=2}^k (1 - p^{i-1})^{| \tH_{i, x}|} \ge e^{-1 - 1/t}. \qedhere
\]
\end{proof}

\vspace{0.1cm}

\subsection{Step 2}

Now we are going to estimate the sizes of $|V'|$ and $|I'|$. For $x\in V\setminus B$, Step~1 gives that $\mathbb{P}(x\not\in D) = e^{-1+ O(1/(\log t)^{2})}$. Recall that the event $x\not\in C$ concerns $x$ itself while the event 
$x\not\in D$ concerns the neighbors of $x$. Hence the events $x\not\in C$ and $x\not\in D$ are independent. Together with $V'= V\setminus (B\cup C\cup D)$ and the fact that $|B|=O(t^{3(k-1)})$, this gives that
\begin{align}
\mE(|V'|)&= \sum_{x\in V\setminus B} \mP(x\not\in C)\mP(x\not\in D) =  (1- p) \sum_{x\in V\setminus B} e^{-1+ O(1/(\log t)^{2})}  \nonumber \\
&= \left(1 -\g/t \right) \left(n - O(t^{3(k-1)}) \right) e^{-1+ O(1/(\log t)^{2})} \nonumber \\
&= n e^{-1+O(1/(\log t)^{2})} \quad \text{as} \quad \g/t = O(1/t ) \ \text{and} \ t^{3k-3}/n = O(1/t).
\label{eq:V'}
\end{align}
Analogously we can derive that $I' = C\setminus (B\cup D)$ is an independent set with
\[
\mE(|I'|) = \sum_{x\in V\setminus B} \mP(x\in C)\mP(x\not\in D) =\frac{\g n}{t} e^{-1+ O(1/(\log t)^{2})}.
\]
Next we apply Lemma~\ref{lem:AI} to obtain the concentrations of $|V'|$ and $|I'|$. Note that two events $\{x \in V'\}$ and $\{y\in V'\}$ are independent if they do not share a neighbor. Thus, for every set $S\subseteq V\setminus B$, the event $x \in V'$ is mutually independent of the collection of events $\{y \in V': y \in S\}$ if $N^2(x) \cap S=\emptyset$. Since $|N(x)|\leq 1+\sum_{i=2}^k(i-1)\D_1(H_i)=O( t^{k-1} )$, it follows that $|N^2(x)| = O(t^{2(k-1)})$. Thus, the event $x\in V'$ is independent of all but at at most $c_k t^{2(k-1)}$ events $\{y\in V'\}
$ for some constant $c_k$ depending on $k$.
Applying Lemma~\ref{lem:AI} for $p=\mP(x\notin C)\mP(x\notin D)$, $s=c_k t^{2(k-1)}$ and $n=|V\setminus B|$, we obtain 
\[
\mP\left( | |V'| - \mE(|V'|) | > \frac{\eps}4  \mE(|V'|) \right)\le 2(c_k t^{2(k-1)}+1) \exp \left( - \frac{\eps^2n (e^{-1+o(1)})}{ 3(c_k t^{2(k-1)} + 1)} \right)= o(1).
\]
Thus  w.h.p. $|V'| = (1 \pm \eps/4) \mE(|V'|)$. Since
$\eps =1/{\log T} \ge 1/ (2\log t) \gg 1/ (\log t)^2$,
we have 
\[
\left(1 \pm \frac{\eps}4\right) e^{O(1/(\log t)^{2})} = 1 \pm \frac{\eps}3.
\]

Therefore, it follows from \eqref{eq:V'} that w.h.p. 
\begin{align} \label{eq:V'2}
|V'| = \left(1 \pm \frac{\eps}3\right) \frac{n}{e}.
\end{align}
Analogously we can derive that w.h.p. 
\begin{align} \label{eq:I'}
| I' |\ge \left(1 - \frac{\eps}3\right) \frac{\g n}{et}.
\end{align} 

\vspace{0.05cm}

\subsection{Step 3}

For $x\in V$ and $2\le j\le i\le k$, we have defined 
\[
\deg_{i\to j}(x)= \left| \{e\in H_{i, x}: e\subseteq V'\cup C,\, |e\cap V'| = j-1\}\right|. 
\]
Our goal is to show that, for every $x\in V \setminus N(B)$,
\begin{align}\label{eq:dij1}
\mE(\deg_{i\to j}(x)\mid x\notin C)  =\left(1\pm \frac{\eps}4 \right)\binom{i-1}{j-1} \deg_{H_i}(x)  p^{i- j} e^{-j +1}. 
\end{align}
By \eqref{degrees} and the definition $p= \g/t$, this means that
\begin{align}
\mE(\deg_{i\to j}(x)\mid x\notin C) = \left(1\pm \frac{\eps}4 \right) (1+\eps)^m\binom{k-1}{k-i}\binom{i-1}{j-1}\a^{k-i}\g^{i-j}\left(\frac{t}{e}\right)^{j-1} \label{eq:dij}.
\end{align}

Fix $x\in V\setminus N(B)$ (thus $x \not \in B$ and $f\cap B= \emptyset$ for all $f\in H_x$) and consider the random space where we condition $x\notin C$.
For every edge $e\in H_{i, x}$, we define the event 
\begin{align}\label{eq:Ye}
Y_e := \left\{|e \cap V'| = j-1 \ \text{and} \ | e \cap C|= i- j\ \middle| \ x\notin C \right\}.
\end{align}
Note that 
\[
\mE(\deg_{i\to j}(x)\mid x\notin C) = \sum_{e\in H_{i,x}} \mP(Y_e)
\]
for $x\in V\setminus N(B)$. To prove \eqref{eq:dij1}, it suffices to prove that
\begin{align} \label{eq:PYe}
\mP(Y_e) = \left(1\pm \frac{\eps}4 \right)\binom{i-1}{j-1} p^{i- j} e^{-j +1}.
\end{align}
for every $e\in H_{i, x}$. 

Fix $e\in H_{i,x}$. For every $R \in \binom{e}{j-1}$, we 
define the event $E_R$ as
\begin{align*}
    E_R=\left\{ R \cap (C\cup D)=\emptyset  \text{  and  } (e\setminus R) \subseteq C \ \middle| \ x \not \in C\right\}.
\end{align*}
Since $x\notin N(B)$, we have $e\cap B=\emptyset$. Consequently, if $R\cap (C\cup D)=\emptyset$, then $R\subseteq V'$. 
Hence,
\[
\mP(Y_e) = \mP\left( \bigvee_{R\in \binom{e }{j-1}} E_R \right) = \sum_{R\in \binom{e}{j-1}} \mP( E_R ),
\]
because for $R\ne R'\in \binom{e}{j-1}$, the two events $E_R$ and $E_{R'}$ are disjoint (in particular, $E_R$ implies that $(R\setminus R')\cap C=\emptyset$ while $E_{R'}$ implies that $R\setminus R'\subseteq C$).
To prove \eqref{eq:PYe}, it suffices to show that for every $R \in \binom{e}{j-1}$,
 \begin{align}\label{eq:PET}
 \mP(E_R) =\left(1\pm \frac{\eps}4 \right) p^{i- j} e^{-j +1}. 
\end{align}

In the rest of this step we prove \eqref{eq:PET}. For convenience, we write 
\[
\overline{e} = e\cup \{x\}.
\]
Fix $R \in \binom{e}{j-1}$ and let $S=  e\setminus R$. 
Using the fact that $\mP(X\cap Y)= \mP(X| Y) \mP(Y)$ for any events $X, Y$ and the fact that two events $P\subseteq C$ and $Q\cap C= \emptyset$ are independent for any disjoint sets $P, Q$, we observe that 
\begin{align*}
    \mP(E_R)&= \mP\left(R \cap (C\cup D)=\emptyset \mid S \subseteq C,  x \not \in C\right) \, \mP(S \subseteq C \mid  x \not \in C) \\
    		&=  \mP\left(R \cap D=\emptyset \mid R\cap C=\emptyset, S \subseteq C,  x \not \in C\right) \, \mP\left( R\cap C=\emptyset \mid S \subseteq C,  x \not \in C \right) \, \mP(S \subseteq C) \\	
		&= \mP\left(R \cap D= \emptyset \mid F_R\right) \, \mP(R\cap C=\emptyset ) \, \mP(S \subseteq C),
\end{align*}
where 
\[
F_R=\{(R\cup \{x\})\cap C = \emptyset \text{ and } S\subseteq C\}.
\]
Therefore, We have 
\begin{align}\label{eq:ET'}
    \mP(E_R)&=\mP\left(R\cap D=\emptyset \ \middle| \ F_R\right) (1-p)^{j-1} p^{i-j}.
\end{align}

\smallskip

We first find the desired lower bound for $\mP(E_R)$ by bounding $ \mP(R\cap D=\emptyset \ | \ F_R)$ from below. Note that $R\cap D=\emptyset$ means that, for every $y \in R$ and every $f\in H_y$, we have $f\not\subseteq C$. 
Furthermore, since $\{(f\setminus \overline e)\not\subseteq C\}$ implies that $\{f\not\subseteq C\}$, we have
\begin{align*}
    \mP\left(R\cap D=\emptyset \ \middle| \ F_R\right) = \mP\left(\bigwedge_{y \in R}\bigwedge_{f\in H_{y}}\left\{f\not \subseteq C \right\}\,\middle| \, F_R\right) \ge \mP\left(\bigwedge_{y \in R} \bigwedge_{f \in H_{y}} \left\{(f\setminus \overline{e})\not \subseteq C \right\}\,\middle| \, F_R\right).
\end{align*}
Let $E_{f,y}$ be the event $\{(f\setminus \overline{e}) \not \subseteq C\}$. Since $E_{f,y}$ concerns only the elements of $V\setminus  \overline{e}$, it is independent of $F_R$. Furthermore, let
\begin{align*}
    \cA_{f,y}=\left\{A\subseteq (V\setminus \overline{e}):\:  (f\setminus \overline{e}) \not \subseteq A \right\}.
\end{align*}
be the family of all possible choices of $C$ in $V\setminus\overline{e}$ that make the event $E_{f,y}$ true. It is clear that $\cA_{f,y}$ is closed under subsets, i.e., $\cA_{f,y}$ is a decreasing family.
Thus, by the FKG inequality (see \cite[Theorem 6.3.2]{MR2437651}) we obtain that
\begin{align*}
   \mP\left(R\cap D=\emptyset \ \middle| \ F_R\right) \ge \mP\left(\bigwedge_{y \in R} \bigwedge_{f \in H_{y}} E_{f,y} \right)\geq \prod_{y \in R}\prod_{f\in H_{y}} \mP\left( E_{f,y} \right).
\end{align*}
Since $H \in \bouquet$, if $f\in H_{y}$ for some $y\in R$, then either $\overline{e}\cap f =\emptyset$ or $i\geq 3$, $|\overline{e}\cap f|=i-2$ and $|f|=i-1$. 
In the former case, we have $\mP(E_{f,y})=\mP((f\setminus\overline{e})\not\subseteq C)=\mP(f\not\subseteq C)=1-p^{|f|}$. In the latter case, note that all such $f$'s are in the connected component of $H_y$ that contains $\overline{e}\setminus \{y\}$. Thus, by Proposition~\ref{clm:BouStr}, there are at most $\Gamma_i:= \max\{i, \D_{i-1}(H_i) \}$ such $f$'s. Furthermore, since $|f\setminus \overline{e}|=1$, we have $\mP(E_{f,y})= 1-p$. Therefore,
\begin{align*}
    \mP(R\cap D=\emptyset \ | \ F_R)&\geq (1-p)^{\Gamma_i } \prod_{y\in R}\prod_{\ell=2}^k(1-p^{\ell-1})^{|H_{\ell,y}|} \ge  (1-p)^{\Gamma_i } e^{-(1+ 1/t)(j-1)}
\end{align*}    
because $\prod_{\ell=2}^k(1-p^{\ell-1})^{|H_{\ell,y}|} \ge e^{-(1 + 1/t) }$ from \eqref{eq:51} on Proposition \ref{clm:tH}. By \eqref{eq:ET'}, it follows that 
\[    
    \mP(E_R) \geq p^{i-j}(1-p)^{\Gamma_i+j-1} e^{-(1+ 1/t)(j-1)}.
\]
Since $1- p \ge e^{-p(1+p)}$ and $\Gamma_i + j- 1 \le (1+\eps)^m t / (\log t)^{i+1} + j-1 \le 2t / (\log t)^{i+1}$, we have
\[
(1-p)^{\D_{i-1}(H_i)+j-1} \ge \exp\left( -p(1+p)\frac{2t}{(\log t)^{i+1}} \right) \ge  \exp\left( - \frac{3}{(\log t)^{i+1}} \right).
\]
We thus derive the lower bound in \eqref{eq:PET}:
\[
 \mP(E_R) \ge p^{i-j} e^{-(1+ 1/t)(j-1) - 3/(\log t)^{i+1}}
 \ge \left(1-\frac{\eps}4 \right) p^{i-j}e^{-j+1}.
\]

Next, we find the desired upper bound for $\mP(E_R)$ by bounding $\mP(R\cap D=\emptyset\ | \  F_R)$ from above. For every $y\in e$, let 
\[
\tH_y  = \{f\in H_y: f\cap N(\overline{e}\setminus \{y\}) = \emptyset \}
\] 
be the family of edges of $H_y$ that do not intersect any $g \in H_z$ for every $z\in \overline{e}\setminus \{y\}$ (including $x$). The set $\tH_y$ is disjoint from $\tH_z$ for any $z\in \overline{e}\setminus \{y\}$ and is also disjoint from $\overline{e}$. Therefore, the events $\{f \not \subseteq C: f\in \tH_y \}$ for $y\in e$ are mutually independent and also independent from $F_R$. Hence,
\begin{align}
    \mP(R\cap D=\emptyset \, | \, F_R)
    &\le \mP\left(\bigwedge_{y \in R}\bigwedge_{f\in \tH_{y}}\{f\not \subseteq C\} \,\middle| \, F_R\right)
    = \prod_{y\in R}\mP\left(\bigwedge_{f\in \tH_{y}}\{f\not\subseteq C\}\right). \label{eq:E'T}
 \end{align}

For every $y \in e$, we will apply Proposition \ref{clm:tH} to bound $\mP\left(\bigwedge_{f\in \tH_{y}}\{f\not\subseteq C\}\right)$. For this purpose, we need to show that $|\tH_y|$ is close to $|H_y|$.

\begin{proposition}\label{cl:vertex_dep}
$|H_y\setminus \tH_y|=O(t/(\log t)^{4})$. 
\end{proposition}

\begin{proof}
For $y'\in \overline{e}\setminus \{y\}$, let $H^{y'}_y = \{f \in H_y:\: f \cap N(y')\ne \emptyset\}$ be the set of edges in $H_y$ that share vertices with $N(y')$. Since $N(\overline{e} \setminus\{y\})=\bigcup_{y'\in \overline{e}\setminus \{y\}} N(y')$, we have $|H_y\setminus \tH_y| \leq \sum_{y'\in \overline{e}\setminus \{y\}} |H_y^{y'}|$. Thus it suffices to show that $|H_y^{y'}|= O(t/(\log t)^{4})$ for any fixed $y'\in \overline{e}\setminus \{y\}$.


\begin{claim}\label{cl:numcomponent}
The set $H_y^{y'}$ has at most $2k-1$ connected components.
\end{claim}

\begin{proof}
We start by partitioning $H_y^{y'}$ into two sets
\[
F_1 = \{ f \in H_y^{y'}: f \cap (\overline{e}\setminus\{y\}) =\emptyset \} \quad \text{and} \quad F_2 = H_y^{y'} \setminus F_1.
\]
Since every element in $F_2$ intersects with $\overline{e}\setminus\{y\}$, we obtain that $F_2$ is connected and therefore it is in a component of $H_y^{y'}$. It suffices to show that $F_1$ has at most $2k-2$ connected components. 

Suppose to the contrary, that $F_1$ has $2k-1$ connected components $C_1,\ldots, C_{2k-1}$. For every $1\leq \ell \leq 2k-1$, consider an edge $f_{\ell}\in C_{\ell}$ (thus $f_1, \dots, f_{2k-1}$ are pairwise disjoint). By the definition of $F_1$, for every $j$, there exists $g_{\ell}\in H_{y'}$ such that $f_{\ell}\cap g_{\ell}\neq \emptyset$. We claim that $y\notin g_{\ell}$ for every $j \in [2k-1]$. 
Suppose, for instance, we have $y\in g_1$. 
For convenience, write $\overline{f_1}:= f_1\cup \{y\}$ and $\overline{g_1}:= g_1\cup \{y'\}$.
Then $|\overline{f_{1}}\cap \overline{g_{1}}|\geq 2$ and $|\overline{g_1}\cap \overline{e}|\geq 2$. By Property i) and ii) of $\bouquet$, this implies that $\overline{f_1}, \overline{g_1} \in H_i$, $i\geq 3$ and $|\overline{f_1}\cap \overline{g_1}|=|\overline{g_1}\cap \overline{e}|=i-1$.
Since $f_1\cap \overline{e}= \emptyset$ (by the definition of $F_1$), it follows that $i-2 \le |\overline{f_1}\cap \overline{e}| =1$. Hence, $i=3$ and three edges $\overline{f_1}, \overline{g_1}$ and $\overline{e}$ contradict Property v) of $\bouquet$.

We form a bipartite graph $G$ with vertices $f_1, \dots, f_{2k-1}$ and $g_1, \dots, g_{2k-1}$ such that $f_{\ell}$ is adjacent to $g_j$ if and only if $f_{\ell}\cap g_j\ne \emptyset$. Since all $f_{\ell}$'s are disjoint and $|g_j|\leq k-1$, we have $\deg_G(g_j)\le k-1$ for every $j$. By averaging, there is a vertex ${\ell}_0\in [2k-1]$ such that $\deg_G(f_{\ell_0})\le k-1$. Since $\deg_G(g_{\ell_0})\le k-1$ and $\deg_G(f_{\ell_0})+\deg_G(g_{\ell_0})< 2k-1$, there exists ${\ell}_1\in [2k-1]$ such that $f_{\ell_1}\cap g_{\ell_0}= \emptyset$ and $f_{\ell_0}\cap g_{\ell_1}= \emptyset$. Together with the fact $y\notin g_j$ for every $j\in [2k-1]$, we obtain that $f_{\ell_0}\cup\{y\}, f_{\ell_1}\cup\{y\}, g_{\ell_1}\cup\{y'\}$ and $g_{\ell_0}\cup\{y'\}$ forms a $4$-cycle with clean intersections, which contradicts Property iv) of $\bouquet$.
\end{proof}

Claim \ref{cl:numcomponent} says that $H_y^{y'}$ has at most $2k-1$ components. Therefore, the edges of $H_y^{y'}$ are in at most $2k-1$ components of $H_y$. By Proposition \ref{clm:BouStr}, each component of $H_y$ has size at most $\max\{r, \D_{r-1}(H_r)\}= O(t/ (\log t)^4)$ for some $3\leq r \leq k$. Hence,
\begin{align*}
    |H_y^{y'}|= (2k-1) \, O\left(\frac{t}{(\log t)^4}\right) =O\left(\frac{t}{(\log t)^4}\right),
\end{align*}
which complete the proof of Proposition \ref{cl:vertex_dep}.
\end{proof}

By Proposition \ref{cl:vertex_dep}, we have $|\tH_{y}| = |H_{y}| - O(t/(\log t)^{4})$. Thus, Proposition \ref{clm:tH} gives that $\mP\left(\bigwedge_{f\in \tH_{y}}\{f\not\subseteq C\}\right) \le e^{-1+ O(1/(\log t)^2)}$.  By \eqref{eq:E'T}, we obtain the desired upper bound in \eqref{eq:PET}:
\[
   \mP(E_R) \leq p^{i-j}(1-p)^{j-1}e^{-j+1+O(1/(\log t)^{2})} \leq \left(1+ \frac{\eps}4 \right)p^{i-j}e^{-j+1}.
\]

\vspace{0.1cm}

\subsection{Step 4}

For $x\in V\setminus N(B)$ and $2\le j\le i\le k$, we computed the mean $\mu_{i\to j}(x):= \mE(\deg_{i\to j}(x)\mid x\notin C)$ in Step 3. The goal of this step is to show that $\deg_{i\to j}(x)$ conditioned to $x\notin C$ concentrate around its mean. As a consequence, we are going to show with high probability that $\deg_{H'_i}(x)\le (1+\eps)^{m+1}\binom{k-1}{k-i}\a_{m+1}^{k-i} (t/e)^{i-1}$  for $x\in V'\setminus N(B)\cup Z$, where $|Z|\leq \frac{\eps}{3}\frac{\g n}{et}$.

Fix $x\in V\setminus N(B)$, $2\le j\le i\le k$ and consider the random space where we condition that $x\notin C$. We know that 
\[
\deg_{i\to j}(x) = \sum_{e\in H_{i, x}} \mathbbm{1}_{Y_e},
\]
where $\mathbbm{1}_{Y_e}$ is the indicator of the event $Y_e$ defined in \eqref{eq:Ye}. The concentration of $\deg_{i\to j}(x)$ depends on the correlation of events $\{Y_e : e\in H_{i,x}\}$.

We observe that given $e \in H_{i,x}$ and $Q\subseteq H_{i,x}$, the event $Y_e$ is mutually independent with the events 
$\left\{Y_f : {f\in Q} \right\}$ if 
\[
\left(N(e)\cap \left(\bigcup_{f\in Q}N(f)\right)\right)\setminus\{x\}=\emptyset.
\] 
Indeed, this holds because $Y_e$ only depends on the events $\{z\in C\}$ for $z\in N(e)\setminus \{x\}$ (since we already know $x\not\in C$). 

We define the graph $D_{i,x}$ whose vertices are the events $\{Y_e:\: e\in H_{i,x}\}$ and two vertices $Y_e,Y_f$ are adjacent if and only if $(N(e)\cap N(f))\setminus \{x\} \neq \emptyset$. Note that by the previous paragraph, the dependency graph of the events $\{Y_e:\: e\in H_{i,x}\}$ is a subgraph of $D_{i,x}$.
Thus, in order to apply the Almost Independent Lemma (Lemma~\ref{lem:AI}), it suffices to find an upper bound for $\D(D_{i, x})$. 

Our approaches and bounds are different when $i=2$ and when $i\ge 3$. For $i=2$, our bound comes from the following more general result that will also be applied in Step~5: Fix $2\le \ell\le k$. For $K\subseteq V$ with $|K|= \ell-1$, let $V_K = \{v \in V\setminus K: \{v\}\cup K\in H\}$. Define a graph $D_K$ on $V_K$ such that two vertices $u, v\in V_K$ are adjacent if and only if $(N(u)\cap N(v)) \setminus K\ne \emptyset$. Note that $D_{\{x\}} \cong D_{2, x}$.

\begin{proposition}\label{clm:DDS}
$\D(D_K)\le k-1$.
\end{proposition}
We use the following claim to facilitate the proof Proposition~\ref{clm:DDS}.

\begin{claim}\label{cl:funnyintersection}
For any $u, v\in V_K$, there are no $z\in V(H)\setminus K$ and edges $e, e'\in H$ such that $e\cap \{u, v, z\}= \{u, z\}$ and $e'\cap \{u, v, z\} = \{v, z\}$. 
\end{claim}

\begin{proof}
Suppose such $z, e, e'$ exist.
If $K\cap e \ne \emptyset$, then $| (K\cup \{u\}) \cap e | \ge 2$, which implies that $e\in H_{\ell}$, $\ell\ge 3$, and $| (K\cup \{u\}) \cap e | = \ell - 1$. Hence $| (K\cup \{v\}) \cap e | = \ell - 2$, and three edges $e, K\cup \{u\}, K\cup \{v\}$ fail Property v') of $\bouquet$. The same holds if $K\cap e' \ne \emptyset$. We thus assume that $K\cap e = \emptyset = K\cap e' $. Now $K\cup \{u\}, K\cup \{v\}, e', e$ form a 4-cycle with clean intersections, contradicting Property iv) of $\bouquet$. 
\end{proof}

\begin{proof}[Proof of Proposition~\ref{clm:DDS}]
Suppose $u, v\in V_K$ are adjacent in $D_K$. Then there exists $z\in (N(u)\cap N(v))\setminus K$, which implies that $\{u, z\} \subseteq e$ and $\{v, z\} \subseteq e'$ for some $e, e'\in H$. By Claim~\ref{cl:funnyintersection}, we have either $\{u,v\}\subseteq e$ or $\{u,v\} \subseteq e'$. 
This implies that if two vertices are adjacent in $D_K$, then they belong to the same edge of $H$.

Suppose to the contrary, that there exists a vertex $x \in V_K$ such that $\deg_{D_K}(x)\geq k$. Let $e_1$ be an edge of $H$ that contains $x$ and the most neighbors of $x$ in $D_K$. Since $\deg_{D_K}(x)\geq |e_1|$, $e_1$ must omit some neighbor $y_2$ of $x$. By the paragraph above, there exists $e_2\in H$ such that $\{x, y_2\}\subseteq e_2$. By the choice of $e_1$, there exists a neighbor $y_1\in e_1\setminus e_2$ of $x$. Since $e_1\cap \{y_1, y_2, x\} = \{y_1, x\}$ and $e_2\cap \{y_1, y_2, x\} = \{y_2, x\}$, we obtain a contradiction with Claim~\ref{cl:funnyintersection}.
\end{proof}

We now bound $\D(D_{i, x})$ for $i\geq 3$.

\begin{proposition}\label{clm:edge_dep}
$\D(D_{i, x})\leq 2\D_{i-1}(H_i) + 3i- 3$ for $i\geq 3$.
\end{proposition}

\begin{proof}
Throughout the proof we will denote the set of vertices of $D_{i,x}$ by the correspondent set of edges $H_{i,x}$. That is, the vertex $Y_e$ will be viewed as the edge $e \in H_{i,x}$. 
For every $e\in H_{i, x}$, let $C_e$ denote the connected components of $H_{i,x}$ that contains $e$. Since every component of $H_x$ is an intersecting family, any two edges 
$e,f$ in the same component are adjacent in $D_{i,x}$. We now study the edges of $D_{i,x}$ between different components.

\begin{claim}\label{clm:ffg}
Suppose $f_1, f_2 \in H_{i, x}$ are disjoint and there exists $g\in H$ such that $g\cap f_i\ne \emptyset$ for $i=1, 2$. Then 
$g\in H_i$ and there exists $j\in \{1, 2\}$ such that $f_j \subseteq g$ and $|g\cap f_{3-j} |=1$ (see Figure \ref{fig:fig5}). Furthermore, $|C_{f_j}|=1$.
\end{claim}

\begin{proof}
Since $f_1\cup\{x\}, f_2\cup\{x\}\notin H_2$, by Property iii) of $\bouquet$, the 3-cycle formed by $f_1\cup\{x\}, f_2\cup\{x\}$ and $g$ is not linear. Hence, either $|(f_1\cup\{x\})\cap g|>1$ or $|(f_2\cup\{x\})\cap g|>1$. By Properties i) and ii) of $\bouquet$, we obtain that $g\in H_i$ and $|(f_j\cup\{x\})\cap g|=i-1$ for some $j\in \{1, 2\}$.
Note that $f_1$ and $f_2$ are disjoint and thus in different components of $H_{i, x}$. Since $g$ has a nonempty intersection with both $f_1$ and $f_2$, it follows that $x\notin g$ (otherwise $f_1$ and $f_2$ are in the same component). 
Thus, $|(f_j\cup\{x\})\cap g|=i-1$ implies that $|f_j\cap g|=i-1$. Since $|f_j|=i-1$ and $g\cap f_{3-j}\ne \emptyset$, we obtain that $f_j\subseteq g$ and $|f_{3-j}\cap g|=1$. Furthermore, if $|C_{f_j}|>1$, then there exists another $f'\in H_{i, x}$ such that $|f'\cap f_j|= i - 1$. Then $|g\cap f'|=i-2$ and $|(f'\cup\{x\})\cap g|=i-2$. 
Since $i\ge 3$, $f'\cup\{x\}$, $f_j\cup\{x\}$ and $g$ form a 3-cycle that contradicts Property v') of $\bouquet$.
\end{proof}

\begin{figure}[h]
	\begin{tikzpicture}[scale=1]
		\coordinate (y) at (0,0);
		\coordinate (y1) at (1,-0.4);
		\coordinate (y2) at (1.5,-0.6);
		\coordinate (y3) at (2,-0.8);
		\coordinate (y4) at (2.5,-1);
		\coordinate (y5) at (3,-1.2);
		
		\coordinate (z1) at (1,0.4);
		\coordinate (z2) at (1.5,0.6);
		\coordinate (z3) at (2,0.8);
		\coordinate (z4) at (2.5,1);
		\coordinate (z5) at (3,1.2);
		
		\fill (y) circle (2pt);
		\fill (y1) circle (2pt);
		\fill (y2) circle (2pt);
		\fill (y3) circle (2pt);
		\fill (y4) circle (2pt);
		\fill (y5) circle (2pt);
	
		\fill (z1) circle (2pt);
		\fill (z2) circle (2pt);
		\fill (z3) circle (2pt);
		\fill (z4) circle (2pt);
		\fill (z5) circle (2pt);
		
		\redge{(y1)}{(y5)}{(y1)}{(y5)}{7.5pt}{1.5pt}{green!80!black}{green,opacity=0.2};
		\redge{(z1)}{(z5)}{(z1)}{(z5)}{7.5pt}{1.5pt}{green!80!black}{green,opacity=0.2};
		\redge{(z1)}{(y5)}{(y1)}{(z1)}{4pt}{1.5pt}{orange!80!black}{orange,opacity=0.2};

    \node (q1) at (0,-7.5pt) [below, font=\small] {$x$};
    \node (q2) at (3cm+7.5pt,-1.2) [above right, font=\small] {$f_1$};
    \node (q3) at (3cm+ 7.5pt,1.2) [above right, font=\small] {$f$};
    \node (q4) at (2.1,0) [right, font=\small] {$g$};
    
        \coordinate (u) at (6,0);
		\coordinate (u1) at (7,-0.4);
		\coordinate (u2) at (7.5,-0.6);
		\coordinate (u3) at (8,-0.8);
		\coordinate (u4) at (8.5,-1);
		\coordinate (u5) at (9,-1.2);
		
		\coordinate (v1) at (7,0.4);
		\coordinate (v2) at (7.5,0.6);
		\coordinate (v3) at (8,0.8);
		\coordinate (v4) at (8.5,1);
		\coordinate (v5) at (9,1.2);
		
		\fill (u) circle (2pt);
		\fill (u1) circle (2pt);
		\fill (u2) circle (2pt);
		\fill (u3) circle (2pt);
		\fill (u4) circle (2pt);
		\fill (u5) circle (2pt);

		\fill (v1) circle (2pt);
		\fill (v2) circle (2pt);
		\fill (v3) circle (2pt);
		\fill (v4) circle (2pt);
		\fill (v5) circle (2pt);
		
		\redge{(u1)}{(u5)}{(u1)}{(u5)}{7.5pt}{1.5pt}{green!80!black}{green,opacity=0.2};
		\redge{(v1)}{(v5)}{(v1)}{(v5)}{7.5pt}{1.5pt}{green!80!black}{green,opacity=0.2};
		\redge{(u1)}{(v1)}{(v5)}{(v5)}{4pt}{1.5pt}{orange!80!black}{orange,opacity=0.2};

    \node (q5) at (6,-7.5pt) [below, font=\small] {$x$};
    \node (q6) at (9cm+7.5pt,-1.2) [above right, font=\small] {$f$};
    \node (q7) at (9cm+ 7.5pt,1.2) [above right, font=\small] {$f_2$};
    \node (q8) at (8.1,0) [right, font=\small] {$g$};
    
	\end{tikzpicture}
	\caption{Two possible structures obtained from Claim~\ref{clm:ffg}}
	\label{fig:fig5}
\end{figure}
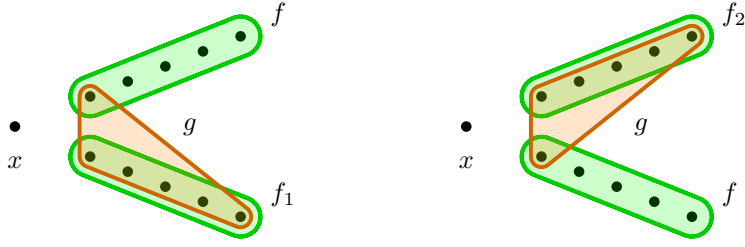

\begin{claim}\label{cl:orientation}
Suppose $f_1, f_2\in H_{i, x}$ are disjoint edges that are adjacent in $D_{i, x}$. Then there exists $g\in H_i$ and $j\in \{1, 2\}$ such that $f_j \subseteq g$ and $|g\cap f |=1$ for some $f \in C_{f_{3-j}}$. Furthermore, $|C_{f_j}|=1$.
\end{claim}

\begin{proof}
If there exists $g\in H$ such that $g\cap f_i\ne \emptyset$ for $i=1, 2$, then the conclusion immediately follows (with $f= f_{3-j}$) from Claim~\ref{clm:ffg}. We thus assume that no $g\in H$ such that $g\cap f_i\ne \emptyset$ for $i=1, 2$,

Since $f_1, f_2$ are adjacent in $D_{i, x}$, we have $(N(f_1)\cap N(f_2))\setminus\{x\}\neq\emptyset$. Since there is no edge in $H$ intersecting both $f_1$ and $f_2$, this implies that there exist $g_1,g_2\in H$ satisfying the following intersection relations:
\begin{align*}
f_1\cap g_1\neq \emptyset,\quad f_2\cap g_2\neq \emptyset,\quad g_1\cap g_2\neq \emptyset, \quad f_1\cap g_2=\emptyset \quad\text{and} \quad f_2\cap g_1=\emptyset.    
\end{align*}
If $x\notin g_1\cup g_2$, then $f_1\cup\{x\}, f_2\cup\{x\}, g_2$ and $g_1$ form a $4$-cycle with clean intersections, contradicting Property iv) of $\bouquet$. Otherwise, assume that $x\in g_1$ (the case when $x\in g_2$ is similar). Then $|g_1\cap (f_1\cup\{x\})|\geq 2$. By Properties i) of $\bouquet$ we obtain that $g_1\in H_i$ and $|g_1\cap f_1|= i-2$. Since $g_1\cap g_2\ne \emptyset$ and $g_2\cap f_1= \emptyset$, it follows that $|g_1\cap g_2|=1$. Let $f:=g_1\setminus\{x\}$. Then $f\in C_{f_1}$.
Applying Claim~\ref{clm:ffg} with $f, f_2$ and $g_2$, we obtain that $f_2\subseteq g_2$ and $|C_{f_2}|=1$ (we already know $|g_2\cap f|=1$), as desired. 
\end{proof}

For convenience, for two disjoint $f_1, f_2 \in H_{i, x}$, we write $f_1\mapsto f_2$ when there exists an edge $g\in H$ such that 
$f_1 \subseteq g$ and $|g\cap f_{2} |=1$. 
For $f_1 \in H_{i, x}$ and $y\in V(H)\setminus f_1$, we write $f_1\mapsto y$ if there exist $g\in H$ and $f_2 \in H_{i, x}$ disjoint from $f_1$ such that 
$f_1 \subseteq g$ and $\{y\}= g\cap f_{2}$.  We claim that for any set $Y\subseteq V(H)$,
\begin{align}
\label{eq:toy}
|\{f\in H_{i, x}: f\mapsto y \text{  for some  } y\in Y\}|\le |Y|. 
\end{align}
To see this, it suffices to show that no $y\in Y$ and two edges  $f_1, f_2 \in H_{i, x}$ such that $f_i\mapsto y$ for $i=1, 2$. This implies that there exist $g_1, g_2\in H$ such that $g_i = f_i \cup \{y\}$ for $i=1, 2$.  By Claim~\ref{clm:ffg}, $|C_{f_1}| = |C_{f_2}|=1$ and thus $f_1, f_2$ are disjoint. Hence $f_1\cup \{x\}, g_1, g_2, f_2\cup \{x\}$ form a 4-cycle with clean intersections, contradiction.

Now we are ready to prove Proposition~\ref{clm:edge_dep}. Fix $f_1\in H_{i, x}$. We first bound $|V(C_{f_1})|$, the number of vertices of $H$ in the component $C_{f_1}$. When $|C_{f_1}| = 1$, we have $|V(C_{f_1})| = i - 1$. When $|C_{f_1}| \ge 2$, 
Proposition~\ref{clm:BouStr} implies that
\begin{align}
\label{eq:VCf}
|V(C_{f_1})| \le \max\{i, |C_{f_1}| + i-2\}=  |C_{f_1}| + i-2.
\end{align}

By Claim~\ref{cl:orientation}, another $f_2\in H_{i, x}$ is adjacent to $f_1$ in $D_{i, x}$ if and only if one of the following holds:
\begin{enumerate}[(i)]
\item $f_2\in C_{f_1}$;
\item $f_1 \mapsto f'$ for some $f'\in C_{f_2}$ and $|C_{f_1}|=1$;
\item $f_2 \mapsto f'$ for some $f'\in C_{f_1}$ and $|C_{f_2}|=1$.
\end{enumerate}

First, assume that $|C_{f_1}|\ge 2$. In this case no $f_2$ satisfies (ii). By Proposition~\ref{clm:BouStr}, there are at most $|C_{f_1}|\le \Gamma_i:= \max\{i, \D_{i-1}(H_i) \}$ members of $H_{i, x}$ satisfying (i), including $f_1$ itself. By \eqref{eq:toy} and \eqref{eq:VCf}, there are at most 
\[
| V(C_{f_1}) |\le |C_{f_1}| + (i-2)\le \Gamma_i + i-2
\]
$f_2$ satisfying (iii). Since $\Gamma_i\le i+\D_{i-1}(H_i)$, it follows that
\[
\deg_{D_{i, x}}(f_1) \le \Gamma_i - 1 + \Gamma_i + i-2 \le 2\D_{i-1}(H_i) + 3i-3.
\]

Second, assume $|C_{f_1}|=1$. There is no other member in $C_{f_1}$ and by \eqref{eq:toy}, at most $|V(C_{f_1})|= i-1$ members of $H_{i, x}$ satisfying (iii). We claim that all $f_2$'s satisfying (ii) are in the same component of $H_{i, x}$, and thus there are at most $\Gamma_i$ such $f_2$'s.  Suppose this is not the case. Then there are $f', f''$ from two different component of $H_{i, x}$ such that $f_1\mapsto f'$ and $f_1\mapsto f''$. Note that $f_1, f', f''$ are pairwise disjoint. 
Then there exist $g', g''\in H$ such that $|g'\cap f'|=1$, $|g''\cap f''|=1$, and $g'\cap g''= f_1$. This implies that $f'\cup \{x\}$, $f''\cup \{x\}$, $g'', g$ form a 4-cycle  with clean intersections, contradiction. We thus obtain that 
\[
\deg_{D_{i, x}}(f_1) \le \Gamma_i  + i - 1 \le \D_{i-1}(H_i)  + 2i - 1.
\]
We thus have $\D(D_{i, x})\leq 2\D_{i-1}(H_i) + 3i- 3$, as desired.
\end{proof}

To obtain the concentration of $\deg_{2\to 2}(x)$, we apply Lemma~\ref{lem:AI} with random variables $\mathbbm{1}_{Y_e}$ for all $e\in H_{2, x}$.
Proposition \ref{clm:DDS} gives that $\D(D_{2, x})\le k-1$. With $s=k-1$, $p\ge  (1-\eps /4) e^{-1}$ from \eqref{eq:PYe}, and $n= |H_{2, x}| = (1+\eps)^m \binom{k-1}{k-2} \a^{k-2} t$ from \eqref{degrees}, Lemma~\ref{lem:AI} implies that
\begin{align}
&\mP\left(  |\deg_{2\to 2}(x) -  \mu_{2\to 2}(x) |> \frac{\eps}4  \mu_{2\to 2}(x)\, \middle|\, x\notin C\right) \nonumber \\
&\le 2k\exp \left(- \left(\frac{\eps}{4}\right)^2 \frac{(1-\frac\eps4)(1+\eps)^m \binom{k-1}{k-2}\a^{k-2} t/e}{3k} \right) \nonumber \\
&=e^{-\Omega\left(t/(\log t)^2\right)}\le \frac1{t^2}, \label{eq:d22}
\end{align} 
where we use the assumption $\a \geq 1$ and $\eps=1/\log t$. 

To obtain the concentration of $\deg_{i\to j}(x)$ for $i>j \ge 2$, we apply Lemma~\ref{lem:AI} with $s=2\D_{i-1}(H_i)+3i-3$ (due to Proposition~\ref{clm:edge_dep}), 
$p\ge  (1-\eps /4) \binom{i-1}{j-1} (\g / t)^{i-j} e^{-j+1}$ from \eqref{eq:PYe} and $n= | H_{i, x}| = (1+\eps)^m \binom{k-1}{k-i} \a^{k-i} t^{i-1}$ from \eqref{degrees}. Since $\eps=1/\log T\geq 1/(2\log t)$ and $2\D_i(H_{i-1})+3i-3\leq 2(1+\eps)^m\frac{t}{(\log t)^{i+1}}+3i-3\leq 4\frac{t}{(\log t)^{i+1}}$ we obtain that
\begin{align*}
&\mP \left( |\deg_{i\to j}(x) -  \mu_{i\to j}(x) | > \frac{\eps}4 \mu_{i\to j}(x) \, \middle| \, x\notin C\right) \\
&\leq 2 \left(2\D_{i-1}(H_i)+i-1\right) \exp \left( - \left(\frac{\eps}{4}\right)^2 \frac{(1-\frac\eps4) \binom{i-1}{j-1} \left({\g}/{t}\right)^{i-j} e^{-j+1}  (1+\eps)^m \binom{k-1}{k-i} \a^{k-i} t^{i-1} } {3\left( 2\D_{i-1}(H_i)+i-1\right)} \right) \\
&\le 8 \frac{t}{ (\log t)^{i+1} }\exp\left( - \frac{1}{64 (\log t)^2 } \frac{ \frac{1}{2}\binom{i-1}{j-1} \binom{k-1}{k-i}  \a^{k-i} \g^{i-j}\left(t/e\right)^{j-1} }{ 12{t}/{ (\log t)^{i+1}} } \right)\\
&\le t \exp\left( - c_1 \, \a^{k-i} \g^{i-j} t^{j-2} (\log t)^{i-1} \right)
\end{align*}
for sufficiently large $t$ and some constant $c_1>0$.
By the estimates of $\a$ and $\g$ from \eqref{eq:ag} and the fact
\[
k-i + (i-1)(k-1) - (i-j)(k-2) = j(k-2)+1 \ge 2k-3, 
\]
there is a constant $c_2>0$ such that 
\[
c_1 \, \a^{k-i} \g^{i-j} t^{j-2} (\log t)^{i-1} \ge c_2 \, (\log t)^{ \frac{k-i}{k-1} - (i-j)\frac{k-2}{k-1} + i-1}
\ge c_2\, (\log t)^{\frac{2k-3}{k-1}} \ge 3 \log t
\]
for sufficiently large $t$. It follows that 
\begin{align} \label{eq:Pdij}
\mP \left( |\deg_{i\to j}(x) -  \mu_{i\to j}(x) | > \frac{\eps}4 \mu_{i\to j}(x) \, \middle| \, x\notin C\right) 
&\leq t \exp \left( -3\log t\right) = \frac 1{t^2}.
\end{align}

Recall that for $2\leq j \leq i \leq k$, the set $Z_{i,j}$ is the  set of vertices $x \in V\setminus (N(B)\cup C)$ such that $\deg_{i\to j}(x) > (1+\eps/4)\mu_{i\to j}(x)$.
Using \eqref{eq:d22} or \eqref{eq:Pdij}, we obtain that
\begin{align*}
    \mP(x\in Z_{i,j} \mid x\notin C)\leq \frac{1}{t^2},
\end{align*}
for $2\leq j\leq i \leq k$. Since $x\in C$ implies that $x\notin Z$, we have that
\begin{align*}
    \mP(x \in Z_{i,j})&=\mP(x\in Z_{i,j}\mid x \notin C)\mP(x\notin C)+\mP(x\in Z_{i,j}\mid x \in C)\mP(x\in C)\\
    &=(1-p)\mP(x\in Z_{i,j}\mid x\notin C)< \mP(x\in Z_{i,j}\mid x\notin C) \leq \frac{1}{t^2}.
\end{align*}
Thus,
\begin{align}\label{eq:Zij}
\mE(|Z_{i,j}|)\leq n/t^2. 
\end{align}
By Markov's inequality, we have $\mP\left(|Z_{i,j}|>k^3\mE(|Z_{i,j}|) \right)\leq \frac1{k^3}$.
Since $Z=\bigcup_{2\leq j\leq i \leq k}Z_{i,j}$, by an union bound we derive that 
\begin{align*}
    &\mP\left(|Z|>\sum_{2\leq j \leq i \leq k}k^3\mE(|Z_{i,j}|)\right) \le \mP\left( \bigvee_{2\leq j \leq i \leq k} \Big(|Z_{i,j}| > k^3\mE(|Z_{i,j}|)\Big)  \right) \\
    &\leq \sum_{2\leq j \leq i \leq k}\mP\Big(|Z_{i,j}| > k^3\mE(|Z_{i,j}|)\Big) < \frac{1}{k}.
\end{align*}
Therefore, by \eqref{eq:Zij}, $|Z|\leq k^5n/t^2$ with probability at least $\frac{k-1}{k}$. 
In other words, with probability at least $\frac{k-1}{k}$, there is a set $Z$ of size 
\begin{align}
\label{eq:Z}
|Z|\le \frac{k^5 n}{t^2} \le \frac{\eps}3 \frac{\g n}{e t} 
\end{align}
such that all vertices $x\in V'\setminus (N(B)\cup Z)\subseteq V\setminus (N(B)\cup C \cup Z)$ satisfy $\deg_{i\to j}(x) \le (1+\eps/4) \mu_{i\to j}(x)$ for all $2\le j\le i\le k$. 

Recall that $H'$ is the contraction of the subgraph $H[V'\cup C]$ on the set of vertices $V'$, where we including multiple edges. For every $x\in V'\setminus (N(B)\cup Z)$, the above arguments give that
\begin{align*}
\deg_{H'_i}(x) &= \sum_{\ell=i}^k \deg_{\ell\to i}(x) \leq \sum_{\ell=i}^k\left(1+\frac{\eps}4\right)\mu_{\ell\to i} \\
&\le \left(1+\frac{\eps}4\right)^2 (1+\eps)^{m}\sum_{\ell=i}^k\binom{k-1}{k-\ell}\binom{\ell-1}{i-1}\a^{k-\ell}\g^{\ell-i}\left(\frac{t}{e}\right)^{i-1}.
\end{align*}
Since $\eps\ll 1$ and $\binom{k-1}{k-\ell}\binom{\ell-1}{i-1} = \binom{k-1}{k-i}\binom{k-i}{\ell-i}$, it follows that
\begin{align*}
\deg_{H'_i}(x) 
&\le (1+\eps)^{m+1} \left(\frac te \right)^{i-1}\sum_{\ell=i}^k\binom{k-1}{k-i}\binom{k-i}{\ell-i}\a^{k-\ell}\g^{\ell-i} \\
&= (1+\eps)^{m+1} \binom{k-1}{k-i} \left( \frac te \right)^{i-1}\sum_{\ell=0}^{k-i} \binom{k-i}{\ell}\a^{k-i-\ell}\g^\ell \\
&=(1+\eps)^{m+1} \binom{k-1}{k-i} (\a+\g)^{k-i} \left( \frac te \right)^{i-1}. 
\end{align*}
Since $\a+\g = \a_{m+1}$, this give the desired bound 
\begin{align}
\label{eq:di}
\deg_{H'_i}(x) \le (1+\eps)^{m+1} \binom{k-1}{k-i} \a_{m+1}^{k-i} \left( \frac te \right)^{i-1}.
\end{align}

\vspace{0.1cm}

\subsection{Step 5}
Recall that $H'$ is a hypergraph with $V(H')=V'$ and $E(H')$ is the multiset of $e \cap V'$ for all $e \in H[V' \cup C]$. We first extend the definition of $\deg_{H'}(S)$ to any set $S\subseteq V$ of size at most $k$ such $\deg_{H'}(S)= \{e\setminus S: e\in H, e\setminus S \subseteq V'\cup C\}$.
In this step we show that \emph{w.h.p for every $P\subseteq V'\setminus N(B)$ of size $2\le |P|<k$},
\begin{align}\label{eq:dH'S}
\deg_{H'}(P)\leq (1+\eps)^{m+1}\frac{t/e}{(\log(t/e))^{|P|+2}}.
\end{align}

Fix a set $P \subseteq V\setminus N(B)$ of size $2\le |P|<k$ and consider the random space where we condition on $P\cap C=\emptyset$. By Property ii) of $\bouquet$, all the edges of $H$ that contain $P$ must have the same size $r$. If $|P|=r$, then $\deg_{H'}(P)\leq \deg_H(P)=1$ and we are done. We thus assume that $r\ge |P|+1\ge 3$.
If $\deg_H(P) \le (1+\eps)^{m+1}\frac{t/e}{(\log (t/e))^{r+1}}$, then \eqref{eq:dH'S} holds because $|P|+1\le r$ and $\deg_{H'}(P)\leq \deg_H(P)$. Thus we may assume that 
\begin{align}\label{eq:dS}
\deg_H(P) > (1+\eps)^{m+1}\frac{t/e}{(\log (t/e))^{r+1}}> k+1. 
\end{align}

By Proposition \ref{cl:bouquet}, there exists a set $K\subseteq V$ of size $\ell-1$, with $P\subseteq K$, such that every edge containing $P$ contains $K$.
Let $V_K=\{x \in V:\: K\cup\{x\} \in H_r \}$. Then $\deg_H(P)=|V_K|$.

\begin{claim}\label{cl:degP}
$\deg_{H'}(P)\leq |V_K\setminus (D\setminus C)|$.
\end{claim}

\begin{proof}
First assume that $K\setminus P \not \subseteq V'\cup C$. Then, by the definition of $H'$, all edges containing $K\setminus P$ will be discarded in $H'$. Since every edge of $H$ containing $P$ contains $K$,
it follows that $\deg_{H'}(P)=0$ and the claim holds trivially. Now assume that $K\setminus P\subseteq V'\cup C$. Thus, by definitions of $\deg_{H'}(P)$ and $V'=V\setminus (B\cup C\cup D)$, we have
\begin{align*}
    \deg_{H'}(P)=\deg_{H'}(K)=|\{x \in V_K:\: x \in V'\cup C\}|\leq |V_K\setminus (D\setminus C)|.
    &\qedhere
\end{align*}
\end{proof}

For $x\in V_K$, let $E_x$ be the event $\{x\notin D\setminus C \mid P\cap C=\emptyset\}$.
Let $H_x^{K}=\{f \in H_x:\: f \cap K\neq \emptyset\}$ be the set of edges of $H_x$ that intersects with $K$. We claim that $H_x^K=\{K\}$. Indeed, suppose $e\in H_x\setminus \{K\}$ satisfies $e\cap K\ne \emptyset$. 
Then $|(e\cup \{x\}) \cap (K\cup \{x\}) |\ge 2$ and Properties i) and ii) of $\bouquet$ imply that $|e|=|K|=r-1$ and $|e\cap K|=r -2$.  By \eqref{eq:dS}, we have $|V_K|>k+1$  and consequently, there exists $y\in V_K\setminus e$. Hence $|(e\cup \{x\}) \cap (K\cup \{y\}) |=r - 2$ and the edges $e\cup \{x\}$, $K\cup \{x\}$, and $K\cup \{y\}$ contradicts Property v') of $\bouquet$.

As in Steps $3$ and $4$, we want to compute the expected size of $\deg_{H'}(P)$ conditioned on $P\cap C=\emptyset$. Let $\tH_x=H_x\setminus H_x^K$ be the set of $e \in H_x$ such that $e\cap K=\emptyset$. Since $P\cap C=\emptyset$ implies that $K \not\subseteq C$ and $e\in \tH_x$ implies that $e\cap K=\emptyset$, we obtain that
\begin{align}\label{eq:PEx2}
    \mP(E_x)&=\mP(x\in C)+\mP(x\notin C)\mP(x\notin D \mid P\cap C=\emptyset) \nonumber\\
    &= p+(1-p)\mP\left(\bigwedge_{e\in H_x}\{e\not\subseteq C\}\, \middle| \,P\cap C=\emptyset\right) \nonumber\\
    &= p+(1-p)\mP\left(\bigwedge_{e\in \tH_x}\{e\not\subseteq C\}\right).
\end{align}

We have $|H_x\setminus \tH_x|=1=O\left(\frac{t}{(\log t)^2}\right)$, and $x\notin B$ because $x\in V_K$ and $P\cap N(B)=\emptyset$. Thus, we can apply Proposition \ref{clm:tH} and obtain that
\begin{align*} 
\mP(E_x) &= p+(1-p)e^{-1+O(1/(\log t)^2)}\le \left(1+\frac{\eps}{4}\right)e^{-1}.
\end{align*}

Let $X$ be the random variable $|V_K\setminus (D\setminus C)|$. By Claim~\ref{cl:degP}, we have
\begin{align}\label{eq:dSX}
\deg_{H'}(P)\le X.
\end{align}
Since $|V_K|= \deg_{H}(K)\le (1+\eps)^m t/ (\log t)^{r+1}$, we obtain that
\begin{align}\label{eq:EXS}
\mE(\deg_{H'}(P)\mid P\cap C=\emptyset)&\le \mE(X\mid P\cap C=\emptyset)=\sum_{x\in V_K} \mP( E_x ) \nonumber\\
&\le \left(1 + \frac{\eps}4 \right)\frac{|V_K|}e \le \left(1+ \frac{\eps}4 \right) (1+\eps)^m \frac{t/e}{ (\log t)^{r+1}}.
\end{align}

Now we consider the concentration of $X$. We claim that \emph{given $x\in V_K$ and $Q\subseteq V_K$,
the event $E_x$ is mutually independent of $\{E_y: y\in Q\}$ if $\left(N(x)\cap \left(\bigcup_{y\in Q} N(y) \right) \right)\setminus K=\emptyset$.}
Indeed, by \eqref{eq:PEx2}, $E_x$ depends on whether $z\in C$ for $z\in \bigcup_{f\in \tH_x} f$. Note that $f\cap K=\emptyset$ for every $f\in \tH_x$. Thus, $E_x$ depends on the event $z\in C$ for $z\in N(x)\setminus K$. Consequently, $E_x$ is mutually independent of $\{E_y: y\in Q\}$ if $N(x)\setminus K$ is disjoint from $\bigcup_{y\in Q} N(y)\setminus K$.

Recall that $D_K$ defined in Step $4$ is a graph on $V_K$ such that $x\in V_K$ is adjacent to $y\in V_K$ if and only if $(N(x)\cap N(y))\setminus K\ne \emptyset$. Thus the dependency graph for the events $\{E_x\}_{x \in V_K}$ is a subgraph of $D_K$. In order to apply the Almost Independent Lemma, it suffices to bound $\D(D_K)$. By Proposition \ref{clm:DDS} we have that $\D(D_P)\leq k-1$. Therefore, Lemma \ref{lem:AI} with $\{\mathbbm{1}_{E_x}\}_{x\in V_K}$, $n=|V_K|$, $p= (1-\eps/4)/e$, $s=k-1$,  and $\eps/4$ gives us that
\begin{align*}
    \mP\left(|X-\mE(X)|\geq \frac{\eps}4 \mE(X)\, \middle|\, P\cap C=\emptyset\right)  \leq 2k\exp \left( -\frac{(\eps/4)^2  (1- \eps/4)|V_K|} {3ke}\right). 
\end{align*}
Since $|V_K|= \deg_H(P)\ge (1+\eps)^{m+1} (t/e) / (\log (t/e))^{r+1}$ by \eqref{eq:dS}, it follows that 
\begin{align*}
    \mP\left(|X-\mE(X)|\geq \frac{\eps}4 \mE(X)\, \middle|\, P\cap C=\emptyset\right)
    & \leq 2k\exp \left(-\frac{\eps^2  (1- \eps/4) (1+\eps)^{m+1} t}{48 e^2k (\log(t/e))^{r+1}}\right) \\
    = 2k e^{- \Omega( t/ (\log t)^{r+3} ) }
\end{align*}
because $\eps \ge 1/ (2\log t)$. By \eqref{eq:tlogn}, we have $t \geq (\log n)^{3/2}$, which implies that 
$t/ (\log t)^{r+3}$ $\gg k\log n $. Consequently,
\begin{align}\label{eq:step5conditioned}
\mP\left(|X-\mE(X)|\geq \frac{\eps}4 \mE(X)\, \middle|\, P\cap C=\emptyset\right) =o(n^{-k}).
\end{align}
Let $\mathcal{P}$ be the set of $P\subseteq V\setminus (N(B)\cup C)$ satisfying (\ref{eq:dS}) and $2\leq |P|< k$ such that 
\begin{align*}
    \deg_{H'}(P)>\left(1+\frac{\eps}{4}\right)\mE(X \mid P\cap C=\emptyset).
\end{align*}
Therefore, by (\ref{eq:step5conditioned}), 
\begin{align*}
    \mP(P\in \mathcal{P})\le \mP(P\cap C=\emptyset)\mP\left(X \geq \left(1+\frac{\eps}4\right) \mE(X)\, \middle|\, P\cap C=\emptyset\right)=o(n^{-k}).
\end{align*}
Thus, by Markov's inequality, we have $\mP(|\mathcal{P}|>0)=o(1)$. Hence, with probability $1- o(1)$, for all subsets $P \subseteq V'\subseteq V\setminus(N(B)\cup C)$ with $2\leq |P|<k$, we have $X\le (1+\eps/4)\mE(X)$. By \eqref{eq:dSX} and  \eqref{eq:EXS}, it follows that 
\begin{align*}
\deg_{H'}(P) & \le (1+\eps/4)\mE(X) \leq \left(1+ \frac{\eps}4 \right)^2 (1+\eps)^m \frac{t/e}{ (\log t)^{r+1} } \\
& \le (1+\eps)^{m+1} \frac{t/e}{(\log t)^{r+1}} \le (1+\eps)^{m+1}\frac{(t/e)}{(\log(t/e))^{|P|+2}}.   
\end{align*}

\subsection{Completing the proof}
As stated earlier, Lemma \ref{lem:key} i) comes from Lemma \ref{lem:structural} and the assumption that $H=H^m \in \bouquet$. To check Lemma \ref{lem:key} ii)--v), we choose a set $C=C_{m+1}$ such that \eqref{eq:V'2}, \eqref{eq:I'}, and \eqref{eq:Z} hold, 
and \eqref{eq:dH'S} holds for all $P\subseteq V'$ of size $2\le |P|<k$. Indeed, we showed that \eqref{eq:V'2} and \eqref{eq:I'} holds with probability $1-o(1)$,
\eqref{eq:Z} holds with probability $\frac{k-1}{k}$, and \eqref{eq:dH'S} holds with probability $1-o(1)$. Therefore, with positive probability, our desired $C$ exists.

Recall that $W=W_{m+1}=N(B)\cup Z$, $V^*=V_{m+1}=V'\setminus W$, and $I=I_{m+1}=I'\setminus W = C\setminus (W\cup D)$.
By \eqref{eq:tn} and \eqref{eq:NB}, we have  $| N(B) |\le \frac{\eps \g n}{3 e t}$;
by \eqref{eq:Z}, we have $|Z|\leq \frac{\eps \g n}{3et}$.
Therefore, \eqref{eq:V'2} and \eqref{eq:I'} imply that
\[
(1-\eps)\frac{n}{e} \le |V'| - |Z| - |N(B)|\le |V^*|\le |V'|\le (1+\eps) \frac{n}{e}, 
\] 
and $|I|\geq |I'|-|Z|-|N(B)|\geq (1-\eps)\frac{\g n}{e t}$. These give Lemma~\ref{lem:key} ii) and iii).

By the AKPSS algorithm, the hypergraph $H^*=H^{m+1}= (H^*_2, \dots, H^*_k)$ is the hypergraph obtained after removing duplicated edges and edges that properly contain other edges from $H[V^*\cup I]\mid_{V^*}$. Note that $H^*$ can be obtained from $H'$ by considering the subhypergraph of $H'[V^*]$ where we remove all the multiple edges, edges that properly contain another edge and the edges that intersect $D$ and $W$ in $H$ before contraction. Thus, $H^*$ is a subhypergraph of $H'[V^*]$ and consequently, $\deg_{H_i^*}(P)\le \deg_{H'_i}(P)$ for any $P$ of size $1\le |P|<k$. By \eqref{eq:di}, for every $x\in V^*$ and $2\le i\le k$, we have
\begin{align*}
    \deg_{H_i^*}(x)\leq \deg_{H'_i}(x)\le (1+\eps)^{m+1} \binom{k-1}{k-i} \a_{m+1}^{k-i} \left( \frac te \right)^{i-1}, 
\end{align*}
which gives Lemma~\ref{lem:key} v). Finally, by \eqref{eq:dH'S},  every 
$P\subseteq V^*$ with $|P|=i-1$ and $3\leq i \leq k$ satisfies
\[ 
\deg_{H_i^*}(P)\leq \deg_{H'_i}(P) \leq (1+\eps)^{m+1}\frac{(t/e)}{(\log(t/e))^{i+1}},
\] 
which gives Lemma~\ref{lem:key} vi). This concludes our proof.

\section*{Acknowledgment}
The authors thank two anonymous referees for their helpful comments.


\section{Appendix A: Proof of Theorem~\ref{th:application} }
\label{sec:P12}

In this section we prove Theorem~\ref{th:application}. The main idea is using vertex sampling to remove most cycles of length $2, 3, 4$ such that we can apply Theorem~\ref{thm:main}.

\begin{proof}[Proof of Theorem~\ref{th:application}]
To handle both Cases $(1)$ and $(2)$ simultaneously, we assume that $H$ is a $k$-graph on $n$ vertices with $\D_{k-2}(H)\leq dn$ for most of the proof, and only work on two cases separately at the end. In either case, if $n/d$ is bounded by a constant, then $\a(H)=\Omega\left(\left(\frac{n}{d}\log \frac{n}{d}\right)^{\frac{1}{k-1}}\right)$ follows from (\ref{Spencer}). Therefore, we may assume that $n/d$ is sufficiently large from now on.

Let $H_{k-1}$ consists of all $(k-1)$-tuples $S\in \binom{V(H)}{k-1}$ such that $\deg_H(S)\ge n^{\frac{k-2}{k-1}}d^{\frac{1}{k-1}}$. 
Consider $T\in  \binom{V(H)}{\ell}$, where $0 \le \ell \le k-2$. By the definition of $H_{k-1}$, we have
\[
\deg_H(T)=\frac{\sum_{S\supseteq T, |S|=k-1}\deg_H(S)}{k-\ell}\geq \frac{\deg_{H_{k-1}}(T) \, n^{\frac{k-2}{k-1}}d^\frac{1}{k-1}}{k-\ell}.
\]
Together with \eqref{eq:degH}, this gives that
\[
\deg_{H_{k-1}}(T) \, n^{\frac{k-2}{k-1}}d^{\frac{1}{k-1}} \le (k-\ell)\D_\ell(H) \leq 2 n^{k-\ell-1} d.
\]
Consequently, for $0\leq \ell \leq k-2$, we have
\begin{align}\label{eq:DlHk}
    \D_\ell(H_{k-1}) \le 2 d^{\frac{k-2}{k-1}} n^{k-\ell-1-\frac{k-2}{k-1}}.
\end{align}

Let $H_k$ be the subhypergraph of $H$ consisting of all edges that do not contain $(k-1)$-tuples from $H_{k-1}$. Since $H_k\subseteq H$, we obtain that, for $0\leq \ell \leq k-2$,
\begin{align}\label{eq:DlHk2}
    \Delta_\ell(H_k)\leq \Delta_\ell(H)\leq \frac{2}{(k-\ell)!} n^{k-\ell-1} d \leq n^{k-\ell-1}d.
\end{align}

Let $\cH=\{H_{k-1},H_k\}$ be the non-uniform hypergraph with edges of uniformity $(k-1)$ and $k$. Our goal is to find a subset $U''\subseteq V(\cH)$ such that the subhypergraph $\cH[U'']$ induced on $U''$ satisfies the assumption of 
Theorem \ref{thm:main} and consequently, contains a large independent set.  In order to do that, we use the deletion method to remove most cycles of size $2$, $3$ and $4$ from $\cH$. 
Given $\d >0$, let $U\subseteq V(H)$ be a random subset whose vertices are chosen independently, each one with probability 
\[
p= n^{\delta - \frac{k-2}{k-1}} d^{-\delta - \frac1{k-1}}.
\] 
Hence
\begin{equation}
\label{eq:pn}
p n = \left( \frac nd \right)^{\frac1{k-1} + \delta}.
\end{equation}

Since $|U|$ follows the binomial distribution $B(n, p)$, by Chernoff's bound, we have
\begin{align}\label{eq:|U|2}
\mP\left( \left| |U| - pn \right| \ge \frac{pn}{10}  \right) < e^{-c pn} < \frac1{8}
\end{align}
for some constant $c>0$ and sufficiently large $n$. In order to apply Theorem~\ref{thm:main}, we need to bound the maximum vertex degree in $H_k$ and $H_{k-1}$. For $k-1\leq i \leq k$, let $Z_i$ be the set of vertices $x \in U$ such that 
$\deg_{H_i[U]}(x) > 40 p^{i-1} \D_{1}(H_i)$.
Note that
\begin{align*}
    \mE(|H_i[U]|)=p^{i}|H_i|\leq \frac{p^in\D_1(H_i)}{i}
\end{align*}
for $k-1\leq i \leq k$. By Markov's inequality, we obtain that
\begin{align*}
    \mP\Big(\{|H_{k-1}(U)| \ge 4\mE(|H_{k-1}(U)|)\}\vee \{|H_{k}(U)|\ge 4\mE(|H_{k}(U)|)\}\Big)\leq \frac{1}{2}.
\end{align*}

If $|H_i[U]|< 4\mE(H_i[U])$, then, since $\sum_{x\in U}\deg_{H_i[U]}(x)=i|H_i(U)|$, we obtain that
\begin{align*}
    |Z_i|<\frac{i|H_i(U)|}{40p^{i-1}\D_1(H_i)}< \frac{4p^{i}n\D_1(H_i)}{40p^{i-1}\D_1(H_i)}=\frac{pn}{10},
\end{align*}
which implies that $|Z_{k-1}\cup Z_k|< \frac{pn}{5}$.
Thus, we have
\begin{align}
\label{eq:Zk-1}
    \mP\left(|Z_{k-1}\cup Z_k|\geq \frac{pn}{5} \right)\leq \frac{1}{2}.
\end{align}
For any $x\in U\setminus (Z_{k-1}\cup Z_k)$, we have 
\[
\deg_{H_i[U]}(x) \le 40 p^{i-1} \D_1(H_i)
\]
for $k-1\leq i \leq k$. Let $U' = U\setminus (Z_{k-1}\cup Z_k)$. Applying \eqref{eq:DlHk}--\eqref{eq:pn}, we derive that
\begin{align}
\label{eq:HU''}
    \D_1(H_k[U']) &\leq 40p^{k-1}\D_1(H_k)\le 40 (p n)^{k-1}/ \left(\frac nd\right) =  40\left(\frac{n}{d}\right)^{(k-1)\delta} \quad \text{and}\\
    \D_1(H_{k-1}[U'])&\leq 40p^{k-2}\D_1(H_{k-1})\le 80  (p n)^{k-2}/ \left(\frac nd\right)^{\frac{k-2}{k-1}} = 80\left(\frac{n}{d}\right)^{(k-2)\delta}. \label{eq:H1U''}
\end{align}

For $2\leq \ell\leq k-1$, recall that
a $(2, \ell)$-cycle is a 2-cycle with
two edges that shares exactly $\ell$ vertices, and $C_{\cH[U]}(2,\ell)$ is the number of $(2,\ell)$-cycles of $\cH[U]$, Note that, since $\cH[U]=\{H_{k-1}[U],H_k[U]\}$, a $(2,\ell)$-cycle may contain two edges of different sizes. To estimate $\mE(C_{\cH[U]}(2,\ell))$ for $2\leq \ell\le k-2$,  we apply \eqref{eq:DlHk} and \eqref{eq:DlHk2} obtaining that  
\begin{align*}
\left| \{ e, f\in H_k: |e\cap f|= \ell \}\right| & \le n^{\ell} (\D_{\ell}(H_k) )^2 \le n^{\ell} ( d n^{k-\ell-1} )^2 = n^{2k- \ell} / \left( \frac nd \right)^2, \\
\left| \{ e\in H_k, f\in H_{k-1}: |e\cap f|= \ell \}\right| & \le n^{\ell} \D_{\ell}(H_k) \D_{\ell}(H_{k-1}) \le 2 n^{2k-\ell-1} / \left(\frac{n}{d}
\right)^{\frac{2k-3}{k-1}}, \\
\left| \{ e, f\in H_{k-1}: |e\cap f|= \ell \}\right| & \le n^{\ell} (\D_{\ell}(H_{k-1}) )^2 \le 4 n^{2k-\ell-2} / \left(\frac{n}{d}
\right)^{\frac{2k-4}{k-1}}.
\end{align*}

Together with \eqref{eq:pn}, these give that
\begin{align*}
\mE(C_{\cH[U]}(2,\ell))&=\sum_{\substack{e,f\in H_k\\|e\cap f|=\ell}}p^{2k-\ell} + \sum_{\substack{e\in H_k, f\in H_{k-1}\\|e\cap f|=\ell}}p^{2k-l-1}+\sum_{\substack{e,f\in H_{k-1}\\|e\cap f|=\ell}}p^{2k-l-2} \\
&=\left(\frac{n}{d}\right)^{\frac{2-\ell}{k-1}+(2k-\ell)\delta} + 2\left(\frac{n}{d}\right)^{\frac{2-\ell}{k-1}+(2k-\ell - 1)\delta} + 4\left(\frac{n}{d}\right)^{\frac{2-\ell}{k-1}+(2k-\ell-2)\delta}\\
&=O\left(\left(\frac{n}{d}\right)^{\frac{2-\ell}{k-1}+(2k-\ell)\delta}\right) = O\left(\left(\frac{n}{d}\right)^{(2k-2)\delta}\right) 
\end{align*}
for $2\le \ell\le k-2$.

Now let $X_3$ be the random variable counting the number of linear 3-cycles in $\cH[U]$. 
We can bound the number of \emph{all} 3-cycles by first picking three vertices and choosing three edges from $H_k$ or $H_{k-1}$ containing three pairs of these three vertices. For example, there are at most $\binom n3 (\D_2(H_k))^3$ 3-cycles in $H_k$, while there are at most $\binom n3 (\D_2(H_{k}))^2\D_2(H_{k-1})$ 3-cycles with two edges in $H_k$ and one in $H_{k-1}$. 
Each linear 3-cycle in $\cH$ has $3k-3-i$ vertices, where $i$ is the number of the edges of $H_{k-1}$ in the cycle. For example, a linear 3-cycle in $H_k$ has $3k-3$ vertices while a linear 3-cycle with two edges in $H_k$ and one in $H_{k-1}$ has $3k-4$ vertices.
Therefore, applying \eqref{eq:DlHk2} and \eqref{eq:pn}, the expected number of linear 3-cycles in $H_k$ is at most
\[
\binom n3 (\D_2(H))^3 p^{3k-3} \le \left( n ( n^{k-3} d) p^{k-1}\right)^3 = \left(\frac{n}{d}\right)^{3(k-1)\delta}
\]
and the expected number of 3-cycles with two edges in $H_k$ and one in $H_{k-1}$ is at most
\[
\binom n3 (\D_2 (H_{k}) )^2\D_2(H_{k-1}) p^{3k-4} \le \frac13  (p\,n) ^{3k-4} /  \left(\frac{n}{d}\right)^{2+\frac{k-2}{k-1}} = \frac13 \left(\frac{n}{d}\right)^{(3k-4)\delta}.
\]
The expected numbers of other kinds of 3-cycles follow similar bounds. Putting these together, we conclude that $\mE(X_3)= O\left(\left(\frac{n}{d}\right)^{3(k-1)\delta} \right) $.

Let $X_4$ be the random variable counting the number of linear 4-cycles in $\cH[U]$ with clean intersections. 
Note that if $e_1, e_2, e_3, e_4\in \cH$ form a linear 4-cycle, it is possible to have $e_1\cap e_3\ne \emptyset$ or $e_2\cap e_4\ne \emptyset$. However, if $\{e_1, e_2, e_3, e_4\}$ is a linear 4-cycle with clean intersections, then 
$e_1\cap e_3= \emptyset = e_2\cap e_4$, which gives
the precise structure of this 4-cycle. In particular, the number of vertices in this 4-cycle is $4k-4-i$, where $i$ is the number of the edges of $H_{k-1}$ in this cycle. Following similar computations as the case of linear 3-cycles, we obtain that $\mE(X_4) = O\left( \left(\frac{n}{d}\right)^{4(k-1)\delta} \right)$.

We now consider two cases of Theorem~\ref{th:application} separately.

\vspace{0.2cm}

\noindent\textbf{Case (1):} \emph{there exists $\eps>0$ such that $\deg_H(S) \not\in \left( n^{ \frac{k-2}{k-1}-\eps} d^{ \frac{1}{k-1}+\eps}, n^{ \frac{k-2}{k-1} } d^{\frac{1}{k-1} }  \right)$ for every $S \in \binom{V(H)}{k-1}$.}

\vspace{0.3cm}

Recall that $H_k$ consists of all edges of $H$ that contain no $(k-1)$-set $S$ such that $\deg_H(S)\ge n^{ \frac{k-2}{k-1} } d^{\frac{1}{k-1} }$. Case (1) thus implies that all $(k-1)$-subsets $S$ of $V(H)$ satisfy either $S\in H_{k-1}$ or $\deg_H(S)\le  n^{\frac{k-2}{k-1}-\eps}d^{\frac{1}{k-1}+\eps}$. Consequently, 
\begin{align}\label{Dk-1Hk}
\D_{k-1}(H_k)\leq n^{\frac{k-2}{k-1}-\eps}d^{\frac{1}{k-1}+\eps}. 
\end{align}

Since $n^{ \frac{k-2}{k-1}-\eps} d^{ \frac{1}{k-1}+\eps}$ decreases as $\eps$ increases, we may assume, without loss of generality, that $\eps\leq 1/(4k)$. 
We apply the sampling described previously with $\delta=\frac{\epsilon}{k+1}\le \frac{1}{4k(k+1)}$.
We also assume that $n/d\ge n_0(\delta)$, where $n_0(\delta)$ is a large number depending on $\delta$.
Recall that $C_{\cH[U]}(2,k-1)$ is the number of $2$-cycles with intersection of size $k-1$ in $\cH[U]$.
In view of $|H_k| = \D_0(H_K) \le \frac{2}{k!}n^{k-1} d$ from \eqref{eq:DlHk2}, we apply \eqref{Dk-1Hk} and \eqref{eq:pn} obtaining that
\begin{align*}
    \mE(C_{\cH[U]}(2,k-1)) &=\sum_{\substack{e,f\in H_k, |e\cap f|=k-1}}p^{k+1}\le \frac{k}{2}|H_k|\D_{k-1}(H_k)p^{k+1}
    \le (pn)^{k+1} /  \left(\frac{n}{d}\right)^{\frac{k}{k-1} + \eps} \\
    &= \left(\frac{n}{d}\right)^{\frac{1}{k-1}+(k+1)\delta-\eps} =  \left(\frac{n}{d}\right)^{\frac{1}{k-1}}.
\end{align*}
Let $X = \sum_{\ell=2}^{k-1} C_{\cH[U]}(2,\ell) + X_3 + X_4$. Combining with the bounds on $\mE(X_3)$ and $\mE(X_4)$ obtained earlier, we derive that 
\begin{align}\label{eq:X}
    \mE(X) = \left(\frac{n}{d}\right)^{\frac{1}{k-1}}+ O\left(\left(\frac{n}{d}\right)^{3(k-1)\delta} \right) +
    O\left(\left(\frac{n}{d}\right)^{(4k-4)\delta}\right) 
    \ll \left(\frac{n}{d}\right)^{\frac 1{k-1}+\delta} = pn 
\end{align}
by using $(4k-4)\delta < \frac{1}{k-1}$.
Applying Markov's inequality, we have 
\begin{align}\label{eq:PX}
\mP( X \ge 4 \mE(X) )\le 1/4.
\end{align}

Recall that $U'= U\setminus (Z_{k-1}\cup Z_k)$.
Let $U''$ be the subset of $U'$ obtained by removing (at least) one vertex from each of 2-cycles, linear 3-cycles, and linear 4-cycles with clean intersections. We claim this actually remove all cycles of length $2, 3, 4$ in $\cH[U']$. Indeed, if $\C$ is a cycle of length $3$ or $4$, then either $\C$ is linear or $\C$ contains a 2-cycle. If $\C=\{e_1, e_2, e_3, e_3\}$ is a linear 4-cycle, then either $\C$ has clean intersections or say, $e_1\cap e_3\ne \emptyset$, which implies that $e_1, e_2, e_3$ form a linear 3-cycle. In all cases, one vertex of $\C$ will be removed. 
Thus $\cH[U'']$ has no cycles of length $2$, $3$ and $4$ and consequently has the $\bouquet$ properties.

By \eqref{eq:|U|2}, \eqref{eq:Zk-1}, \eqref{eq:X}, and \eqref{eq:PX}, we have
\begin{align}\label{eq:U''}
|U''|\geq |U|-\frac{pn}{5} - 4\,\mE(X) \ge \frac7{10} pn- 4\,\mE(X) \ge 
\frac{pn}2=\frac{1}{2}\left(\frac{n}{d}\right)^{\frac 1{k-1}+\delta}
\end{align}
with probability at least $1 - (\frac12 + \frac14 + \frac18) = \frac18>0$. 
Since $U''\subseteq U'$, we have 
\begin{align} \label{eq:D1U''}
\D_1(H_k[U''])\le 20 (n/d)^{(k-1)\d} \quad  \text{and} \quad \D_1(H_{k-1}[U''])\le 40 (n/d)^{(k-2)\d}
\end{align} 
from \eqref{eq:HU''} and \eqref{eq:H1U''}. We now pick a set $U''$ of size $\frac12\left(\frac{n}{d}\right)^{\frac{1}{k-1}+\delta}$ satisfying \eqref{eq:D1U''} and the $\bouquet$ properties.

We apply Theorem \ref{thm:main} with $T=3\left(\frac{n}{d}\right)^{\delta}$ and $N=|U''|= \frac12\left(\frac{n}{d}\right)^{\frac{1}{k-1}+\delta}$ to $\cH[U'']$. Note that $(\log N)^3\le T\le N^{1/(4k)}$ because $0< \delta < \frac{1}{4k(k-1)}$ and $\frac nd$ is sufficiently large.
By \eqref{eq:D1U''} and $k\ge 4$, we have $\D_1(H_{k}[U''])\le T^{k-1}$ and $\D_1(H_{k-1}[U''])\le T^{k-2} (\log T)^{\frac{1}{k-1}}$. 
Since $\cH[U'']$ contains no 2-cycle, we have $\D_{i-1}(H_i[U''])\le 1$ for $i=k-1, k$.
Theorem \ref{thm:main} thus gives that 
\begin{align*}
\a(\cH[U'']) &\ge c_k \frac{|U''|}{T} \left(\log T \right)^{\frac1{k-1}} = c_k \frac{\frac12 (n/d)^{\frac 1{k-1}+\delta} }{3(n/d)^{\delta}} \left(\log \Big[ 3\left(\frac nd\right)^{\delta}\Big] \right)^{\frac1{k-1}} \\
&\ge \frac{c_k}{6} \left(\delta \frac nd \log \left(\frac nd\right) \right)^{\frac1{k-1}},
\end{align*}
where $c_k>0$ is a constant. Since $\a(H)\ge \a(\cH[U''])$, this gives the desired bound for $\a(H)$.

\vspace{0.2cm}

\noindent \textbf{Case (2):} \emph{$\deg_H(S)\not\in \left(n^{\frac{k-2}{k-1}}d^{\frac{1}{k-1}}/\left[\log(\frac{n}{d})\right]^{k+1}, n^{\frac{k-2}{k-1}}d^{\frac{1}{k-1}} \right)$ for every $S \in \binom{V(H)}{k-1}$, and $H$ contains no 4-cycle with clean intersections.}

\vspace{0.3cm}

We first observe that $\cH$ contains no 4-cycle with clean intersections. Suppose that edges $e_1, e_2, e_3, e_4$ form a 4-cycle in $\cH$ with $e_1\cap e_3=\emptyset$ and $e_2\cap e_4=\emptyset$. Then we can form a $4$-cycle $f_1,f_2,f_3,f_4 \in H$ as follows: If $e_i \in H_k$, then let $f_i=e_i$. Otherwise $e_i\in H_{k-1}$, we take $f_i\in H$ such that $f_i=e_i\cup\{y_i\}$ for some $y_i\notin \bigcup_{j\neq i}^4f_j$ (this is possible because $\deg_H(e_i)\geq n^{\frac{k-2}{k-1}}d^{\frac{1}{k-1}}\ge 4k$). Hence, $f_1,f_2,f_3$ and $f_4$ form a $4$-cycle with $f_1\cap f_3=e_1\cap e_3=\emptyset$, $f_2\cap f_4=e_2\cap e_4=\emptyset$, which contradicts the assumption that $H$ has no $4$-cycle with clean intersections.

Similar as in Case $(1)$, the definition of $H_k$ forces that 
\begin{align}\label{eq:Dk-1Hk}
\D_{k-1}(H_k)\leq n^{\frac{k-2}{k-1}}d^{\frac{1}{k-1}}/\left(\log \frac{n}{d}\right)^{k+1}. 
\end{align}
However, the strategy here will be slightly different from the previous case. The bound given by (\ref{eq:Dk-1Hk}) will not be enough to remove $(2, k-1)$-cycles of $H_k$. However, since $\cH$ has no $4$-cycles with clean intersections, we will be able to find a subhypergraph satisfying the $\bouquet$ properties and then apply Theorem \ref{thm:main}.

Apply the sampling described previously with $0<\delta<\frac{1}{4 k^2}$. Let $X = \sum_{\ell=2}^{k-2} C_{\cH[U]}(2,\ell) + X_3$ (we know $X_4=0$ because $H$ contains no $4$-cycle with clean intersections). Earlier calculations give that $\mE(X) = O\left(\left(\frac{n}{d}\right)^{3k-3)\delta}\right)=o\left(\left(\frac{n}{d}\right)^{\frac{1}{k-1}+\delta}\right)$. 

Let $U''$ be the subset of $U'=U\setminus(Z_{k-1}\cup Z_k)$ obtained by removing (at least) one vertex from each of the linear $3$-cycles and $(2,\ell)$-cycles in $\cH[U']$ with $2\leq \ell \leq k-2$. We claim that $\cH[U'']$ has the $\bouquet$ property. Indeed, 
Properties i) and ii) follow from the fact that only 2-cycles in $\cH$ are $(2,k-1)$-cycles from $H_k$. 
Properties iii) and iv) hold because $\cH[U'']$ contains no linear 3-cycle and no 4-cycle with clean intersections. To see v), suppose that there exist three edges $e_1, e_2, e_3\in H_3$ such that $|e_1 \cap e_2|=|e_2\cap e_3|=2$ and $|e_1\cap e_3|=1$. Since the only 2-cycles in $\cH$ are those with intersection of size $k-1$, we must have $k=3$, which contradicts the assumption of Theorem~\ref{th:application}.

By a similar argument as Case (1), one can prove that $\cH[U'']$ satisfies \eqref{eq:U''} and \eqref{eq:D1U''}. It remains to bound $\D_{k-1}(H_k)$. 
We first extend the definition of  
\[ \deg_{\cH[U]}(S):=|\{x\in U:\: S\cup\{x\} \in \cH\}|. \]
by allowing $S$ to be any $(k-1)$-set $S\subseteq V(H)$.
For any $(k-1)$-set $S$, we have
\begin{align*}
    \mE(\deg_{\cH[U]}(S))=p\deg_{H_k}(S)\leq p\D_{k-1}(H_k)\leq \left(\frac{n}{d}\right)^\delta/\left(\log \frac{n}{d}\right)^{k+1} \quad \text{by } \eqref{eq:Dk-1Hk}.
\end{align*}
Therefore, by Chernoff's bound, we have
\begin{align*}
    \mP\left(|\deg_{\cH[U]}(S)-\mE(\deg_{\cH[U]}(S))|\geq \left(\frac{n}{d}\right)^\delta/\left(\log \frac{n}{d}\right)^{k+1}\right) &\leq 2\exp\left(-\frac{\left(\frac{n}{d}\right)^{2\delta}/\left(\log \frac{n}{d}\right)^{2k+2}}{3\mE(\deg_{\cH[U]}(S))}\right)\\
    &\leq 2\exp\left(-\frac{\left(\frac{n}{d}\right)^\delta}{3\left(\log \frac{n}{d}\right)^{k+1}}\right).
\end{align*}
Hence, by a union bound, the probability that there exists a $(k-1)$-set $S$ such that $S\subseteq U$ and
$
    |\deg_{\cH[U]}(S)-\mE(\deg_{\cH[U]}(S))|\geq
    \left(\frac{n}{d}\right)^{\delta} / \left(\log\frac{n}{d}\right)^{k+1} 
$
is at most
\begin{align*}
    2p^{k-1} {n}^{k-1}\exp\left(-\frac{\left(\frac{n}{d}\right)^\delta}{3\left(\log \frac{n}{d}\right)^{k+1}}\right) =  2\left(\frac{n}{d}\right)^{1+(k-1)\delta}\exp\left(-\frac{\left(\frac{n}{d}\right)^\delta}{3\left(\log \frac{n}{d}\right)^{k+1}}\right)=o(1)
\end{align*}
as $n/d$ is sufficiently large. Then, with probability $1- o(1)$, every $(k-1)$-tuple $S\subseteq U$ satisfies that
\begin{align*}
    \deg_{\cH[U]}(S)\leq
    2 \left(\frac{n}{d}\right)^{\delta} / \left(\log\frac{n}{d}\right)^{k+1}.
\end{align*}
Thus, with positive probability, we can find a set $U''$ such that $\cH[U''] \in \bouquet$, satisfying \eqref{eq:U''}, \eqref{eq:D1U''}, 
\begin{align*}
\D_{k-2}(H_{k-1}[U''])\le 1 \quad \text{and} \quad 
    \D_{k-1}(H_k[U''])\leq  \frac{2\left(\frac{n}{d}\right)^{\delta}}{(\log(n/d))^{k+1}} \le \frac{\left(\frac{n}{d}\right)^{\delta}}
    {\left(\log(n/d)^{\delta} \right)^{k+1}}.
\end{align*}
After removing vertices if necessary, we may assume that $|U''|= \frac12\left(\frac{n}{d}\right)^{\frac{1}{k-1}+\delta}$.
We now apply Theorem \ref{thm:main} with $T=3(n/d)^{\delta}$ and $N=|U''|$ to obtain the desired bound on $\a(\cH)$ exactly as in the first case.
\end{proof}

\section{Appendix B: Proof of Theorem \ref{cor2}}

\begin{proof}[Proof of Theorem \ref{cor2}]
Let $U\subseteq V(H)$ be a random subset whose vertices are chosen independently, each one with probability $p=t^{\delta-1}$ for $\delta=\frac{\eps}{4k}$. Then 
\begin{align}\label{eq:171}
    \mE(|U|)=pn=nt^{\delta-1}.
\end{align}
Since $|U|$ follows a binomial distribution, by Chernoff's bound, we have
\begin{align}\label{eq:172}
    \mP\left(||U|-pn|>\frac{pn}{10}\right)<e^{-cpn}<\frac{1}{8}
\end{align}
for some constant $c>0$ and all sufficiently large $n$.

As in the proof of Theorem \ref{th:application}, our strategy is to apply Theorem \ref{thm:main} to a subgraph $H[U']$ for some $U'\subseteq U$. We obtain $U'$ by removing vertices of high degree from $U$ and at least one vertex from every $(2,\ell)$-cycle with $2\leq \ell \leq k-2$ and from every linear $3$-cycle in $H[U]$.

We start by counting the expected number of $(2,\ell)$-cycles in $H[U]$ for $2\leq \ell \leq k-2$. Note that we can upper bound $C_{H[U]}(2, \ell)$, the number of $(2,\ell)$-cycles as follows: There are $|H|$ ways to choose the first edge $e$, $\binom{k}{\ell}$ ways to fix an $\ell$-tuple in $e$ and at most $\D_{\ell}(H)$ ways to extend this $\ell$-tuple to a second edge. Thus, by using that $\D_\ell(H)\leq t^{k-\ell-\eps}$, $|H|\leq nt^{k-1}$ and $\delta=\eps/4k$, we have 
\begin{align}\label{eq:173}
    \mE(C_{H[U]}(2,\ell))\leq |H|\binom{k}{\ell}\D_{\ell}(H)p^{2k-\ell}=O(nt^{\delta(2k-\ell)-\eps-1})=o(nt^{\delta-1}).
\end{align}

Let $X_3$ be the number of linear $3$-cycles in $H[U]$. One can estimate the number of linear $3$-cycles in $H$ in a similar fashion: There are $|H|$ choices for the first edge $e_1$, $k$ ways to choose a vertex $v_1\in e_1$, at most $\D_1(H)$ ways to extend $v_1$ to a second edge $e_2$, at most $k^2$ ways to fix a vertex $v_3\in e_1\setminus \{x_1\}$ and a vertex $v_2\in e_2\setminus \{x_1\}$, and finally at most $\D_2(H)$ ways to extend $v_3, v_2$ to an edge $e_3$. Thus,
\begin{align}\label{eq:174}
    \mE(X_3)\leq |H| k^3\D_1(H)\D_2(H)p^{3k-3}
    =O(nt^{\delta(3k-3)-\eps-1})=o(nt^{\delta-1}),
\end{align}
because $\delta(3k-3)-\eps<\delta$. Therefore, by (\ref{eq:173}) and (\ref{eq:174}), we have
\begin{align*}
    \mE\left(\sum_{\ell=2}^{k-2}C_{H[U]}(2,\ell)+X_3\right)\leq \frac{nt^{\delta-1}}{20}.
\end{align*}
Hence, by Markov's inequality,
\begin{align}\label{eq:175}
    \mP\left(\sum_{\ell=2}^{k-2}C_{H[U]}(2,\ell)+X_3>\frac{nt^{\delta-1}}{5}\right)\leq \frac{1}{4}.
\end{align}

Let $Z$ be the random variable counting the number of vertices $x\in U$ such that
\begin{align*}
    \deg_{H[U]}(x)>10\mE\left(\deg_{H[U]}(x)\right)=10p^{k-1}\deg_H(x).
\end{align*}
By Markov's inequality,
\begin{align*}
    \mP\left(\deg_{H[U]}(x)>10p^{k-1}\deg_H(x)\right)\leq \frac{1}{10}
\end{align*}
for every $x\in V(H)$. Since the event $x\in U$ is independent of the events $\{y\in U\}$ for $y\neq x$, we obtain
\begin{align*}
    \mP(x\in Z)=\mP(x\in U)\mP\left(\deg_{H[U]}(x)>10p^{k-1}\deg_H(x)\right)\leq \frac{p}{10}.
\end{align*}
Hence, $\mE(|Z|)\le pn/10$.
By another application of Markov's inequality,
\begin{align}\label{eq:176}
    \mP\left(|Z|\geq \frac{pn}{5}\right)\leq \frac{1}{2}.
\end{align}
By (\ref{eq:172}), (\ref{eq:175}) and (\ref{eq:176}), we have
\begin{align*}
    |U|-|Z|-\sum_{\ell=2}^{k-2}C_{H[U]}(2,\ell)-X_3 \ge \frac{9pn}{10}-\frac{pn}{5}-\frac{nt^{\delta-1}}{5}=\frac{1}{2}nt^{\delta-1}
\end{align*}
with probability at least $1/8$. Thus, with probability at least $1/8$, there exists a set $U'\subseteq U$ of size $|U|=\frac{1}{2}nt^{\delta-1}$ such that $U'$  contains no $(2,\ell)$-cycle or linear $3$-cycle,  and if $x\in U'$, then
\begin{align}\label{eq:177}
    \deg_{H[U']}(x)<\deg_{H[U]}(x)\leq 10p^{k-1}\deg_{H}(x)\leq 10p^{k-1}\D_1(H)\leq 10t^{(k-1)\delta}.
\end{align}

We observe that $H[U'] \in \bouquet$. Indeed, since $H[U']$ is $k$-uniform, Conditions i) and v) of $\bouquet$ trivially hold. Conditions ii) and iii) follows from the fact that $H[U']$ has no linear $3$-cycles and $(2,\ell)$-cycles for $2\leq \ell \leq k-2$. Finally, Condition iv) comes from the fact that $H$ has no $4$-cycle with clean intersections and $H[U']$ is a subgraph of $H$. Hence, $H[U'] \in \bouquet$.

It remains to bound $\D_{k-1}(H[U'])$. For any fixed $(k-1)$-set $S\subseteq V(H)$, the variable $\deg_{H[U]}(S)$ satisfies a binomial distribution $B(|\deg_H(S)|,p)$ with mean
\begin{align*}
    \mE\left(\deg_{H[U]}(S)\right)=p\deg_H(S).
\end{align*}
Thus, by Chernoff's inequality,
\begin{align*}
    &\mP\left(\left|\deg_{H[U]}(S) -p\deg_H(S)\right| <\frac{pt}{(\log t)^{k+1}} \right) \leq 2 \exp\left(\frac{-\left(\frac{pt}{(\log t)^{k+1}}\right)^2}{3p\deg_H(S)}\right)  \\
    & \leq 2 \exp\left(\frac{-\left(\frac{pt}{(\log t)^{k+1}}\right)^2}{3p\frac{t}{(\log t)^{k+1}} }\right) = 2\exp\left(-\frac{t^{\delta}}{3(\log t)^{k+1}}\right)\le\frac{2}{n^{k}},
\end{align*}
because $\frac{t^{\delta}}{3(\log t)^{k+1}}\ge k t^{\delta/3}
= k t^{\eps/(12k)}\geq k \log n$. Hence, by an union bound, with probability at least $1 - \binom{n}{k-1}\frac2{n^k}= 1 - o(1)$,
\begin{align}
\label{eq:178}
    \deg_{H[U']}(S)\leq \deg_{H[U]}(S) \leq \mE\left(\deg_{H[U]}(S)\right)+\frac{pt}{(\log t)^{k+1}}\leq \frac{2pt}{(\log t)^{k+1}}
\end{align}
for \emph{all} $(k-1)$-sets $S\subseteq V(H)$. 
Therefore, with probability at least $1/16$, there exists a subset $U'\subseteq U$ of size $|U'|=\frac{1}{2}nt^{\delta-1}$ such that $H[U']\in \bouquet$, (\ref{eq:177}) and (\ref{eq:178}) hold, that is,
\begin{align*}
    \D_1(H[U'])\leq 10t^{(k-1)\delta} \quad \text{and} \quad \D_{k-1}(H[U'])\leq \frac{2t^{\delta}}{(\log t)^{k+1}}.
\end{align*}

Set $T=10^{\frac1{k-1}}t^{\delta}$ and $N=|U'|=\frac{n}{2}t^{\delta-1}$. 
Then $T \ge t^{\frac{\eps}{4k}}\geq (\log n)^{3}\ge (\log N)^3$ and 
$T^{4k}= c t^{4k\delta} = c t^{\eps} \le c n/t \le N$, where
$c =10^{\frac{4k}{k-1}}$.
We have $\D_1(H[U'])\leq 10t^{(k-1)\delta}=T^{k-1}$ and 
$\D_{k-1}(H[U'])\leq \frac{2t^{\delta}}{(\log t)^{k+1}}
\le T/ (\log T)^{k+1}$ because $4\delta^{k+1}\le 4 (\frac1{4k})^{k+1}< 10^{\frac1{k-1}}$.
Then $H[U']$ satisfies the hypothesis of Theorem \ref{thm:main}. Therefore,
\begin{align*}
    \a(H)\geq \a(H[U']) = \Omega\left(\frac{N}{T}(\log T)^{\frac{1}{k-1}}\right)
    =\Omega\left(\frac{n}{t}(\log t)^{\frac{1}{k-1}}\right). &\qedhere
\end{align*}
\end{proof}

\end{document}